\documentclass[11pt,fullpage]{article}
\usepackage{amsmath}
\usepackage{amsfonts,amssymb}
\usepackage{color}	
\usepackage{amsthm}
\usepackage{graphics,graphicx}
\usepackage{hyperref}
\usepackage[right = 2.5cm, left=2.5cm, top = 2.5cm, bottom =2.5cm]{geometry}
\numberwithin{equation}{section}

\usepackage{marginnote}

\newcommand{\Real}{\mathbb R}
\newcommand{\T}{\mathbb T}
\newcommand{\Integer}{\mathbb Z}
\newcommand{\I}{\mathbf I}
\newcommand{\norm}[1]{\left\lVert#1\right\rVert}
\newcommand{\abs}[1]{\left\vert#1\right\vert}
\newcommand{\set}[1]{\left\{#1\right\}}

\newcommand{\grad}{\nabla}

\newcommand{\G}{\mathcal{G}}

\newcommand{\Naturals}{\mathbb N}
\newcommand{\Torus}{\mathbb T}

\newcommand{\jap}[1]{\left\langle #1 \right\rangle} 

\newcommand{\eps}{\epsilon}

\newcommand{\avint}[1]{\mathchoice
{\mathop{\vrule width 6pt height 3 pt depth -2.5pt
\kern -8.8pt \intop}\nolimits_{#1}}%
{\mathop{\vrule width 5pt height 3 pt depth -2.6pt
\kern -6.5pt \intop}\nolimits_{#1}}%
{\mathop{\vrule width 5pt height 3 pt depth -2.6pt
\kern -6pt \intop}\nolimits_{#1}}%
{\mathop{\vrule width 5pt height 3 pt depth -2.6pt \kern -6pt
\intop}\nolimits_{#1}}}

\newtheorem{theorem}{Theorem}
\theoremstyle{definition}
\newtheorem{remark}{Remark}
\newtheorem{rmk}{Remark}
\theoremstyle{lemma}

\newtheorem{proposition}{Proposition}[section]
\newtheorem{corollary}{Corollary}
\theoremstyle{definition}

\theoremstyle{lemma}
\newtheorem{lemma}{Lemma}[section]

\begin{document}
\title{Inviscid damping and the asymptotic stability of planar shear flows in the 2D Euler equations}

\author{Jacob Bedrossian\footnote{\textit{jacob@cims.nyu.edu}, Courant Institute of Mathematical Sciences. Partially supported by NSF Postdoctoral Fellowship in Mathematical Sciences, DMS-1103765} \, and Nader Masmoudi\footnote{\textit{masmoudi@cims.nyu.edu}, Courant Institute of Mathematical Sciences. Partially supported by NSF  grant  
DMS-1211806 }}

\date{\today}
\maketitle

\section*{Abstract} 

We prove  asymptotic stability of shear flows close to the
 planar Couette flow in the 2D inviscid Euler equations on $\Torus \times \Real$.  
That is, given an initial perturbation of the  Couette flow small in a suitable regularity class, 
 specifically Gevrey space of class smaller than 2, 
 the velocity converges strongly in $L^2$ to a shear flow which is also close to the Couette flow. 
The vorticity is asymptotically driven to small scales by a linear evolution and weakly converges as $t \rightarrow \pm\infty$.
The strong convergence of the velocity field is sometimes referred to as \emph{inviscid damping}, due to the relationship with Landau damping in the Vlasov equations.  
This convergence was formally derived at the linear level  by Kelvin in 1887 
and it  occurs at  
 an   algebraic rate  first     computed  
 by Orr in 1907; our work appears to be the first rigorous 
confirmation of this behavior on the nonlinear level.

\bigskip

\setcounter{tocdepth}{1}
{\small\tableofcontents}

\newcommand{\nm}[1]{{\marginnote{n:#1}}}
\newcommand{\jb}[1]{{\marginnote{j:#1}}}
 
\section{Introduction} \label{sec:Intro}

We consider the 2D Euler system in the vorticity formulation with a background shear flow: 
\begin{equation} \label{def:2DEuler}
\left\{
\begin{array}{l}
  \omega_t + y\partial_x\omega + U \cdot \grad \omega = 0, \\ 
  U  = \grad^{\perp}(\Delta)^{-1} \omega,  \quad \quad \omega(t=0) =\omega_{in}. 
\end{array}
\right. 
\end{equation}
Here,  $(x,y) \in \mathbb T \times \Real$, $\grad^\perp = (-\partial_y,\partial_x)$ and $(U,\omega)$ are periodic in the $x$ variable with period normalized to $2\pi$.   
The physical velocity is $(y,0) + U$ where $U = (U^x,U^y)$ denotes the velocity perturbation and the total vorticity is $-1 + \omega$. 
We denote the streamfunction by $\psi = \Delta^{-1}\omega$. 
The velocity itself satisfies the momentum equation
\begin{equation} \label{def:2DEulerMomentum}
\left\{
\begin{array}{l}
  U_t + y\partial_x U + (U^y,0) + U \cdot \grad U = -\grad P, \\ 
  \grad \cdot U = 0, 
\end{array}
\right. 
\end{equation}
where $P$ denotes the pressure. 
Linearizing the vorticity equation \eqref{def:2DEuler} yields the linear evolution 
\begin{equation} \label{def:LinearizedEuler}
\left\{
\begin{array}{l}
  \omega_t + y\partial_x\omega = 0, \\ 
  U  = \grad^{\perp}(\Delta)^{-1} \omega,  \quad \quad \omega(t=0) =\omega_{in}. 
\end{array}
\right. 
\end{equation}
In this work, we are interested in the long time behavior of \eqref{def:2DEuler} for small initial perturbations 
$ \omega_{in} $. 
In particular, we show that all sufficiently small perturbations in a suitable regularity class undergo `inviscid damping' and satisfy $(y,0) + U(t,x,y) \rightarrow  (y + u_\infty(y),0)$ as $t \rightarrow \infty$ for some $u_\infty(y)$ determined by the evolution.

The field of hydrodynamic stability
started in the  nineteenth century
with Stokes, Helmholtz, Reynolds, Rayleigh, Kelvin, Orr, Sommerfeld and many others. 
Rayleigh \cite{Rayleigh80} studied the linear stability and instability of planar inviscid shear flows
using what is now referred to as the \emph{normal mode method}.
Such a method yields {\em spectral instability} or {\em spectral  stability} depending on whether or not 
  an unstable  eigenvalue exists.
In that  work, Rayleigh proves the famous inflection point theorem which gives a necessary condition for spectral instability. 
At around the same time, Kelvin \cite{Kelvin87} constructed exact solutions to the
  linearized problem  around the Couette flow (which are actually solutions of the nonlinear problem).
This was the first attempt to solve the initial value problem of the linearized problem 
which was later developed further in \cite{Orr07,Case60,MP77}. 

Even to the present day, the  methods used and the  conclusions
   of these works are debated both on   physical and  mathematical grounds 
\cite{Yaglom12}. Experimental realizations of Couette and similar spectrally stable flows show instability and transition to turbulence for sufficiently high Reynolds numbers \cite{Reynolds83,Orszag80,lundbladh91,Tillmark92,bottin98}.
However, experiments are ultimately inconclusive (mathematically) since many factors are notoriously difficult to control, such as imperfections in the walls and viscous boundary layers. 
The paradox that Couette flow is known to be spectrally stable for all Reynolds numbers in contradiction with instabilities observed in experiments is now often referred to as the `Sommerfeld paradox', or `turbulence paradox'.
Of course, from a mathematical point of view, the notion 
of linear stability was not completely precise in the early works:  was it enough that the
linear operator has no growing mode ({\em spectral stability}) or should one consider
general initial perturbation and study the time evolution under the linear equation using,
for instance, a Laplace transform in time (see  \cite{Case60,MP77,Briggs70}).
These early works also pre-dated the notion of Lyapunov stability and Sobolev spaces; indeed, the stability of \eqref{def:LinearizedEuler} depends heavily on the norm chosen. 
A more variational approach, which is based on the conserved quantities and uses the notion of 
Energy-Casimir, was introduced by  Arnold  \cite{AK98}, and yields Lyapunov stability for a class of shear flows (which does not include Couette flow).  
 We also refer to \cite{GR01,LMR12} for the use of the variational approach in the 
Vlasov case.
Recently there were many mathematical studies  of stability and 
instability of various flows  (see for instance   \cite{FSV97,BGS02,Grenier00,Lin04}). 
We also refer to the following textbooks on the topic of hydrodynamic 
  stability and  instability   \cite{Lin55,DR81,Yaglom12}. 

There were many attempts in the literature to  find an explanation to
the  Sommerfeld paradox  (see \cite{LiLin11} and the references therein). 
The  first attempt   might be  due to  Orr  \cite{Orr07} in 1907, whose work plays a central role in ours. 
Orr's observation can be summarized in modern terminology (and adapted to our infinite-in-$y$ setting) as follows.  
Given a disturbance in the vorticity, the linear evolution under \eqref{def:LinearizedEuler} is simply advected by the background shear flow: $\omega(t,x,y) = \omega_{in}(x-ty,y)$. 
If one changes coordinates to $z = x-ty$ then the stream-function $\phi(t,z,y)$ in these variables solves $\partial_{zz}\phi +
  (\partial_y - t\partial_z)^2\phi = \omega_{in}$. 
On the Fourier side, $(z,y) \to (k,\eta) \in \Integer \times \Real$,
\begin{align} \label{orr-cri} 
\hat{\phi}(t,k,\eta) = -\frac{\hat{\omega}_{in}(k,\eta)}{k^2 + \abs{\eta - kt}^2}. 
\end{align}  
From \eqref{orr-cri}, Orr made two important observations, together known now as the \emph{Orr mechanism}.  
Firstly, if  $\eta, \,  k > 0 $  and $\eta$ is very large relative to $k$, 
 then the stream-function amplifies by a  factor $\frac{\eta^2}{k^2}$  at a {\em critical time}  given by 
$t_c = \frac{\eta}k$. These modes correspond to waves tilted against the shear which are being advected to larger length-scales (lower frequencies). 
Orr suggested that this  transient growth is 
 a possible explanation for the observed \emph{practical instability} 
or at least as a reason to question the validity of the linear approximation. 
Moreover, this shows that the  Couette flow is linearly unstable 
(in the sense of Lyapunov) in the kinetic energy norm.
On the other hand, Orr states in \cite[ART. 12]{Orr07} that ``the motion 
is stable, for the most general disturbance, if sufficiently small''.  
Orr does not precise the meaning of {\em sufficiently small} but concludes in this case
  that ``the $y$ velocity-component eventually diminishes indefinitely as 
$t^{-2}$,  and, 
the $x$ component of the relative velocity as $t^{-1}$''.
In fact, on the linear level, it is not about smallness but about regularity. Indeed, rigorous proof of the stability and decay on the linear level requires the use of
a stronger norm on the initial data than on the evolution, as already noticed in
Case \cite{Case60} and Marcus and Press \cite{MP77} where this linear stability and decay 
are proved.

Physically, the decay predicted by the Orr mechanism can be understood as the transfer of enstrophy to small scales (which yields the decay of the velocity
 by the Biot-Savart law) 
and the transient growth can be understood as the time-reversed phenomenon: the transfer of enstrophy from  small scales to large scales 
 and hence the growth of the velocity
 (see also \cite{Boyd83,Lindzen88} for further discussion). 
The transfer  to small scales by mixing is now considered a fundamental mechanism intimately connected with the stability of coherent structures and the theory of 2D turbulence \cite{Kraichnan67,Gilbert88}.
However, to our knowledge, our work is the first mathematically rigorous study of this mechanism in the full 2D Euler equations.  We refer to 
\cite{Yaglom12,Shnirelman12,KS12,GSV13} 
  for the most recent developments.  
Mathematically, one can also explain the transient growth
by the non-normality of the linearized operator (also an insight first due to Orr). 
See for example \cite{TTRD93}, where the implications of this are studied in terms of the spectra and pseudospectra of 
the linearized Couette and Poiseuille flows. 
Indeed, the fact that for non-normal operators the $\eps-$pseudospectrum can be very different from the spectrum 
can be seen as another explanation of the  transient  linear growth \cite{RSH93}. 
See also \cite{SH01} for further information.    

In 1946, Landau \cite{Landau46} predicted rapid decay of the electric field in hot plasmas perturbed from homogeneous equilibrium by solving the linearized Vlasov equation with a Laplace transform. 
Now referred to as \emph{Landau damping}, this somewhat controversial prediction of collisionless relaxation in a time-reversible physical model was confirmed by experiments much later in \cite{MalmbergWharton64} and is now a well-accepted, ubiquitous phenomenon in plasma physics \cite{Ryutov99}.  
In \cite{VKampen55}, van Kampen showed that one way to interpret this mechanism was through the transfer of information
to small scales in velocity space; a scenario completely consistent with time-reversibility and conservation of entropy. 
In this scenario, the  free-streaming of 
 particles creates   rapid oscillations of the distribution function which are averaged away by the non-local Coulomb interactions 
(see also \cite{Case59,Degond86}). 
The fundamental stabilizing mechanism in this picture is the \emph{phase-mixing} due to  particle streaming.
The gap between the linear and nonlinear theory of Landau damping was only bridged recently by the ground-breaking work of Mouhot and Villani, who showed that the phase-mixing indeed persists in the nonlinear Vlasov equations for small perturbations \cite{MouhotVillani11} (see also \cite{CagliotiMaffei98,HwangVelazquez09}). 

The algebraic decay of the velocity field for solutions to \eqref{def:LinearizedEuler} predicted by Orr can be most readily understood as a consequence of \emph{vorticity mixing} driven by the shear flow, and hence can be considered as 
 a hydrodynamic analogue of Landau damping, a viewpoint furthered by many authors \cite{BouchetMorita10,SchecterEtAl00,Briggs70,BM95}.
Hence the origin of the term \emph{inviscid damping}. 
The first, and most fundamental, difference between \eqref{def:LinearizedEuler} and the linearized Vlasov equations is the fact that the velocity field induced by mean-zero solutions to \eqref{def:LinearizedEuler} in general does not converge back to the Couette flow, but in fact converges to a different nearby shear flow, whereas the electric field in the linearized Vlasov equations converges to zero. 
This `quasi-linearity' will be a major difficulty in studying inviscid damping on the nonlinear level. 
Another key difference is that unlike in the Vlasov equations, the decay of the velocity field in \eqref{def:LinearizedEuler} cannot generally be better than the algebraic rate predicted by Orr, which is not even integrable for  the $x$ component of the 
velocity; to contrast, 
in the Vlasov equations the decay is exponential for analytic perturbations.

It is well-known that the nonlinearity can change the picture dramatically.
A clear example of this are the results of Lin and Zeng
\cite{LinZeng11}  who prove  that there exists
non-trivial periodic solutions to the vorticity equation \eqref{def:2DEuler} which are arbitrarily close to the Couette flow in $H^s$ for $s < 3/2$. They have also proved the corresponding, and related, result for the Vlasov equations  \cite{LZ11b}. 
In our setting, the primary interest is to rule out the possibility that weakly nonlinear effects create a self-sustaining process and  push 
the solution out of  the linear regime. 
The idea that the interaction between nonlinear effects and non-normal transient growth can lead to instabilities is classical in fluid mechanics (see e.g. \cite{TTRD93}).
The basic mechanism suggested in \cite{TTRD93} is that nonlinear effects can repeatedly excite growing modes 
and precipitate a sustained cascade or so-called `nonlinear bootstrap', studied further in the fluid mechanics context in, for example, \cite{BaggettEtAl,VMW98,Vanneste02}. 
Actually, this effect is very similar to what is at work behind \emph{plasma echos} in the Vlasov equations, first captured experimentally in \cite{MalmbergWharton68}. 
This phenomenon is referred to as an `echo' because the measurable result of nonlinear effects can occur long after the event.   
Very similar echos have been studied and observed in 2D Euler, both numerically \cite{VMW98,Vanneste02} and experimentally \cite{YuDriscoll02,YuDriscollONeil} (interestingly, non-neutral plasmas in certain settings make excellent realizations of 2D Euler). 

The plasma echos play a pivotal role in the work of Mouhot and Villani on Landau damping \cite{MouhotVillani11}. 
Although our approach to this challenge is quite different, 
one of the main difficulties we face is to precisely understand the weakly nonlinear effects at work; sometimes called {\em nonlinear transient growth} \cite{VMW98}. 
We will need a more precise alternative to the moment estimates of  
\cite{MouhotVillani11} which is tailored to the specific structure of 2D Euler; 
what we call the ``toy model'' 
(see \S\ref{sec:Conclusion} for a detailed discussion about the relationship of our work to \cite{MouhotVillani11}). 
The toy model, formally derived in \S\ref{sec:Toy}, provides mode-by-mode upper bounds on the `worst possible' growth of high frequencies that the weakly nonlinear effects can produce. 
The model is not just a heuristic and in fact plays a key role in our work: it is used in the construction of a norm specially designed to match the evolution of \eqref{def:2DEuler}; this norm is the subject of \S\ref{sec:Growth}. 
We remark that our model has not appeared in the literature before to our knowledge, however related models have been studied in \cite{VMW98,Vanneste02}. 

 The mixing phenomenon behind the inviscid damping also appears in many other fluid models, 
for example, more general shear profiles \cite{BMSEI95,BouchetMorita10}, stratified shear flows \cite{Majda03,CV13}
and 2D Euler with the $\beta$-plane approximation to the Coriolis force \cite{Boyd83,Tung83}. 
A particularly fundamental setting is the `axisymmetrization' of vortices in 2D Euler 
which has important implications for the meta-stability of coherent vortex structures in atmosphere and ocean dynamics 
(see e.g. \cite{Gilbert88,GilbertBassom98,SchecterEtAl00,YuDriscoll02,YuDriscollONeil} for a small piece of the extensive literature).  
Actually, this stability problem was mentioned by 
 Rayleigh \cite{Rayleigh80} and was considered by Orr as well \cite{Orr07}.  
Interestingly, it is also relevant to the stability of charged particle beams in cyclotrons \cite{CerfonEtAl13}. 

In general, phase-mixing, or `continuum damping', can be directly associated with the continuous spectrum of the linearized operator
and is a phenomenon shared by a number of infinite-dimensional Hamiltonian systems, for example the damping of MHD waves \cite{TataronisGrossman73}, the Caldeira-Legget model from quantum mechanics \cite{HagstromMorrison} and synchronization models in biology \cite{StrogatzEtAl92}.
See the series of works \cite{BM95,BMSEI95,Morrison98,Morrison00,BM02} which draws a connection between the van Kampen generalized eigenfunctions and the normal form transform
to write the linearized 2D Euler and Vlasov-Poisson equations as a continuum of decoupled harmonic oscillators.  
See also \cite{BMT13} and the references therein for a recent survey which contains other examples and discusses some connections between these various models. 
 
Phase mixing also shares 
 certain similarities with scattering in the theory of dispersive wave equations
(see for instance \cite{lax90,GMS12,LR05}) as already pointed out in \cite{Degond86,CagliotiMaffei98}.
In both cases the long time behavior is governed by a linear operator, or a modified version of it due to long range interactions \cite{GV00,Nakanishi02} (something like this occurs in our Theorem \ref{thm:Main}). 
Unlike dissipative equations, the final linear evolution is usually chosen by the entire nonlinear dynamics and cannot be completely characterized by the relevant conservation laws. 
Also in both cases, the phenomena can be related to the continuous spectrum in the linear problem; for example, the RAGE theorem applies equally well to transport equations as to dispersive equations \cite{ConstantinEtAl08}. 
However, there are also clear differences since in dispersive wave equations, 
the dispersion uses the fact that different wave packets travel with different group velocities to yield decay of the $L^\infty$ norm  and hence nonlinear terms often become weaker.   
Normally, this decay costs \emph{spatial localization} rather than \emph{regularity}.
In the inviscid damping (and Landau damping), the decay is due to the combination of the 
 mixing which sends the information into high frequencies and the application of the inverse Laplacian (or any operator of negative order), which averages out the small scales.
That is, dispersion transfers information to infinity in space whereas mixing 
transfers information to infinity in frequency.

\subsection{Statement and discussion}

In this section we state our nonlinear stability result and a few immediate corollaries. 
The key aspects of the proof are discussed after the statement.
 
The data will be chosen in a Gevrey space of class $1/s$ for $s> 1/2$ 
 \cite{Gevrey18}; the origin of this restriction is (mathematically) natural and arises from the weakly nonlinear effects, discussed further in \S\ref{sec:Growth}.
We note that the analogous space for the Vlasov equations with Coloumb/Newton interaction is Gevery-3 (e.g. $s = 1/3$) \cite{MouhotVillani11}. 
It is worth noting that unlike, for example, \cite{GM13} where 
the Gevrey regularity is required due to the linear growth of high frequencies, 
here (and \cite{MouhotVillani11}) the Gevrey regularity is required because of a potential \emph{nonlinear} frequency cascade. 

 Our main result is
\begin{theorem} \label{thm:Main} 
For all $1/2 < s \leq 1$, $\lambda_0 > \lambda^\prime > 0$ there exists an $\epsilon_0  = \epsilon_0(\lambda_0,\lambda^\prime,s) \leq 1/2$ such that for all $\epsilon  \leq  \epsilon_0   $  
 if $\omega_{in}$ satisfies 
  $\int \omega_{in} dx dy= 0$, $\int \abs{y\omega_{in}(x,y)} dx dy < \epsilon$ and   
\begin{align*}  
\norm{\omega_{in}}^2_{\G^{\lambda_0}} = \sum_k\int \abs{\hat{\omega}_{in}(k,\eta)}^2 e^{2\lambda_0 \abs{k,\eta}^{s}} d\eta \leq  \epsilon^2,
\end{align*}  
then there exists $f_\infty$ with $\int f_\infty dxdy = 0$ 
 and $\norm{f_\infty}_{\G^{\lambda^\prime}} \lesssim \epsilon$ such that 
\begin{equation} \label{main-omega} 
\norm{\omega(t,x + ty + \Phi(t,y),y) - f_\infty(x,y) }_{\G^{\lambda^\prime}} \lesssim \frac{\epsilon^2}{\jap{t}}, 
\end{equation} 
where $\Phi(t,y)$ is given explicitly by 
\begin{align} 
\Phi(t,y) = \frac{1}{2\pi}\int_0^t\int_\T U^x(\tau,x,y) dx d\tau = u_\infty(y)t + O(\epsilon),    \label{def:phi}
\end{align} 
with $u_\infty = \partial_y \partial_{yy}^{-1}\frac{1}{2\pi}\int_\T f_\infty(x,y) dx$.
Moreover, the velocity field $U$ satisfies
\begin{subequations} \label{ineq:damping}
\begin{align} 
\norm{\frac{1}{2\pi}\int U^x(t,x,\cdot) dx - u_\infty}_{\G^{\lambda^{\prime}}} &\lesssim \frac{\epsilon^2}{\jap{t}^2}, \label{ineq:xdamping_slow} \\ 
\norm{U^x(t) - \frac{1}{2\pi}\int U^x(t,x,\cdot) dx}_{L^2} &\lesssim \frac{\epsilon}{\jap{t}}, \label{ineq:xdamping} \\ 
\norm{U^y(t)}_{L^2} & \lesssim \frac{\epsilon}{\jap{t}^2}. \label{ineq:ydamping}
\end{align}
\end{subequations}
\end{theorem} 

\begin{remark} 
Of course, by time-reversibility, Theorem \ref{thm:Main} is also true $t \rightarrow -\infty$ for some $f_{-\infty}$ and $u_{-\infty}$ (which will generally not be equal to their $+\infty$ counterparts).
Also,  due to the Hamiltonian structure of \eqref{def:2DEuler}  (see e.g. \cite{AK98,Morrison98}),  one could only hope to prove asymptotic stability in a norm weaker 
than the norm in which the initial data is given. 
This is an important theme underlying our work, and the works of \cite{CagliotiMaffei98,HwangVelazquez09,MouhotVillani11}, which is that \emph{decay costs regularity}.   
\end{remark}

\begin{remark} \label{rmk:finfty}
From the proof of Theorem \ref{thm:Main}, it is clear that $\norm{\omega_{in} - f_\infty}_{\G^{\lambda^\prime}} \lesssim \epsilon^2$, as the effect of the nonlinear evolution is one order weaker than that of the linear evolution. 
\end{remark} 

\begin{remark} \label{rmk:FormTheta} 
Notice the surprisingly rapid convergence in \eqref{ineq:xdamping_slow} (it is of course matched by a similar rapid convergence of the $x$-averaged vorticity).
This arises from a subtle cancellation between the oscillations of $\omega$ and $U^y$ upon taking $x$ averages; 
 indeed it was previously believed that the convergence should be $O(t^{-1})$ and that \eqref{def:phi} involved a logarithmic correction. 
The origin of the rapid convergence rate can be best understood from studying the linearized problem \eqref{def:LinearizedEuler}, a computation that we carry out in \S\ref{apx:LogCorrection}.
\end{remark}

\begin{remark} \label{rmk:compactsupp}
The proof of Theorem \ref{thm:Main} implies that if $\omega_{in}$ is compactly supported then $\omega(t)$ remains supported in a strip $(x,y) \in \T \times [-R,R]$ for some $R > 0$ for all time.
\end{remark}

\begin{remark} 
The primary difficulty in treating more general shear flows is on the weakly nonlinear level (in contrast to the  Vlasov case), which would most clearly manifest in \S\ref{sec:Elliptic}.
More information on this difficulty, along with other related open problems, is discussed in \S\ref{sec:Conclusion}. 
\end{remark}

\begin{remark} 
Both Orr and Kelvin (and many others) expressed doubt that the inviscid problem was stable unless the set of permissible data was of a certain type, suggesting that for \emph{general data} the stability restriction would diminish with the inverse Reynolds number. 
To reconcile this viewpoint with Theorem \ref{thm:Main}, we conjecture that for high (but finite) Reynolds number flows, an analogous result to Theorem \ref{thm:Main} holds with initial data $\omega^R_{in} + \omega_{in}^\nu$ where $\omega^R_{in}$ has Gevrey-$\frac{1}{s}$ regularity uniformly in the Reynolds number and $\omega_{in}^{\nu}$ has Sobolev regularity with norm small with respect to the inverse Reynolds number.
Note that in the viscous case, the flow will return to Couette, but on a time-scale comparable to the Reynolds number.
We are currently investigating the proof of this conjecture.  
\end{remark}

\begin{remark} 
The spatial localization $\int \abs{y\omega_{in}(x,y)} dxdy < \epsilon$ is only used to assert that the velocity $U^x$ is in $L^2$ and to ensure the coordinate transformations used in the proof are not too drastic.  
This assumption can be relaxed to $\int \abs{y}^{\alpha}\abs{\omega_{in}(x,y)} dxdy < \epsilon$ for any $\alpha > 1/2$.
It might be possible to treat more general cases with $U^x \not\in L^2$ with some technical enhancements, as $U^x \in L^2$ does not play an important role in the proof. 
\end{remark}

The proof of Theorem \ref{thm:Main} (actually Remark \ref{rmk:finfty}) provides the following corollaries.

\begin{corollary} \label{cor1}
There exists an open set of smooth solutions to \eqref{def:2DEuler} for which $\set{\omega(t)}_{t \in \Real}$ is not pre-compact in $L^2$ as $t \rightarrow \pm\infty$. In particular, $\omega(t) \rightharpoonup \omega_\infty = \frac{1}{2\pi}\int_\T f_\infty(x,y) dx$ and in general $\norm{\omega_\infty}_2 < \norm{\omega(t)}_2$. 
\end{corollary} 
This shows the existence of solutions for which enstrophy is lost  to high frequencies 
in the limit $t \rightarrow \infty$, which to our knowledge was not previously known for 2D Euler in any setting. 
See  \cite{Shnirelman12,KS12,GSV13} 
 for further discussions on the physical 
interest of this fact and the potential relationship with 2D turbulence. 
A related corollary is the following which shows the linear growth of Sobolev norms as a direct consequence of the mixing. 
Compare with the construction of Denisov \cite{Denisov09} which yields super-linear growth of the gradient. 

\begin{corollary} \label{cor2}
There exists  an open set of  smooth solutions to \eqref{def:2DEuler} for which $\norm{\jap{\grad}^N\omega(t)}_2 \approx \jap{t}^N$ for all $N \in [0,\infty)$ and for all $s^\prime < s$,$\lambda_0 > 0$, $\norm{\omega(t)}_{\G^{\lambda_0;s^\prime}} \approx e^{\lambda_0 t^{s^\prime}}$. 
\end{corollary} 

Let us now outline the main new steps in the proof of Theorem \ref{thm:Main}. 
First, we provide a (well chosen) change of variable that adapts to the solution 
as it evolves and yields a new `relative' velocity which is time-integrable while keeping the Orr critical 
times as in \eqref{orr-cri}. 
This change of variables allows us to 
work on a quantity $f(t)$ which has a strong limit as $t$ goes to infinity. 
This is related to the notion of ``profile'' used in dispersive wave equations (see \cite{GMS12} for instance)
as well as the notion of ``gliding regularity'' in \cite{MouhotVillani11}. 
However, here it is important that the coordinate transformation depends on the solution, 
a source of large technical difficulty and an expression of the `quasi-linearity' alluded to above. 
 
A second new idea is the use of a special norm that loses regularity in a very 
precise way adapted to the Orr critical times and the associated nonlinear effect.
The construction of this norm is based on the so-called ``toy model'' which mimics the worse possible growth of high frequencies (derived in \S\ref{sec:Toy}).
 This special norm allows us to control the nonlinear growth due to the resonances at the critical times.  
However, this comes with a big danger: energy estimates and cancellations 
tend to dislike `unbalanced' norms, namely norms that assign different regularities 
to different frequencies (see for instance \cite{MN08} for a similar problem). 
In particular, by design, our norm is not an algebra.  
This is one of the main technical problems that we have to overcome, and here the decay of the velocity is crucial. 

In the course of the proof, we need to gain regularity from inverting 
the Laplacian to get the streamfunction from the vorticity; indeed the ellipticity is the origin of the decay. 
However, in the new variables the Laplacian is transformed to a weakly elliptic operator with coefficients that depend on the solution.  
This additional nonlinearity presents huge difficulties due to the limited regularity of the coefficients (relative to what is desired). 
This has similarities with elliptic estimates in domains with limited regularity 
 used for water waves  (see for instance \cite[Appendix A]{SZ08}). 
 Here, the interplay between regularity and decay will be crucial to ensure that the final estimate holds. 
As in \eqref{orr-cri}, the loss of ellipticity is an expression of the Orr critical times. 
It will be important for our work that the norm derived from the toy model precisely `matches' the loss of ellipticity.  

Related to the issue of inverting the Laplacian in the new variables is the final technical ingredient in our proof, which is the need to obtain a variety of precise controls on the evolving coordinate system (see Proposition \ref{prop:CoefControl_Real} below).
This will require us to quantify the convergence of the background shear flow in several ways. 
In particular, we will need to carefully estimate how the modes that depend on $x$ force those that do not and 
in fact, this forcing loses a derivative (see the last term in \eqref{def:barh}).
However, the estimates turn out to be possible precisely under the assumption of Gevrey class with $s \geq 1/2$ (see the discussion after Proposition \ref{prop:CoefControl_Real}).
Second to the toy model, here seems to be next most fundamental use of the regularity $s \geq 1/2$.

\subsection{Notation and conventions} \label{sec:Notation}
See \S\ref{Apx:LPProduct} for the Fourier analysis conventions we are taking.
A convention we generally use is to denote the discrete $x$ (or $z$) frequencies as subscripts.   
By convention we always use Greek letters such as $\eta$ and $\xi$ to denote frequencies in the $y$ or $v$ direction and lowercase Latin characters commonly used as indices such as $k$ and $l$ to denote frequencies in the $x$ or $z$ direction (which are discrete).
Another convention we use is to denote $K,M,N$ as dyadic integers $K,M,N \in \mathbb{D}$ where  
\begin{align*} 
\mathbb{D} = \set{\frac{1}{2},1,2,...,2^j,...}. 
\end{align*}
When a sum is written with indices $K,M,M^\prime,N$ or $N^\prime$ it will always be over a subset of $\mathbb D$. 
This will be useful when defining Littlewood-Paley projections and paraproduct decompositions, see \S\ref{Apx:LPProduct}. 
Given a function $m \in L^\infty$, we define the Fourier multiplier $m(\grad) f$ by 
\begin{align*} 
(\widehat{m(\grad)f})_k(\eta) =  m( (ik,i\eta) ) \hat{f}_k(\eta). 
\end{align*}   

We use the notation $f \lesssim g$ when there exists a constant $C > 0$ independent of the parameters of interest 
such that $f \leq Cg$ (we analogously $f \gtrsim g$ define). 
Similarly, we use the notation $f \approx g$ when there exists $C > 0$ such that $C^{-1}g \leq f \leq Cg$. 
We sometimes use the notation $f \lesssim_{\alpha} g$ if we want to emphasize that the implicit constant depends on some parameter $\alpha$.
We will denote the $l^1$ vector norm $\abs{k,\eta} = \abs{k} + \abs{\eta}$, which by convention is the norm taken in our work. 
Similarly, given a scalar or vector in $\Real^n$ we denote
\begin{align*} 
\jap{v} = \left( 1 + \abs{v}^2 \right)^{1/2}. 
\end{align*} 
We use a similar notation to denote the $x$ (or $z$) average of a function: $<f> = \frac{1}{2\pi}\int f(x,y) dx = f_0$. 
We also frequently use the notation $P_{\neq 0}f = f - f_0$. 
We denote the standard $L^p$ norms by $\norm{f}_{L^p}$. 
We make common use of the Gevrey-$\frac{1}{s}$ norm with Sobolev correction defined by 
\begin{align*} 
\norm{f}_{\G^{\lambda,\sigma;s}}^2 = \sum_{k}\int \abs{\hat{f}_k(\eta)}^2 e^{2\lambda\abs{k,\eta}^s}\jap{k,\eta}^{2\sigma} d\eta. 
\end{align*} 
Since most of the paper we are taking $s$ as a fixed constant, it is normally omitted.
We refer to this norm as the $\mathcal{G}^{\lambda,\sigma;s}$ norm and occasionally refer to the space of functions
\begin{align*} 
\mathcal{G}^{\lambda,\sigma;s} = \set{f \in L^2 :\norm{f}_{\G^{\lambda,\sigma}}<\infty}. 
\end{align*}
See \S\ref{apx:Gev} for a discussion of the basic properties of this norm and some related useful inequalities.

For $\eta \geq 0$, we define $E(\eta)\in \Integer$ to be the integer part.  
We define for $\eta \in \Real$ and $1 \leq \abs{k} \leq E(\sqrt{\abs{\eta}})$ with $\eta k \geq 0$, 
$t_{k,\eta} = \abs{\frac{\eta}{k}} - \frac{\abs{\eta}}{2\abs{k}(\abs{k}+1)} =  \frac{\abs{\eta}}{\abs{k}+1} +  \frac{\abs{\eta}}{2\abs{k}(\abs{k}+1)}$ and $t_{0,\eta} = 2 \abs{\eta}$ and 
the critical intervals  
\begin{align*} 
I_{k,\eta} = \left\{
\begin{array}{lr}
[t_{\abs{k},\eta},t_{\abs{k}-1,\eta}] & \textup{ if } \eta k \geq 0 \textup{ and } 1 \leq \abs{k} \leq E(\sqrt{\abs{\eta}}), \\ 
\emptyset & otherwise.
\end{array} 
\right. 
\end{align*} 
For minor technical reasons, we define a slightly restricted subset as the \emph{resonant intervals}
\begin{align*} 
\mathbf I_{k,\eta} = \left\{
\begin{array}{lr} 
I_{k,\eta} & 2\sqrt{\abs{\eta}} \leq t_{k,\eta}, \\ 
\emptyset & otherwise.
\end{array} 
\right. 
\end{align*} 
Note this is the same as putting a slightly more stringent requirement on $k$: $k \gtrsim \frac{1}{2}\sqrt{\abs{\eta}}$.

\section{Proof of Theorem \ref{thm:Main}} 

We now give the proof of Theorem \ref{thm:Main}, stating the primary steps as propositions which are proved in subsequent sections. 

\subsection{Coordinate transform} \label{sec:CoordinateTrans}
The original equations in vorticity form are \eqref{def:2DEuler}, and we are trying essentially to prove that 
\begin{align*} 
\omega(t,x,y) \rightarrow f_\infty\left(x - ty - u_\infty(y)t,y\right), 
\end{align*} 
as $t \rightarrow \infty$, where $u_\infty(y)$ is the correction to the shear flow determined by $f_\infty$.  
 From the initial data alone, there is no simple way to determine $u_\infty$;
it is chosen by the nonlinear evolution. 
In order to deal with this lack of information about how the final state evolves we choose a coordinate system which adapts to the solution and converges to the expected form as $t \rightarrow \infty$. 
The change of  coordinates used is $(t,x,y) \to (t,z,v)$, where 
\begin{subequations} \label{def:zv}
\begin{align} 
z(t,x,y) & = x - tv \\ 
v(t,y)   & = y + \frac{1}{t}\int_0^t < U^x >(s,y) ds, \label{def:v}
\end{align}      
\end{subequations}
where we recall $< w >$ denotes the average of $w $ in the $x$ variable (or equivalently in the 
$z$ variable), namely 
$ <w>  = \frac{1}{2\pi}\int_\T  w  dx  $. 
The reason for the change   $y \to v$ is not immediately clear, however $v$ is named as such since it is an approximation for the background shear flow. 
If the velocity field in the integrand were constant in time, then we are simply transforming the $y$ variables so that the shear appears linear.
It will turn out that this choice of $v$ ensures that the Biot-Savart law is in a form amenable to Fourier analysis in the variables $(z,v)$; in particular, even when the shear is time-varying we may still study the {\em Orr critical times} as was 
explained in \eqref{orr-cri}.   
In this light, the motivation for the shift in $z$ is clear: we are following the flow in the horizontal variable to guarantee compactness, as done even by Orr \cite{Orr07}, Kelvin \cite{Kelvin87} and many authors since then.

Define $f(t,z,v)  = \omega(t,x,y)  $ and    $\phi(t,z,v) = \psi(t,x,y) $, hence
\begin{align*} 
  \partial_t \omega = \partial_t f +\partial_t z   \partial_zf    + \partial_t v \partial_{v}f, \quad\quad 
\partial_x \omega = \partial_z f, \quad\quad \partial_y \omega =\partial_{y}v   \left(\partial_{v} f - t \partial_z f\right), 
\end{align*}
where
\begin{align*}
\partial_t z & = -y - < U^x>(t,y) \\ 
\partial_t v & = \frac{1}{t}\left[< U^x >(t,y) - \frac{1}{t}\int_0^t < U^x>(s,y) ds\right] \\ 
\partial_y v & = 1 - \frac{1}{t}\int_0^t <\omega>(s,y) ds \\ 
\partial_{yy} v & = - \frac{1}{t}\int_0^t \partial_y <\omega>(s,y) ds. 
\end{align*}
Expressing $[\partial_tv](t,v) = \partial_tv(t,y)$, $v^{\prime}(t,v) = \partial_y v(t,y)$ and $v^{\prime\prime}(t,v) = \partial_{yy} v(t,y)$, we get the following evolution equation for $f$,
\begin{align*} 
\partial_t f + [\partial_t v] \partial_{v}f + \partial_t z\partial_z f = - y\partial_z f + v^\prime \left[\partial_v\phi + \partial_z\phi \partial_{v}z - \partial_z \phi \partial_{v}z \right] \partial_z f - v^{\prime}\partial_z \phi \partial_{v} f.
\end{align*} 
Using the definition of $\partial_t z$ and the Biot-Savart law to transform $<U^x>$ to $-v^\prime \partial_{v}<\phi>$ in the new variables, this becomes
\begin{align*}
\partial_t f  - \left(v^\prime \partial_{v}(\phi - <\phi>)\right)\partial_z f   
     +  \left( [\partial_t v]  + v^\prime \partial_z\phi \right)  \partial_{v}f     = 0. 
\end{align*}
The Biot-Savart law also gets transformed into: 
\begin{align}
f & = \partial_{zz}\phi  + (v^\prime)^2\left( \partial_v - t \partial_z \right)^2\phi + v^{\prime\prime}\left(\partial_{v} - t \partial_z\right) \phi = \Delta_t \phi. \label{def:Deltat}
\end{align} 
The original 2D Euler system \eqref{def:2DEuler} is expressed as 
\begin{align} \label{Euler2}
 \left\{
\begin{array}{l} 
\partial_tf + u \cdot \grad_{z,v} f = 0,  \\
u =  (0, [\partial_t v] )   + v^\prime \grad_{z,v}^\perp P_{\neq 0} \phi,  \\ 
\phi =\Delta_t^{-1}[f] . 
\end{array}
\right.
\end{align} 
It what follows we will write $\grad_{z,v} = \grad$ and specify when other variables are used. 
Next we transform the momentum equation to allow us to express $[\partial_t v]$ in a form amenable to estimates.
Denoting $\tilde u(t,z,v) = U^x(t,x,y)$ and $p(t,z,v) = P(t,x,y)$ we have by the same derivation on $f$, 
\begin{align*} 
\partial_t \tilde u + [\partial_t v] \partial_{v}\tilde u + \partial_z P_{\neq 0} \phi + v^\prime \grad^\perp P_{\neq 0} \phi \cdot \grad \tilde u = -\partial_z p. 
\end{align*} 
Taking averages in $z$ we isolate the zero mode of the velocity field, 
\begin{align} 
\partial_t \tilde u_0 + [\partial_t v] \partial_{v} \tilde u_0 +  v^\prime \left<\grad^\perp P_{\neq 0}\phi \cdot \grad \tilde u \right> = 0. \label{eq:tildeu0_moment}
\end{align}
Finally, one can express $v^{\prime}$ and $[\partial_t v]$ as solutions to a system of PDE in the $(t,v)$ variables coupled to \eqref{Euler2} (see \S\ref{sec:CordDeriv} below for a detailed derivation): 
\begin{subequations} \label{def:vPDE}
\begin{align}
\partial_t(t(v^\prime - 1)) + [\partial_t v] t\partial_v v^\prime  = -f_0 \label{def:pdevp}\\
\partial_t[\partial_t v] + \frac{2}{t}[\partial_t v] + [\partial_t v]\partial_v[\partial_t v]  = -\frac{v^\prime}{t}\jap{\grad^\perp P_{\neq 0}\phi \cdot \grad \tilde u}  \label{def:ddtv} \\
v^{\prime\prime}(t,v) = v^\prime(t,v) \partial_v v^{\prime}(t,v).  \label{def:vpp}
\end{align}
\end{subequations}
Note that to leading order in $\epsilon$, one can express $v^\prime-1$ as a time average of $-f_0$.  
Note also that we have a simple expression for $\partial_v \tilde u_0$ from the Biot-Savart law: 
\begin{align} 
\partial_{v}\tilde u_0(t,v) = \frac{1}{v^\prime(t,v)} \partial_y U^x_0(t,y) = -\frac{1}{v^\prime(t,v)}\omega_0(t,y) = -\frac{1}{v^\prime(t,v)} f_0(t,v). \label{eq:tildeu0_simple}
\end{align}

Given a priori estimates on the system \eqref{Euler2}, \eqref{def:vPDE}, we can recover estimates
on the original system \eqref{def:2DEuler} by the inverse function theorem as long as $v^\prime - 1$ remains sufficiently small (see \S\ref{sec:ProofConcl}). 
Compared to the original system  \eqref{def:2DEuler},  the system \eqref{Euler2}, \eqref{def:vPDE}   
appears much more complicated and nonlinear.  
Indeed, $u$ is not divergence free and the dependence of $\phi$ on $f$ through $\Delta_t $ is significantly more subtle than in the original variables. 
The main advantage of \eqref{Euler2} is that $u$ formally has an integrable decay, indeed, we will see that if one is willing to pay four derivatives, the decay rate is formally  $O(t^{-2}\log t)$ (the decay we deduce is not quite as sharp). 

\subsection{Main energy estimate} \label{sec:MainEnergy}
In light of the previous section, our goal is to control solutions to \eqref{Euler2} uniformly in a suitable norm as $t \rightarrow \infty$.   
The key idea we use for this is the carefully designed time-dependent norm written as 
\begin{align*} 
\norm{A(t,\grad)f}^2_2 = \sum_k\int_\eta \abs{A_k(t,\eta)\hat f_k(t,\eta)}^2 d\eta. 
\end{align*} 
The multiplier $A$ has several components,
\begin{align*} 
A_k(t,\eta) = e^{\lambda(t)\abs{k,\eta}^s }  \jap{k,\eta}^\sigma  J_k(t,\eta). 
\end{align*}
The index $\lambda(t)$ is the bulk Gevrey$-\frac{1}{s}$ regularity and will be chosen to satisfy 
\begin{subequations} \label{def:lambdat}
\begin{align} 
\lambda(t) & = \frac{3}{4}\lambda + \frac{1}{4}\lambda^\prime, \quad t \leq 1 \label{def:lambdashort} \\   
\dot\lambda(t) & = -\frac{\delta_\lambda}{\jap{t}^{2\tilde q}}(1 + \lambda(t)), \quad t > 1
\end{align} 
\end{subequations}
where $\delta_\lambda \approx \lambda - \lambda^\prime$ is a small parameter that ensures $\lambda(t) > \lambda/2 + \lambda^{\prime}/2$ and $1/2 < \tilde q \leq s/8 + 7/16$ is a parameter chosen by the proof. 
The reason for \eqref{def:lambdashort} is to account for the behavior of the solution on the time-interval $[0,1]$; see Lemma \ref{lem:shorttime} for this relatively minor detail. 
The use of a time-varying index of regularity is classical, for example the Cauchy-Kovalevskaya local existence theorem of Nirenberg \cite{Nirenberg72,Nishida77}.
For more directly relevant works which use norms of this type, see \cite{FoiasTemam89,LevermoreOliver97,Chemin04,KukavicaVicol09,CGP11,MouhotVillani11}. 
Let us remark here that to study analytic data, $s = 1$, we would need to add an additional Gevrey-$\frac{1}{\beta}$ correction to $A$ with $1/2 < \beta < 1$ as an intermediate regularity so that we may take advantage of certain beneficial properties of Gevrey spaces; see for example Lemma \ref{lem:GevProdAlg}.
In this case, the analytic regularity would simply be propagated more or less passively through the proof. 
Using the same idea, we may assume without loss of generality that $s$ is close to $1/2$ (say $s < 2/3$), which simplifies some of the technical details but is not essential. 
The Sobolev correction with $\sigma > 12$ fixed is included mostly for technical convenience so we may easily quantify loss of derivatives without disturbing the index of regularity.   
We will also use the slightly stronger multiplier $A^R(t,\eta)$ that satisfies $A^R(t,\eta) \geq A_0(t,\eta)$ to control the coefficients $v^\prime$ and $v^{\prime\prime}$; see \eqref{def:AR} below for the definition. 

The main multiplier for dealing with the Orr mechanism and the associated nonlinear growth is 
\begin{align} 
J_k(t,\eta) =   \frac{e^{\mu\abs{\eta}^{1/2}}}{w_k(t,\eta)}  + e^{\mu\abs{k}^{1/2}}, \label{def:J}
\end{align} 
where $ w_k(t,\eta)$ is constructed in \S\ref{sec:Growth} and describes the expected `worst-case' growth due to nonlinear interactions at the critical times. 
What will be important is that $J$ imposes more regularity on modes which satisfy $t \sim \frac{\eta}{k}$ (the `resonant modes') than those that do not (the `non-resonant modes').
The multiplier $J$  replaces growth in time 
by  controlled loss of regularity  and 
 is reminiscent of the notion of losing regularity 
estimates used in \cite{BC94,CM01}.
 One of the main differences is that here we have 
to be more precise in the sense that the loss of regularity occurs for different 
frequencies during different time intervals. 

With this special norm, we can define our main energy: 
\begin{align} \label{Et}
E(t) &= \frac{1}{2}\norm{A(t)f(t)}_2^2 + E_{v}(t), 
\end{align}
where, for some constants $K_v$, $K_D$ depending only on $s,\lambda,\lambda^\prime$ fixed by the proof,  
\begin{align} 
E_{v}(t) = \jap{t}^{2+2s}\norm{\frac{A}{\jap{\partial_v}^s} v^\prime \partial_v [\partial_t v](t)}_{2}^2 + \jap{t}^{4-K_D\epsilon}\norm{[\partial_t v](t)}_{\G^{\lambda(t),\sigma-6}}^2 + \frac{1}{K_v}\norm{A^R(v^\prime-1)(t)}_2^2.
\end{align}
In a sense, there are two coupled energy estimates: the one on $Af$ and the one on $E_v$. The latter quantity is encoding information about the coordinate system, or equivalently, the evolution of the background shear flow.  
It turns out $v^\prime \partial_v [\partial_t v]$ is a physical quantity that measures the convergence of the $x$-averaged vorticity to its time average (see \eqref{def:barh2} in \S\ref{sec:CordDeriv}) and 
satisfies a useful PDE (see \eqref{def:barh} in \S\ref{sec:PropCtrlReal}). 
It will be convenient to get two separate estimates on $[\partial_t v]$ as opposed to just one ($[\partial_t v]$ is essentially measuring how rapidly the $x$-averaged velocity is converging to its time average).
 
The goal is to prove by a continuity argument that this energy $E(t)$ (together with some related quantities) is uniformly 
bounded for all time if $\epsilon$ is sufficiently small.
We define the following controls referred to in the sequel as the bootstrap hypotheses,
\begin{itemize} 
\item[(B1)] $E(t) \leq 4\epsilon^2$; 
\item[(B2)] $\norm{v^\prime - 1}_\infty \leq \frac{3}{4}$ 
\item[(B3)] `CK' integral estimates (for `Cauchy-Kovalevskaya'): 
\begin{align*}
\int_0^t \left[ CK_\lambda(\tau) + CK_w(\tau) + CK_w^{v,2}(\tau) + CK_\lambda^{v,2}(\tau) \right. & \\ & \hspace{-7cm} \left. + K_v^{-1}\left( CK_w^{v,1}(\tau) + CK_\lambda^{v,1}(\tau) \right) +  K_v^{-1}\sum_{i=1}^2\left(CCK_w^i(\tau) + CCK_\lambda^i(\tau)\right) \right] d\tau \leq 8\epsilon^2 
\end{align*}
\end{itemize} 
The CK terms above that appear without the $K_v^{-1}$ prefactor arise from the time derivatives of $A(t)$
and are naturally controlled by the energy estimates we are making. 
The others are related quantities that are controlled separately in Proposition \ref{prop:CoefControl_Real} below. 

By the well-posedness theory for 2D Euler in Gevrey spaces \cite{BB77,FT98,FoiasTemam89,LevermoreOliver97,KukavicaVicol09} we may safely ignore the time interval (say) $[0,1]$ by further restricting the size of the initial data. 
That is, we have the following lemma; see \S\ref{apx:coTrans} for a sketch of the proof. 
\begin{lemma} \label{lem:shorttime}
For all $\epsilon > 0$, there exists an $\epsilon^\prime > 0$ such that if $\norm{\omega_{in}}_{\G^{\lambda_0}} < \epsilon^\prime$ 
and $\int \abs{y\omega_{in}} dx dy < \epsilon^\prime$, then 
$\sup_{t \in [0,1]}\norm{f(t)}_{\G^{3\lambda/4 + \lambda^\prime/4, \sigma}} < \epsilon$, $E(1) < \epsilon$, $\sup_{t \in [0,1]} \norm{1-v^\prime(t)}_\infty < 6/10$ with
\begin{align}
\int_0^1 \left[ CK_\lambda(\tau) + CK_w(\tau) + CK_w^{v,2}(\tau) + CK_\lambda^{v,2}(\tau) \right. & \nonumber \\ & \hspace{-7cm} \left. + K_v^{-1}\left( CK_w^{v,1}(\tau) + CK_\lambda^{v,1}(\tau) \right) +  K_v^{-1}\sum_{i=1}^2\left(CCK_w^i(\tau) + CCK_\lambda^i(\tau)\right) \right] d\tau \leq 2\epsilon^2. \label{ineq:CKctrlShort}
\end{align}
\end{lemma}

By Lemma \ref{lem:shorttime}, for the rest of the proof we may focus on times $t \geq 1$. 
Let $I_E$ be the connected set of times $t \geq 1$ such that the bootstrap hypotheses (B1-B3) are all satisfied.  
We will work on regularized solutions for which we know $E(t)$ takes values continuously in time, and hence $I_E$ is a closed interval $[1,T^\star]$ with $T^\star > 1$.
    The bootstrap is complete if we show that $I_E$ is also open, which is the purpose of the following proposition, the proof of which constitutes the majority of this work. 

\begin{proposition}[Bootstrap] \label{lem:bootstrap}
There exists an $\epsilon_0 \in (0,1/2)$ depending only on $\lambda,\lambda^\prime,s$ and $\sigma$ such that if $\epsilon < \epsilon_0$, and on $[1,T^\star]$ the bootstrap hypotheses (B1-B3) hold, then for $\forall \, t \in [1,T^\star]$,
\begin{enumerate} 
\item $E(t) < 2\epsilon^2$,
\item $\norm{1-v^\prime}_\infty < \frac{5}{8}$,
\item and the CK controls satisfy:
\begin{align*} 
\int_0^t \left[ CK_\lambda(\tau) + CK_w(\tau) + CK_w^{v,2}(\tau) + CK_\lambda^{v,2}(\tau) \right. & \\ & \hspace{-7cm} + \left. K_v^{-1}\left( CK_w^{v,1}(\tau) + CK_\lambda^{v,1}(\tau) \right) +  K_v^{-1}\sum_{i=1}^2\left(CCK_w^i(\tau) + CCK_\lambda^i(\tau)\right) \right] d\tau \leq 6\epsilon^2,   
\end{align*}
\end{enumerate}
from which it follows that $T^\star = +\infty$.
\end{proposition}
The remainder of the section is devoted to the proof of Proposition \ref{lem:bootstrap}, the primary step being to show that on $[1,T^\star]$,  we have 
\begin{align}  
E(t) + \frac{1}{2}\int_1^t \left[ CK_\lambda(\tau) + CK_w(\tau) + CK_w^{v,2}(\tau) + CK_\lambda^{v,2}(\tau) \right. & \nonumber \\ & \hspace{-9cm} \left. + K_v^{-1}\left( CK_w^{v,1}(\tau) + CK_\lambda^{v,1}(\tau) \right) +  K_v^{-1}\sum_{i=1}^2\left(CCK_w^i(\tau) + CCK_\lambda^i(\tau)\right) \right] d\tau \leq E(1) + K \epsilon^3 \label{ineq:Ectrl}
\end{align} 
for some constant $K$ which is independent of $\eps$ and $T^\star$. 
If $\epsilon$ is sufficiently small then \eqref{ineq:Ectrl} implies Proposition \ref{lem:bootstrap}. 
Indeed, the control $\norm{1-v^\prime} < 5/8$ is an immediate consequence of (B1) by Sobolev embedding for $\epsilon$ sufficiently small.   

To prove \eqref{ineq:Ectrl}, it is natural to compute the time evolution of $E(t)$,
\begin{align*} 
\frac{d}{dt}E(t) & = \frac{1}{2}\frac{d}{dt}\int \abs{Af}^2 dx + \frac{d}{dt}E_{v}(t) 
\end{align*}
The first contribution is of the form 
\begin{align} 
\frac{1}{2}\frac{d}{dt}\int \abs{Af}^2 dx  
& = -CK_\lambda - CK_{w} - \int Af A(u \cdot \grad f) dx, \label{ineq:timeevoE}
\end{align}
where the CK stands for `Cauchy-Kovalevskaya' since these three terms arise from the progressive weakening of the norm in time, and are expressed as 
\begin{subequations} \label{def:CKterms}
\begin{align}
CK_\lambda & = -\dot{\lambda}(t)\norm{\abs{\grad}^{s/2}Af}_2^2 \\ 
CK_{w} & = \sum_k\int\frac{\partial_t w_k(t,\eta)}{w_k(t,\eta)} e^{\lambda(t)\abs{k,\eta}^{s}}\jap{k,\eta}^\sigma \frac{e^{\mu\abs{\eta}^{1/2}}}{w_k(t,\eta)}A_{k}(t,\eta)\abs{\hat{f}_k(t,\eta)}^2 d\eta. 
\end{align}   
\end{subequations}
In what follows we define
\begin{subequations}
\begin{align} 
\tilde{J}_k(t,\eta) & = \frac{e^{\mu\abs{\eta}^{1/2}}}{w_k(t,\eta)}, \label{def:tildeJB} \\ 
\tilde{A}_k(t,\eta) & = e^{\lambda(t)\abs{k,\eta}^{s}}\jap{k,\eta}^\sigma  \tilde J_k(t,\eta). 
\end{align} 
\end{subequations}
Note that $\tilde A \leq A$ and if $\abs{k} \leq  \abs{\eta}$ then $A \lesssim \tilde A$.

Strictly speaking, equality  \eqref{ineq:timeevoE} is not rigorous since it involves a derivative of $Af$, which is not a priori well-defined. 
 To make this calculation rigorous, we have first to approximate the initial data 
of \eqref{def:2DEuler} by (for instance) analytic initial data  and use that the global 
solutions of \eqref{def:2DEuler} stay analytic for all time (see 
\cite{BB77,FoiasTemam89,FT98}). 
Hence, we can perform all calculations on these solutions with regularized 
initial data and then perform a passage to the limit to infer that \eqref{ineq:Ectrl} still holds.
Similarly, the bootstrap is performed on these regularized solutions for which $E(t)$ takes values continuously in time.  

To treat the main term in \eqref{ineq:timeevoE}, begin by integrating by parts, as in the techniques \cite{FoiasTemam89,LevermoreOliver97,KukavicaVicol09}
\begin{align} 
\int Af A(u \cdot \grad f) dx = -\frac{1}{2}\int \grad \cdot u \abs{Af}^2 dx + \int Af \left[ A(u \cdot \grad f) - u \cdot \grad Af \right] dx. \label{eq:Aenergu}
\end{align} 
Notice that the relative velocity is not divergence free:
\begin{align*} 
\grad \cdot u = \partial_{v}[\partial_tv] + \partial_{v}v^\prime \partial_{z}\phi. 
\end{align*} 
The first term is controlled by the bootstrap hypothesis (B1). 
For the second term we use the `lossy' elliptic estimate, Lemma \ref{lem:LossyElliptic}, 
which shows that under the bootstrap hypotheses we have 
\begin{align} 
\norm{P_{\neq 0}\phi(t)}_{\G^{\lambda(t),\sigma-3}} \lesssim \frac{\epsilon}{\jap{t}^2}. \label{ineq:LossyIntro}
\end{align}
Therefore, by Sobolev embedding, $\sigma > 5$ and the bootstrap hypotheses, 
\begin{align} 
\abs{\int \grad \cdot u \abs{Af}^2 dx} \leq \norm{\grad u}_{\infty} \norm{Af}_2^2 \lesssim \frac{\epsilon}{\jap{t}^{2-K_D\epsilon/2}}\norm{Af}_2^2 \lesssim \frac{\epsilon^3}{\jap{t}^{2-K_D\epsilon/2}}. \label{ineq:div}
\end{align}

To handle the commutator, $\int Af \left[ A(u \cdot \grad f) - u \cdot \grad Af \right] dx$, we use a paraproduct decomposition (see e.g. \cite{Bony81,BCD11}).
Precisely, we define three main contributions: \emph{transport}, \emph{reaction} and \emph{remainder}: 
\begin{align}  
\int Af \left[ A(u \cdot \grad f) - u \cdot \grad Af \right] dx = \frac{1}{2\pi}\sum_{N \geq 8}T_N + \frac{1}{2\pi}\sum_{N \geq 8}R_N +  \frac1{2\pi} \mathcal{R}, \label{def:commdiv} 
\end{align}
where (the factors of $2\pi$ are for future notational convenience)
\begin{align*} 
T_N & = 2\pi\int Af \left[ A(u_{<N/8} \cdot \grad f_N) - u_{<N/8} \cdot \grad Af_N \right] dx \\ 
R_N & = 2\pi\int Af \left[ A(u_{N} \cdot \grad f_{<N/8}) - u_{N} \cdot \grad Af_{<N/8} \right] dx \\ 
\mathcal{R} & = 2\pi\sum_{N \in \mathbb D}\sum_{\frac{1}{8}N \leq N^\prime \leq 8N} \int Af \left[ A(u_{N} \cdot \grad f_{N^\prime}) - u_{N} \cdot \grad Af_{N^\prime} \right] dx. 
\end{align*}
Here $N \in \mathbb D = \set{\frac{1}{2},1,2,4,...,2^j,...}$ and $g_N$ denotes the $N$-th Littlewood-Paley projection and $g_{<N}$ means the Littlewood-Paley projection onto frequencies less than $N$ (see \S\ref{Apx:LPProduct} for the Fourier analysis conventions we are taking). 
Formally, the paraproduct decomposition \eqref{def:commdiv} represents a kind of `linearization' for the evolution of higher frequencies around the lower frequencies.
The terminology `reaction' is borrowed from Mouhot and Villani \cite{MouhotVillani11} (see \S\ref{sec:Conclusion} for more information).

Controlling the transport contribution is the subject of \S\ref{sec:Transport}, in which we prove: 
\begin{proposition}[Transport] \label{prop:Transport}
Under the bootstrap hypotheses, 
\begin{align*} 
\sum_{N \geq 8}\abs{T_N} \lesssim \epsilon CK_\lambda + \epsilon CK_{w} + \frac{\epsilon^3}{\jap{t}^{2 - K_D\epsilon/2}}. 
\end{align*} 
\end{proposition}
The proof of Proposition \ref{prop:Transport} uses ideas from the works of \cite{FoiasTemam89,LevermoreOliver97,KukavicaVicol09}. 
Since the velocity $u$ is restricted to `low frequency', we will have the available regularity required to apply \eqref{ineq:LossyIntro}.
However, the methods of  \cite{FoiasTemam89,LevermoreOliver97,KukavicaVicol09} do not adapt immediately since $J_k(t,\eta)$ is imposing slightly different regularities to certain frequencies, which is problematic.
Physically speaking, we need to ensure that resonant frequencies do not incur a very large growth due to nonlinear interactions with non-resonant frequencies (which are permitted to be slightly larger than the resonant frequencies). 
Controlling this imbalance is why $CK_w$ appears in Proposition \ref{prop:Transport}. 

Controlling the reaction contribution in \eqref{def:commdiv} is the subject of \S\ref{sec:RigReac}. 
Here we cannot apply \eqref{ineq:LossyIntro}, as an estimate on this term requires $u$ in the highest norm on which we have control, and hence we have no regularity to spare.
Physically, here in the reaction term is where the dangerous nonlinear effects are expressed and a great deal of precision is required to control them.
In \S\ref{sec:RigReac} we prove
\begin{proposition}[Reaction] \label{prop:ReactionIntro}
Under the bootstrap hypotheses, 
\begin{align} 
\sum_{N \geq 8}\abs{R_N} &  \lesssim \epsilon CK_\lambda + \epsilon CK_{w} + \frac{\epsilon^3}{\jap{t}^{2-K_D\epsilon/2}} + \epsilon CK_\lambda^{v,1} + \epsilon CK_w^{v,1} \nonumber \\  
 & \quad  + \epsilon\norm{\jap{\frac{\partial_v}{t\partial_z}}^{-1} \left(\partial_z^2 + (\partial_v-t\partial_z)^2\right)\left(\frac{\abs{\grad}^{s/2}}{\jap{t}^s}A + \sqrt{\frac{\partial_t w}{w}}\tilde A\right) P_{\neq 0}\phi}_2^2. \label{ineq:ReacIntro}
\end{align} 
\end{proposition} 
The $CK^{v,1}$ terms are defined below in \eqref{ineq:CKcoefs}. 
The first step to controlling the term in \eqref{ineq:ReacIntro} involving $\phi$ is Proposition \ref{lem:PrecisionEstimateNeq0}, proved in \S\ref{sec:Precf}. This proposition treats $\Delta_t$ as a perturbation of $\partial_{zz} + (\partial_v - t\partial_z)^2$ and passes the multipliers in the last term of \eqref{ineq:ReacIntro} onto $f$ \emph{and} the coefficients of $\Delta_t$. 
Physically, these latter contributions are indicating the nonlinear interactions between the higher modes of $f$ and the coefficients $v^\prime$, $v^{\prime\prime}$ (which involve time-averages of $f_0$ \eqref{def:vPDE}). 
\begin{proposition}[Precision elliptic control] \label{lem:PrecisionEstimateNeq0} 
Under the bootstrap hypotheses, 
\begin{align} 
\norm{\jap{\frac{\partial_v}{t\partial_z}}^{-1} \left(\partial_z^2 + (\partial_v-t\partial_z)^2\right)\left(\frac{\abs{\grad}^{s/2}}{\jap{t}^s}A + \sqrt{\frac{\partial_t w}{w}}\tilde A\right) P_{\neq 0}\phi }_2^2 & \nonumber \\ & 
\hspace{-6cm} \lesssim CK_\lambda + CK_w  + \epsilon^2 \sum_{i=1}^2 CCK_\lambda^i + CCK_w^i, \label{ineq:PrecisionPhiNeq0}
\end{align} 
where the `coefficient Cauchy-Kovalevskaya' terms are given by
\begin{subequations} \label{def:QL}
\begin{align}
CCK_\lambda^1 & =  -\dot\lambda(t)\norm{\abs{\partial_v}^{s/2}A^R\left(1-(v^{\prime})^2\right) }_2^2, \\ 
CCK_w^1 & = \norm{\sqrt{\frac{\partial_t w}{w}}A^R\left(1-(v^{\prime})^2\right)}_2^2,  \\
CCK_\lambda^2 & = -\dot{\lambda}(t)\norm{\abs{\partial_v}^{s/2}\frac{A^R}{\jap{\partial_v}} v^{\prime\prime} }_2^2, \\
CCK_w^2 & = \norm{\sqrt{\frac{\partial_t w}{w}}\frac{A^R}{\jap{\partial_v}} v^{\prime\prime} }_2^2. 
\end{align}
\end{subequations}
\end{proposition}

The next step in the bootstrap is to provide good estimates on the coordinate system and the associated CK and CCK terms, a procedure that is detailed in \S\ref{sec:CoordControls}.  
The following proposition provides controls on $v^\prime-1$,  the CCK terms arising in \eqref{def:QL}, 
the pair $[\partial_t v]$, $v^\prime \partial_v [\partial_t v]$ and finally all of the $CK^{v,i}$ terms.
The norm defined by $A^R(t)$ is stronger than that defined by $A(t)$, which we use to measure $f$. 
 It turns out that we will be able to propagate this stronger regularity on $v^\prime - 1$ due to a time-averaging effect, derived via energy estimates on \eqref{def:vPDE}.  
By contrast, $[\partial_t v]$ is expected basically to have the regularity of $\tilde u_0$ and hence even \eqref{ineq:bhc} has $s$ fewer derivatives than expected. 
On the other hand, it has a significant amount of time decay, which near critical times can be converted into regularity. 

\begin{proposition}[Coordinate system controls] \label{prop:CoefControl_Real}
Under the bootstrap hypotheses, for $\epsilon$ sufficiently small and $K_v$ sufficiently large there is a $K > 0$ such that 
\begin{subequations} \label{ineq:CCKEEsts}
\begin{align} 
\norm{A^R(v^\prime - 1)(t)}_2^2 + \frac{1}{2}\int_1^t \sum_{i=1}^2 CCK_w^{i}(\tau) d\tau + \frac{1}{2}\int_1^t  \sum_{i=1}^2 CCK_\lambda^{i}(\tau) d\tau  & \leq \frac{1}{2}K_v \epsilon^2 \label{ineq:hc} \\ 
\jap{t}^{2+2s}\norm{\frac{A}{\jap{\partial_v}^{s}} v^\prime \partial_v[\partial_t v]  }_2^2 + \frac{1}{2}\int_1^t CK_\lambda^{v,2}(\tau) + CK_w^{v,2}(\tau) d\tau & \leq \epsilon^2 + K\epsilon^3 \label{ineq:bhc}  \\ 
\jap{t}^{4-K_D\epsilon} \norm{[\partial_t v]}_{\G^{\lambda(t),\sigma-6}}^2 & \leq \epsilon^2 + K\epsilon^3  \label{ineq:partialvintro} \\ 
\int_1^t CK_\lambda^{v,1}(\tau) + CK_w^{v,1}(\tau) d\tau & \leq \frac{1}{2} K_v \epsilon^2, \label{ineq:ckvctrl}
\end{align}
\end{subequations}
where the $CK^{v,i}$ terms are given by 
\begin{subequations} \label{ineq:CKcoefs} 
\begin{align}
CK_{w}^{v,2}(\tau) & = \jap{\tau}^{2+2s} \norm{\sqrt{\frac{\partial_t w}{w}} \frac{A}{\jap{\partial_v}^s} v^\prime \partial_v [\partial_t v](\tau)}_2^2 \\ 
CK_{\lambda}^{v,2}(\tau) & = \jap{\tau}^{2+2s} (-\dot{\lambda}(\tau))\norm{\abs{\partial_v }^{s/2}\frac{A}{\jap{\partial_v}^s}v^\prime \partial_v [\partial_t v](\tau)}_2^2 \\ 
CK_{w}^{v,1}(\tau) & = \jap{\tau}^{2+2s} \norm{\sqrt{\frac{\partial_t w}{w}} \frac{A}{\jap{\partial_v}^s}[\partial_t v](\tau)}_2^2 \\ 
CK_{\lambda}^{v,1}(\tau) & = \jap{\tau}^{2+2s} (-\dot{\lambda}(\tau))\norm{\abs{\partial_v }^{s/2}\frac{A}{\jap{\partial_v}^s}[\partial_t v](\tau)}_2^2. 
\end{align}
\end{subequations}
\end{proposition}
Note that neither \eqref{ineq:bhc} nor \eqref{ineq:partialvintro} controls the other: at higher frequencies the former is stronger than the latter and at lower frequencies the opposite is true. 
One of the advantages of this scheme is that $v^\prime \partial_v [\partial_t v]$ satisfies an 
equation that is simpler than $[\partial_t v]$ and so is easier to get good estimates on.
Both \eqref{ineq:bhc} and \eqref{ineq:partialvintro} are linked to the convergence of the background shear flow; in particular, they rule out that the background flow oscillates or wanders due to nonlinear effects.  
An interesting structure arises in the coupling between $v^\prime \partial_v [\partial_t v]$ and $f$ that is similar to the abstract system (see \eqref{def:barh})
\begin{align*}
\partial_t f & = \frac{1}{\jap{t}^2}g \\ 
\partial_t g & = \frac{1}{\jap{t}^2}\abs{\grad} f.
\end{align*}
One can verify that this system is globally well-posed in Gevrey-$\frac{1}{s}$ for $s \geq 1/2$ using an energy estimate coupling $f$ and $\jap{\grad}^{-s}g$, but not if $s < 1/2$.  

Finally we need to control the remainder term in \eqref{def:commdiv}. 
This is straightforward and is detailed in \S\ref{sec:RemainderEnergy}.
There we prove
\begin{proposition}[Remainders] \label{prop:RemainderIntro}
Under the bootstrap hypotheses,
\begin{align*} 
\mathcal{R} \lesssim \frac{\epsilon^3}{\jap{t}^{2 - K_D\epsilon/2}}. 
\end{align*}
\end{proposition} 

Collecting Propositions \ref{prop:Transport}, \ref{prop:ReactionIntro}, \ref{lem:PrecisionEstimateNeq0}, \ref{prop:CoefControl_Real}, \ref{prop:RemainderIntro} with \eqref{def:commdiv} and \eqref{ineq:div}, we have finally \eqref{ineq:Ectrl} for $\epsilon$ sufficiently small 
with constants independent of both $\epsilon$ and $T^\star$; hence for $\epsilon$ sufficiently small we may propagate the bootstrap control and prove Proposition \ref{lem:bootstrap}.

\subsection{Conclusion of proof} \label{sec:ProofConcl}
By Proposition \ref{lem:bootstrap} we have a global uniform bound on $E(t)$, and therefore the uniform bounds 
\begin{align} 
\norm{f(t)}_{\G^{\lambda(t),\sigma}}^2 +  \jap{t}^{4}\norm{P_{\neq 0}\phi}^2_{\G^{\lambda(t),\sigma-3}} + \jap{t}^{4-K_D\epsilon}\norm{[\partial_tv]}_{\G^{\lambda(t),\sigma-6}}^2 + K_v^{-1}\norm{v^\prime - 1}_{\G^{\lambda(t),\sigma}}^2 \lesssim \epsilon^2. \label{ineq:apriori}
\end{align} 
Define $\lambda_\infty = \lim_{t \rightarrow \infty} \lambda(t)$. 
By the method of characteristics and Sobolev embedding, Lemma \ref{lem:LossyElliptic} and \eqref{ineq:apriori} also imply 
the spatial localization $\int \abs{vf(t,z,v)}dz dv \lesssim \epsilon$. 
From this localization and \eqref{eq:tildeu0_simple} we have  
\begin{align}
\norm{\tilde u_0}_{\G^{\lambda(t),\sigma}} \lesssim \epsilon. \label{ineq:tildeu0Ctrl}
\end{align}
In order to deduce the convergence expressed in \eqref{main-omega} and \eqref{ineq:damping} we now undo the change of coordinates in $v$, switching to the more physically natural coordinates $(z,y)$.
 Writing $h(t,z,y) = f(t,z,v) = \omega(t,x,y)$ and $\tilde\psi(t,z,y) = \phi(t,z,v)$ one derives from \eqref{def:2DEuler} as in \S\ref{sec:CoordinateTrans}, 
\begin{align}
\partial_t h + \grad_{z,y}^\perp P_{\neq 0}\tilde\psi \cdot \grad_{z,y}h = 0. \label{ineq:2DhEuler}
\end{align}
In order to quantify the Gevrey regularity of $h(t,z,y) = f(t,z,v(t,y))$ we apply the composition inequality Lemma \ref{lem:CompoIneq} (with \eqref{ineq:RelatScale}) and hence we must show $v(t,y)-y$ is small in a suitable Gevrey class.  
Actually what we have a priori control on from \eqref{ineq:apriori} is $v^\prime(t,v(t,y)) = \partial_yv(t,y)$. 
From the $C^\infty$ inverse function theorem we may solve for $y = y(t,v)$, 
which implies the spatial control $\int \abs{y h(t,z,y)} dz dy \lesssim \epsilon$ (and Remark \ref{rmk:compactsupp}). 
From \eqref{ineq:tildeu0Ctrl} and \eqref{def:zv}, we also get at least the $L^2$ estimates  $\norm{v(t,y)-y}_2 + \norm{y(t,v)-v}_2 \lesssim \epsilon$. 
To get control on the Gevrey regularity, we apply
\begin{align*}
\partial_vy(t,v) = \frac{1}{v^\prime(t,v)} = 1 + \left(\frac{1}{v^\prime(t,v)} - 1\right)
\end{align*}  
together with Lemma \ref{lem:CompoIneq}, (B2) and \eqref{ineq:RelatScale} which shows that for any $\lambda_\infty^\prime \in (\lambda^\prime, \lambda_\infty)$:
\begin{align*} 
\norm{y(t,v) - v}_{\G^{\lambda_\infty^\prime}} \lesssim \epsilon.
\end{align*} 
Therefore, by a Gevrey inverse function theorem (Lemma \ref{lem:IFT}) and \eqref{ineq:RelatScale},  
for any $\lambda_\infty^{\prime\prime} \in (\lambda^\prime, \lambda_\infty^\prime)$,
we may choose $\epsilon$ sufficiently small such that $\norm{v(t,y) - y}_{\G^{\lambda_\infty^{\prime\prime}}} \lesssim \epsilon$. 
Together with Lemma \ref{lem:CompoIneq}, \eqref{ineq:RelatScale} and \eqref{ineq:apriori} this implies for any $\lambda_\infty^{\prime\prime\prime} \in (\lambda^\prime, \lambda_\infty^{\prime\prime})$ (adjusting $\epsilon$ if necessary),   
\begin{align} 
\norm{h(t)}_{\G^{\lambda_\infty^{\prime\prime\prime}}} +\jap{t}^{2}\norm{P_{\neq 0}\tilde \psi(t)}_{\G^{\lambda_\infty^{\prime\prime\prime}}} &\lesssim \epsilon. \label{ineq:hpsitildeest}
\end{align} 
Therefore by integrating \eqref{ineq:2DhEuler}, we define $f_\infty$ by the absolutely convergent integral
\begin{align}
f_\infty = h(1) - \int_1^\infty\grad_{z,y}^\perp P_{\neq 0}\tilde\psi(s) \cdot \grad_{z,y}h (s) ds.
\end{align} 
More precisely, since $\lambda^\prime < \lambda_\infty^{\prime\prime\prime}$, Minkowski, \eqref{ineq:GAlg} and \eqref{ineq:SobExp} imply 
 \begin{align}
\norm{h(t) - f_\infty}_{\G^{\lambda^{\prime}}} & = \norm{\int_t^\infty\grad_{z,y}^\perp P_{\neq 0}\tilde\psi(\tau) \cdot \grad_{z,y}h (\tau) d\tau}_{\G^{\lambda^{\prime}}} \lesssim \frac{\epsilon^2}{\jap{t}}. \label{ineq:hconvg}
\end{align}  
By the definition of $z$ and the phase $\Phi$ \eqref{def:phi}, this is equivalent to \eqref{main-omega}.
Shifting in $z$, \eqref{ineq:xdamping} and \eqref{ineq:ydamping} follow from the decay estimates on $\psi$.

Finally, an argument similar to that used to derive \eqref{ineq:hconvg} can be applied to $U^x_0(t,y) = \tilde u_0(t,v)$, which satisfies the following from \eqref{def:2DEulerMomentum} denoting $\tilde U(t,z,y) = \tilde u(t,z,v) = U^x(t,x,y)$: 
\begin{align}
\partial_t U_0^x + \jap{\grad^\perp_{z,y} P_{\neq 0} \tilde \psi \cdot \grad_{z,y} \tilde U} = 0. \label{def:Umomen}
\end{align}
 By \eqref{ineq:gradulossy} below, it follows that $\norm{P_{\neq 0} \grad_{z,y} \tilde U(t)}_{\G^{\lambda_\infty^{\prime\prime\prime}}} \lesssim \epsilon\jap{t}^{-1}$ by the argument used to deduce \eqref{ineq:hpsitildeest}.  
This, together with \eqref{ineq:hpsitildeest}, implies \eqref{ineq:xdamping_slow}
by integrating \eqref{def:Umomen} as in \eqref{ineq:hconvg} (see \S\ref{apx:LogCorrection} for more discussion). 

\section{Growth mechanism and construction of $A$}\label{sec:Growth}

\subsection{Construction of $w$}
\subsubsection{Formal derivation of toy model} \label{sec:Toy}
From \S\ref{sec:MainEnergy} we see that the basic challenge to the proof of Theorem \ref{thm:Main} is 
controlling the regularity of solutions to \eqref{Euler2}.
Since we must pay regularity to deduce decay on the velocity $u$, it is natural to consider the 
frequency interactions in the product $u\cdot \grad f$ with the frequencies of $u$ much larger than $f$.
This leads us to study a simpler model 
\begin{align} 
\partial_t f = -u\cdot \grad f_{lo}, \label{def:origlin}
\end{align}
where $f_{lo}$ is a given function that we think of as much smoother than $f$. 
As we see from \eqref{Euler2}, $u$ consists of several terms, however let us focus on the term we think should be the worst and also ignore the $v^\prime$, further reducing to the linear problem: 
\begin{align*} 
\partial_tf = \partial_v P_{\neq 0} \phi \partial_z f_{lo}.    
\end{align*} 
Suppose that instead of $f = \Delta_t\phi$, we had $f = \partial_{zz}\phi + (\partial_y - t\partial_z)^2\phi$ as in \eqref{orr-cri}, then on the Fourier side:
\begin{align*} 
\partial_t\hat{f}(t,k,\eta) = \frac{1}{2\pi}\sum_{l \neq 0}\int_\xi \frac{\xi (k-l)}{l^2 + \abs{\xi - lt}^2}\hat{f}(l,\xi) \hat f_{lo}(t,k-l,\eta-\xi) d\xi.    
\end{align*}  
Since $f_{lo}$ weakens interactions between well-separated frequencies, let us consider a discrete model with $\eta$ as a fixed parameter:  
\begin{align} 
\partial_t\hat{f}(t,k,\eta) =  \frac{1}{2\pi}\sum_{l \neq 0} \frac{\eta (k-l)}{l^2 + \abs{\eta - lt}^2}\hat{f}(l,\eta)f_{lo}(t,k-l,0). 
\end{align} 
As time advances this system of ODEs will go through resonances or ``critical times'' given by $t = \frac{\eta}{k}$, at which time the $k$ mode strongly forces the others.   
If $\abs{\eta}k^{-2} \ll 1$ then the critical time does not have a serious detriment and so focus on the case $\abs{\eta}k^{-2} > 1$.  
The scenario we are most concerned with is a high-to-low cascade in which the $k$ mode has a strong effect at time $\eta/k$ that excites the $k-1$ mode which has a strong effect at time $\eta/(k-1)$ that excites the $k-2$ mode and so on. The echoes observed experimentally in \cite{YuDriscoll02,YuDriscollONeil} arise from an effect similar to this \cite{VMW98,Vanneste02}.
Now focus near one critical time $\eta/k$ on a time interval of length roughly $\eta/k^2$ and consider the interaction between the mode $k$ and a nearby mode $l$ with $l \neq k$.
If one takes absolute values and retains only the leading order terms, then this reduces to the much simpler system of two ODEs (thinking of $f_{lo} = O(\kappa)$) which we refer to as the \emph{toy model}: 
\begin{subequations} \label{toy}
\begin{align}
\partial_tf_R & = \kappa\frac{k^2}{\abs{\eta}}f_{NR}, \\
\partial_tf_{NR} & = \kappa\frac{\abs{\eta}}{k^2 + \abs{\eta-kt}^2}f_{R}, 
\end{align}
\end{subequations} 
where we think of $f_R$ as being the evolution of the $k$ mode and $f_{NR}$ being the evolution of a nearby mode $l$ with $l \neq k$.
The factor $k^2/\abs{\eta}$ in the ODE for $f_R$ is an upper bound on the strongest interaction a non-resonant mode, for example the $k-1$ mode, can have with the resonant mode. 
Obviously \eqref{toy} represents a major simplification compared to \eqref{def:origlin}, however it will be sufficient to prove Theorem \ref{thm:Main}. See \S\ref{sec:Conclusion} for a discussion and speculation on whether it is possible to improve Theorem \ref{thm:Main} by using a model closer to \eqref{def:origlin}. 

\begin{rmk} 
The toy model dropped several terms, one of which being a weak self-interaction. 
For example, one could replace the second equation in \eqref{toy} by
$$\partial_\tau f_{NR}   =  \kappa \frac{\abs{\eta}}{k^2 + \abs{\eta-kt}^2} f_R +  \kappa \frac{k^2}{\abs{\eta}}f_{NR}.$$
However, this does not significantly change the worse possible growth predicted by the model and is not necessary for the proof of Theorem \ref{thm:Main}.  
Actually, the proof of Theorem \ref{thm:Main} strongly suggests that the most substantial simplifications in the derivation of \eqref{toy} was the replacement of $\Delta_t$ by $\partial_z^2 + (\partial_v - t\partial_z)^2$ and the omission of $[\partial_t v]$.  
\end{rmk}

\subsubsection{Construction of $w$}
For simplicity of notation in this section we usually take $\eta, k > 0$ but the work applies equally well to $\eta, k < 0$. 
Note that modes where $\eta k < 0$ do not have resonances for positive times.  
Keeping with the intuition from the derivation of \eqref{toy}, in this section we will think of $\eta$ as a fixed parameter and time varying.  
Here we will be concerned with critical times for which we have $\eta/k^2 \geq 1$. 
Accordingly, in this section we will use $I_{k,\eta}$ (see  \S\ref{sec:Notation}) to denote any resonant interval with $\eta/k^2 \geq 1$, in contrast to subsequent sections
where this notation is restricted further. 

A key feature of our methods is how the toy model is used to construct a norm which precisely matches the estimated worst-case behavior that the reaction terms create, done by choosing $w_k(t,\eta)$ to be an approximate solution to \eqref{toy}.
First we have the following (easy to check) Proposition. 

\begin{proposition}\label{grw}
Let $\tau = t - \frac{\eta}{k}$ and consider the solution $(f_R(\tau),f_{NR}(\tau))$ to \eqref{toy} with $f_R\left(-\frac{\eta}{k^2}\right) = f_{NR}\left(-\frac{\eta}{k^2}\right) = 1$. 
There exists a constant $C$ such that for all $\kappa < 1/2$ and $\frac{\eta}{k^2} \geq 1$, 
\begin{align*} 
f_R(\tau) &\leq  C     \left( {\frac{k^2}{\eta} (1 + |\tau|)} \right)^{-C \kappa}  \quad & -\frac{\eta}{k^2} \leq \tau \leq 0, \\
f_{NR}(\tau) &\leq  C    \left( { \frac{k^2}{\eta} (1 + |\tau|)} \right)^{-C \kappa-1}  \quad &  -\frac{\eta}{k^2} \leq \tau \leq 0, \\  
f_R(\tau) &\leq  C  \left(\frac{\eta}{k^2} \right)^{C\kappa}   ( 1 + |\tau|  )^{1+ C \kappa}  \quad &  0\leq \tau \leq \frac{\eta}{k^2} , \\
f_{NR}(\tau) &\leq  C   \left(\frac{\eta}{k^2} \right)^{C\kappa+1}   ( 1 + |\tau|  )^{C \kappa}     \quad &  0\leq  \tau \leq \frac{\eta}{k^2}. 
\end{align*} 
For the remainder of the paper we fix $\kappa$ such that $3/2 < (1 + 2C\kappa) < 10$.  
\end{proposition}

\begin{rmk}
It is important to notice that over the whole interval  $ \left[-\frac{\eta}{k^2}, \frac{\eta}{k^2}  \right] $, both 
$f_R$ and $f_{NR}$ are at most  amplified by roughly the same factor
$C  (\tfrac{\eta}{k^2})^{1+2C\kappa}$. 
Over the interval $[-\frac{\eta}{k^2},0 ] $,  $f_{NR}$  is amplified at most  by 
$C  (\tfrac{\eta}{k^2})^{1+C\kappa}$ and  $f_R$ is  amplified at most  by 
$C  (\tfrac{\eta}{k^2})^{C\kappa}$. Whereas, 
 over the interval $[0, \tfrac{\eta}{k^2}] $,  $f_{NR}$  is  amplified at most  by
$C  (\tfrac{\eta}{k^2})^{C\kappa}$ and $f_R$  is amplified at most  by
$C (\tfrac{\eta}{k^2})^{1+C\kappa}$.   
Near the critical time, the imbalance between $f_{NR}$ and $f_R$ is the largest - in particular, the resonant mode $f_R$ is a factor of $\frac{\eta}{k^2}$ less than $f_{NR}$ at this time. 
However by the end of the interval, the total growth of the resonant and non-resonant modes are comparable. 
The fact that $f_R$ and $f_{NR}$ are amplified 
the same over that interval will simplify the construction of $w$; specifically, we will be able to assign $w_R$ and $w_{NR}$ below to agree at the end points of the critical interval.  
\end{rmk}

On each interval $I_{k,\eta}$, growth of the resonant mode $(k,\eta)$ will 
be modeled by $w_{R}$ and the rest of the modes (which are non-resonant) 
will be modeled by  $w_{NR}$.  
By Proposition \ref{grw}, we will able to choose $w$ such that the total growth of $w_R$ and $w_{NR}$ exactly agree.

 The construction is done backward in time, starting with $k=1$. 
For $t \in I_{k,\eta}$ and $\tau = t - \frac{\eta}{k}$, we will choose
 $(w_{NR},w_R)$   such that over the interval $I_{k,\eta}$ they approximately satisfy \eqref{toy}: 
\begin{equation}  \label{w-grw}
\begin{aligned}  
 \partial_\tau w_R  & \approx  \kappa \frac{k^2}{\eta}w_{NR},  \\ 
 \partial_\tau w_{NR} & \approx \kappa \frac{\eta}{k^2 (1 + \tau^2)} w_{R},  
\end{aligned} 
\end{equation} 
We first construct the non-resonant component $w_{NR}$ and then 
explain how we should modify it over each interval  $I_{k,\eta}$ to 
construct $w_R$.  

Let $w_{NR}$ be a non-decreasing function of time with $w_{NR}(t,\eta) = 1  $ for $t \geq  2\eta $. 
For $ k \geq 1$, we assume that $w_{NR} (t_{k-1,\eta})  $ was computed.  
To compute $w_{NR}$ on the interval $I_{k,\eta} $, we use the growth predicted 
by Proposition \ref{grw}:  for $k=1,2,3,..., E(\sqrt{\eta}) $, we define 
\begin{subequations} \label{def:wNR}
\begin{align} 
w_{NR}(t,\eta) &=   \Big( \frac{k^2}{\eta}   \left[ 1 + b_{k,\eta} |t-\frac{\eta}k | \right]   \Big)^{C\kappa}  w_{NR} (t_{k-1,\eta}),  \quad& 
  \quad  \forall t \in  I^R_{k,\eta} =  \left[ \frac{\eta}k ,t_{k-1,\eta}  \right], \\ 
w_{NR}(t,\eta) &=   \Big(1 + a_{k,\eta} |t-\frac{\eta}k |   \Big)^{-1-C\kappa}  w_{NR} \left(\frac{\eta}k\right),  \quad& 
  \quad \forall  t \in  I^L_{k,\eta} =  \left[ t_{k,\eta}  , \frac{\eta}k   \right].  
\end{align} 
\end{subequations}  
The constant $b_{k,\eta}  $   is chosen to ensure that $ \frac{k^2}{\eta}   \left[ 1 + b_{k,\eta} |t_{k-1,\eta}-\frac{\eta}k | \right]  =1$, hence for $k \geq2$, we have 
\begin{align} \label{bk} 
 b_{k,\eta} = \frac{2(k-1)}{k} \left(1 - \frac{k^2}{\eta} \right)
\end{align} 
and $b_{1,\eta} = 1 - 1/\eta$. 
Similarly,  $a_{k,\eta}$ is chosen to ensure $ \frac{k^2}{\eta}\left[ 1 + a_{k,\eta} |t_{k,\eta}-\frac{\eta}k | \right] = 1$, which implies
\begin{align} \label{ak} 
 a_{k,\eta} = \frac{2(k+1)}{k} \left(1 - \frac{k^2}{\eta} \right). 
\end{align} 
  Hence, $ w_{NR} (\frac{\eta}k) = w_{NR} (t_{k-1,\eta})  \Big( \frac{k^2}{\eta} \Big)^{C\kappa}   $  
and  $ w_{NR} ( t_{k,\eta} ) = w_{NR} (t_{k-1,\eta})  \Big( \frac{k^2}{\eta} \Big)^{1+ 2 C\kappa}   $.  
The choice of $a_{k,\eta}  $  and $b_{k,\eta}  $ was made to ensure that the ratio between $w_{NR} ( t_{k,\eta} ) $ and   $w_{NR} (t_{k-1,\eta})$ is exactly $ \Big( \frac{k^2}{\eta} \Big)^{1+ 2 C\kappa}$.  
Finally, we take $ w_{NR}$ to be constant on the interval $[0, t_{E(\sqrt{\eta}),\eta }] $, namely 
 $ w_{NR}(t,\eta)   = w(t_{E(\sqrt{\eta}),\eta } ,\eta)  $  for $ t \in [0, t_{E(\sqrt{\eta}),\eta }].$ 
Note that we always have $0\leq b_{k,\eta}  < 1$ and 
 $0\leq a_{k,\eta}   < 4$, but that $a_{k,\eta}$ and $b_{k,\eta}$ approach $0$ when $k$ approaches 
$E(\sqrt{\eta})$. 
This will present minor technical difficulties in the sequel since this implies that $\partial_t w$ vanishes near this time and hence a loss of the lower bounds in \eqref{w-grw}. 

On each interval $I_{k,\eta}  $, we define $w_R(t,\eta) $ by 
\begin{subequations} \label{def:wR}
\begin{align} 
w_R(t,\eta) &=   \frac{k^2}{\eta}   \left( 1 + b_{k,\eta} \abs{t-\frac{\eta}k} \right)w_{NR}(t,\eta),   \quad& 
  \quad \forall  t \in  I^R_{k,\eta} =  \left[ \frac{\eta}k ,t_{k-1,\eta}  \right], \\ 
w_R(t,\eta) &=   \frac{k^2}{\eta}   \left( 1 + a_{k,\eta} \abs{t-\frac{\eta}k} \right)   w_{NR}(t,\eta),    \quad& 
  \quad \forall  t \in  I^L_{k,\eta} =  \left[ t_{k,\eta}  , \frac{\eta}k   \right].   
\end{align}
\end{subequations}    
Due to the choice of $ b_{k,\eta}$  and $a_{k,\eta}$, we get that 
$w_R( t_{k,\eta}  ,\eta) =w_{NR}( t_{k,\eta}  ,\eta)      $  and $ w_R(  \frac{\eta}{k}  ,\eta) 
 =\frac{k^2}{\eta} 
 w_{NR}(  \frac{\eta}{k}   ,\eta)   $.

To define the full $w_k(t,\eta)$,  we then have 
\begin{align} 
w_k(t,\eta) =
\left\{
\begin{array}{ll} 
w_k(t_{E(\sqrt{{\eta}}),\eta},\eta) & t < t_{E(\sqrt{\eta}),\eta} \\
w_{NR}(t,\eta) & t \in [t_{E(\sqrt{\eta}),\eta},2\eta]\setminus I_{k,\eta} \\
w_{R}(t,\eta) & t \in I_{k,\eta} \\
1 & t \geq 2\eta. 
\end{array}
\right. \label{def:wk}
\end{align} 
Since $w_R$ and $w_{NR}$ agree at the end-points of $I_{k,\eta}$, $w_k(t,\eta)$ is Lipschitz continuous in time. 
This completes the construction of $w$ which appears in the $J$ defined above \eqref{def:J}.  

We also define $J^R(t,\eta)$ and $A^R(t,\eta)$ to assign resonant regularity at \emph{every} critical time: 
\begin{align} 
J^R(t,\eta) & =
\left\{
\begin{array}{ll} e^{\mu\abs{\eta}^{1/2}} w^{-1}_R(t_{E(\sqrt{{\eta}}),\eta},\eta) & t < t_{E(\sqrt{\eta}),\eta} \\
  e^{\mu\abs{\eta}^{1/2}}w^{-1}_{R}(t,\eta) & t \in [t_{E(\sqrt{\eta}),\eta},2\eta] \\
e^{\mu\abs{\eta}^{1/2}} & t \geq 2\eta,
\end{array}
\right. \nonumber \\ 
A^R(t,k,\eta) & = e^{\lambda(t)\abs{\eta}^s}\jap{\eta}^\sigma J^R(t,\eta). \label{def:AR}
\end{align}
We can easily see from \eqref{def:wR} that $A^R(t,\eta) \geq A_0(t,\eta)$ and since the zero frequency is always non-resonant from \eqref{def:wk}, we see that near the critical times, $A^R$ can be as much as a factor of $\abs{\eta}$ larger. 

\subsubsection{Total growth of $w_k(t,\eta)$}
The following lemma shows that the toy model predicts a growth of high frequencies which amounts to a loss of Gevrey-2 regularity, which is the primary origin of the restriction $s > 1/2$ in Theorem \ref{thm:Main}. 
C. Mouhot and C. Villani have informed the authors that a heuristic similar to that used in Section 7 of \cite{MouhotVillani11} can also be used to predict the same Gevrey-2 regularity requirement, though note that the primary purpose of the toy model \eqref{toy} is to provide precise mode-by-mode information about how this loss can occur.  

\begin{lemma}[Growth  of $w$] \label{basic}
For $\eta > 1$, we have for $\mu = 4(1 + 2C\kappa)$, 
\begin{align}  \label{w-grwth}
\frac1{w_k(0,\eta)}  =  \frac1{w_k( t_{E(\sqrt{\eta}),\eta},\eta)} 
  \sim \frac1{\eta^{\mu/8}}  e^{\frac{\mu}{2} \sqrt{\eta} }. 
\end{align} 
Here $\sim$ is in the sense of asymptotic expansion. 
\end{lemma} 

\begin{proof}
Counting the growth over each interval implied by \eqref{def:wk} gives the exact formula: 
\begin{align*} 
 \frac1{w_k(0,\eta)} & =\left(   \frac{\eta}{N^2}  \right)^c  \left(    \frac{\eta}{(N-1)^2}  \right)^c ...
  \left(    \frac{\eta}{1^2}  \right)^c = \left[\frac{\eta^{N}} { (N!)^2 }\right]^c,   
\end{align*}  
where $c = 1 + 2 C\kappa $. 
Recall Stirling's formula
  $N !   \sim \sqrt{2 \pi  N}   (N/e)^N$, which implies
$$ (w_k(0,\eta)  )^{-1/c} \sim     \frac{\eta^{N}} { ( 2\pi N   ) (N/e)^{2N} } \sim 
   \frac{1} {  2\pi \sqrt{\eta}   } e^{2\sqrt{\eta}}  
     \left[ \frac{\sqrt{\eta}}{N}    e^{2N - 2 \sqrt{\eta}} 
     \Big( \frac{ \eta}{N^2}  \Big)^N   \right]    $$
 and the result follows since the term between $[..]$ is $\approx 1$ by $\abs{N - \sqrt{\eta}} \leq 1$. 
\end{proof}

\subsection{Properties of $w$ and $J$} 
In this section we prove some of the important and useful properties of $J$ and $w$. 
This section is fundamental to our work but at times the proofs are tediously combinatorial. 

The following trichotomy expresses the well-separation of critical times and is used several times in the sequel. 
\begin{lemma} \label{lem:wellsep}
Let $\xi,\eta$ be such that there exists some $\alpha \geq 1$ with $\frac{1}{\alpha}\abs{\xi} \leq \abs{\eta} \leq \alpha\abs{\xi}$ and let $k,n$ be such that $t \in I_{k,\eta}$ and $t \in I_{n,\xi}$  (note that $k \approx n$).  
Then at least one of following holds:
\begin{itemize} 
\item[(a)] $k = n$ (almost same interval); 
\item[(b)] $\abs{t - \frac{\eta}{k}} \geq \frac{1}{10 \alpha}\frac{\abs{\eta}}{k^2}$ and $\abs{t - \frac{\xi}{n}} \geq \frac{1}{10 \alpha}\frac{\abs{\xi}}{n^2}$ (far from resonance);
\item[(c)] $\abs{\eta - \xi} \gtrsim_\alpha \frac{\abs{\eta}}{\abs{n}}$ (well-separated). 
\end{itemize}
\end{lemma} 
\begin{proof} 
To see that $k \approx n$ note 
\begin{align} \label{jsn} 
\frac{\abs{k}}{\abs{n}} = \frac{\abs{tk}}{\abs{tn}} = \frac{\abs{\xi}}{\abs{\eta}}\frac{\abs{tk}}{\abs{\eta}} \frac{\abs{\xi}}{\abs{tn}} \approx_\alpha 1. 
\end{align}
If $k = n$ then there is nothing to prove. 
Suppose now both (a) and (b) are false, which means one of the two inequalities in (b) fails. 
Without loss of generality suppose  $\abs{t - \frac{\xi}{n}} < \frac{1}{10 \alpha}\frac{\abs{\xi}}{n^2}$. 
Then, 
\begin{align*}
\abs{\frac{\eta}{n} - \frac{\xi}{n}} \geq \abs{t - \frac{\eta}{n}} - \abs{t - \frac{\xi}{n}} \geq \frac{\abs{\eta}}{2n(n+1)} -  \frac{1}{10 \alpha}\frac{\abs{\xi}}{n^2} \gtrsim \frac{\abs{\eta}}{n^2}, 
\end{align*}   
where we also used $k \neq n$. This proves (c). 
\end{proof} 

From the definition of $w$, for $t \in I_{k,\eta}$ and $t> 2  \sqrt{\abs{\eta}}$, we have for $\tau = t - \frac\eta{k}$: 
\begin{equation}  \label{w-up-low}
\begin{aligned}  
\partial_\tau w_R &\approx  \kappa \frac{k^2}{\abs{\eta}}w_{NR},  \\ 
\partial_\tau w_{NR} & \approx   \kappa \frac{\abs{\eta}}{k^2 (1 + \tau^2)} w_R.   
\end{aligned}
\end{equation} 
Moreover, we also have the following: 
\begin{lemma}
For $t \in I_{k,\eta}$ and $t> 2  \sqrt{\abs{\eta}}$, we have the following with $\tau = t - \frac\eta{k}$: 
\begin{align}  \label{dtw}
\frac{\partial_t w_{NR}(t,\eta)}{w_{NR}(t,\eta)}  
\approx  \frac{1}{1+\abs{\tau}}  \approx 
  \frac{\partial_t w_R(t,\eta)}{w_R(t,\eta)}. 
\end{align} 
\end{lemma}

The following two lemmas are more substantial and show that although the toy model neglected interactions in $\eta$ and $\xi$,
$w(t,\eta)$ with $w(t,\xi)$ can still be compared effectively.  

\begin{lemma} \label{lem:WtFreqCompare}
\begin{itemize}
\item[(i)] For $t \geq 1$, and $k,l,\eta,\xi$ such that  $ \max(2\sqrt{\abs{\xi}}, \sqrt{\abs{\eta}}) < t < 2  \min( \abs{\xi},  \abs{\eta})    $,
\begin{align} \label{dtw-xi} 
 \frac{\partial_t w_k(t,\eta)}{w_k(t,\eta)}\frac{w_l(t,\xi)}{\partial_t w_l(t,\xi)}
 \lesssim \jap{\eta - \xi} 
\end{align}
\item[(ii)] For all $t \geq 1$, and $k,l,\eta,\xi$, such that for some $\alpha\geq1$, $\frac{1}{\alpha}\abs{\xi} \leq \abs{\eta} \leq \alpha\abs{\xi}$,  
\begin{align}
\sqrt{\frac{\partial_t w_l(t,\xi)}{w_l(t,\xi)}} \lesssim_\alpha \left[\sqrt{\frac{\partial_t w_k(t,\eta)}{w_k(t,\eta)}} + \frac{\abs{\eta}^{s/2}}{\jap{t}^{s}}\right]\jap{\eta-\xi}. \label{ineq:partialtw_endpt}  
\end{align}
\end{itemize}
\end{lemma}
\begin{remark} 
Notice the requirement that $t > 2\sqrt{\abs{\xi}}$ in \eqref{dtw-xi} and $t > 2\sqrt{\abs{\eta}}$ in \eqref{dtw}. This is due to the fact that $\partial_t w(t,\xi) \rightarrow 0$ as $t \searrow E(\sqrt{\abs{\xi}})$, and hence we do not have the lower bounds. The upper bounds still hold.  
\end{remark}

\begin{proof}[Proof of Lemma \ref{lem:WtFreqCompare}]
We first prove \eqref{dtw-xi}.
By \eqref{dtw}, $k$ and $l$ do not play a role and are omitted for the duration of the proof.  
Without loss of generality, we may  restrict to $\eta\xi \geq 0$ and $ \eta/2  < \xi < 2  \eta  $
  as otherwise by \eqref{dtw}
\begin{align*} 
\abs{\frac{\partial_t w(t,\eta)}{w(t,\eta)}\frac{w(t,\xi)}{\partial_t w(t,\xi)}}    
 \lesssim \jap{ \xi}   \lesssim \jap{\eta - \xi}. 
\end{align*}
Let $j$ and $n$ be such that $t \in I_{j,\xi}$ and $t \in I_{n,\eta}$.
Notice that like in \eqref{jsn} above, $j \approx n$. 
From \eqref{dtw}, 
\begin{align*} 
\abs{\frac{\partial_t w(t,\eta)}{w(t,\eta)}\frac{w(t,\xi)}{\partial_t w(t,\xi)}} 
 \lesssim  \frac{1 + \abs{t - \frac{\xi}{j}}}{1 + \abs{t - \frac{\eta}{n}}}. 
\end{align*}    
In the case  $ n = j$, \eqref{dtw-xi} follows from the inequality: for $a,b \geq 0$,
\begin{align} 
\frac{1 + a}{1 + b} \leq 1 + \abs{a-b}. \label{ineq:ab}
\end{align}
In the case $n \neq j$, \eqref{dtw-xi} follows from Lemma \ref{lem:wellsep}.
Indeed, if (b) holds then since $n \approx j$ and $\eta \approx \xi$: 
\begin{align*} 
\abs{\frac{\partial_t w(t,\eta)}{w(t,\eta)}\frac{w(t,\xi)}{\partial_t w(t,\xi)}}  \lesssim     \frac{1 + \abs{ \frac{\xi}{j^2}}}{1 + \abs{\frac{\eta}{n^2}}} \lesssim 1. 
\end{align*} 
If (c) holds then it follows that 
\begin{align*} 
\abs{\frac{\partial_t w(t,\eta)}{w(t,\eta)}\frac{w(t,\xi)}{\partial_t w(t,\xi)}}  \lesssim 1 + \abs{ \frac{\xi}{j^2}} \lesssim \jap{\eta - \xi}, 
\end{align*} 
which finishes the proof of \eqref{dtw-xi}. 

Next we prove \eqref{ineq:partialtw_endpt}. 
First, there is nothing to prove unless $E(\sqrt{\abs{\xi}}) \leq t \leq 2\abs{\xi}$, so assume this is the case. 
If $2\sqrt{\abs{\eta}} < t < 2\abs{\eta}$ then \eqref{ineq:partialtw_endpt} is a consequence of \eqref{dtw-xi}. 
If $t \leq 2\sqrt{\abs{\eta}}$ then \eqref{ineq:partialtw_endpt} follows from 
\begin{align*}
\sqrt{\frac{\partial_t w_l(\xi)}{w_l(\xi)}} \lesssim 1 \lesssim \frac{\abs{\eta}^{s/2}}{\jap{t}^s}. 
\end{align*}
Lastly, consider $t \geq 2\abs{\eta}$.
If $\abs{t-\abs{\xi}} < \frac{1}{2\alpha}\abs{\xi}$ then we have
\begin{align*} 
\abs{\eta - \xi} \geq t - \abs{\eta} + \abs{\xi} - t  > \abs{\eta} - \frac{1}{2\alpha}\abs{\xi} \geq \frac{1}{2\alpha}\abs{\xi} 
\end{align*}
If instead  $\abs{t-\abs{\xi}} \geq \frac{1}{2\alpha}\abs{\xi}$ then by \eqref{dtw}, 
\begin{align*}
\sqrt{\frac{\partial_t w_l(\xi)}{w_l(\xi)}} \lesssim \frac{\alpha^{1/2}}{\sqrt{\abs{\xi}}}. 
\end{align*}
Hence, in both cases it follows that 
 \begin{align*}
\sqrt{\frac{\partial_t w_l(\xi)}{w_l(\xi)}} \lesssim_\alpha \frac{\jap{\eta - \xi}^{1/2}}{\jap{\xi}^{1/2}}. 
\end{align*}
Since $2\abs{\eta} \leq t \leq 2 \abs{\xi} \leq 2\alpha\abs{\eta}$, $t\geq 1$ and $s - 1/2 \leq s/2$, 
\begin{align*} 
\frac{1}{\jap{\xi}^{1/2}} 
\lesssim_\alpha \frac{\abs{\eta}^{s/2}}{\jap{t}^s}, 
\end{align*}
which completes the proof of \eqref{ineq:partialtw_endpt}.
\end{proof}

\begin{lemma} \label{lem:WFreqCompare}
For all $t, \eta, \xi$, we have 
\begin{align} 
\frac{w_{NR}(t,\xi)}{w_{NR}(t,\eta)} \lesssim e^{\mu\abs{\eta - \xi}^{1/2}}. \label{ineq:WFreqCompNonRes}
\end{align} 
\end{lemma} 

\begin{proof}[Proof of Lemma \ref{lem:WFreqCompare}   ]
For the proof of Lemma \ref{lem:WFreqCompare}, we use $w(t,\eta) = w_{NR}(t,\eta)$ and $w(t,\xi) = w_{NR}(t,\xi)$ as there is no possible ambiguity. 
 Switching the roles of $\xi$ and $\eta$, we may assume without loss of generality that $|\xi| \leq |\eta|$
and prove instead of \eqref{ineq:WFreqCompNonRes} that 
\begin{align}
e^{-\mu\abs{\eta - \xi}^{1/2}} \lesssim \frac{w_{NR}(t,\xi)}{w_{NR}(t,\eta)} \lesssim e^{\mu\abs{\eta - \xi}^{1/2}}. \label{ineq:WFreqpf} 
\end{align}
If $ |\xi| < |\eta|/2$, then \eqref{ineq:WFreqpf}  
is clear since by Lemma \ref{basic}, 
$$    e^{-\frac{\mu}{2} \sqrt{\xi}  } \leq  w(t,\xi) \leq 1,     $$ 
 $|\xi| \leq |\eta -\xi| $  and $ |\eta| \leq {2} |\eta -\xi|  $.
Hence, in the sequel we may assume that $\eta,\xi \geq 0$ and  $\eta/2 \leq   \xi \leq  \eta$. 

First, if $ t \geq 2\eta$, then  $ {w(t,\xi)} = {w(t,\eta)} = 1$ so there is nothing to prove. 

If $ t \leq  \min( t_{E(\sqrt{\xi}),\xi} , t_{E(\sqrt{\eta}),\eta})  $, then by Lemma \ref{basic}, \eqref{ineq:WFreqpf} follows by: 
$$   \frac{w(t,\xi)}{w(t,\eta)} =  \frac{w(0,\xi)}{w(0,\eta)}  
   \approx  \left(\frac{\xi}{\eta}\right)^{\mu/8}
  e^{\frac{\mu}{2} (\sqrt{\eta} - \sqrt{\xi})  }  .   $$

If $  2\xi \leq  t \leq 2\eta$, then, since $w(t,\eta)$ is non-decreasing in time and $w(t,\xi)$ is constant for $t \geq 2\xi$, 
$$    1 \leq   \frac{w(t,\xi)}{w(t,\eta)}  \leq    \frac{w(2\xi ,\xi)}{w(2\xi,\eta)}. 
    $$ 

If  $   t_{E(\sqrt{\xi}),\xi } \leq  t \leq t_{E(\sqrt{\eta}),\eta}$ then by similar logic,
$$    \frac{w(0,\xi)}{w(0,\eta)}   \leq   \frac{w(t,\xi)}{w(t,\eta)}  \leq  
  \frac{w(  t_{E(\sqrt{\eta}),\eta} ,\xi)}{w( t_{E(\sqrt{\eta}),\eta}  ,\eta)}, 
    $$  

and if  $   t_{E(\sqrt{\eta}),\eta } \leq  t \leq t_{E(\sqrt{\xi}),\xi}$ (note that this can occur even if $\eta \geq \xi$), then 
$$  \frac{w(  t_{E(\sqrt{\xi}),\xi} ,\xi)}{w( t_{E(\sqrt{\xi}),\xi}  ,\eta)} 
   \leq   \frac{w(t,\xi)}{w(t,\eta)}  \leq   \frac{w(0,\xi)}{w(0,\eta)}  
  . 
    $$  
Hence, \eqref{ineq:WFreqpf} is reduced to the case where 
$   \max( t_{E(\sqrt{\xi}),\xi} , t_{E(\sqrt{\eta}),\eta})  \leq  t  \leq 2\xi $. 
Let $j$ and $n$ be such that   $t \in I_{n,\eta}$ and $t \in I_{j,\xi}$. 
Arguing as in \eqref{jsn}, we see that $n \approx j \leq n$.
We consider three cases. 

\textit{Case $ j=n$:} \\
 Assume first that $t \in I_{n,\eta}^R \cap   I_{n,\xi}^R  $, hence denoting $c = 1 + 2C\kappa$ by the definition \eqref{def:wNR}, 
\begin{align} 
 {w(t,\eta)}  &=\left(   \frac{1^2}{\eta}  \right)^{c}  \left(  
  \frac{2^2}{\eta}  \right)^{c} ...
  \left(    \frac{(n-1)^2}{\eta}  \right)^{c} 
\Big( \frac{n^2}{\eta}   [ 1 + b_{n,\eta} |t-\frac{\eta}n | ]   \Big)^{C\kappa} 
 \label{teta},  \\
  \label{txi}   {w(t,\xi)}  &=\left(   \frac{1^2}{\xi}  \right)^{c}  \left(  
  \frac{2^2}{\xi}  \right)^{c} ...
  \left(    \frac{(n-1)^2}{\xi}  \right)^{c} 
\Big( \frac{n^2}{\xi}   [ 1 + b_{n,\xi} |t-\frac{\xi}n | ]   \Big)^{C\kappa} , \quad 
\hbox{and}
\end{align} 
\begin{align} \label{Rat}   
 \frac{w(t,\xi)}{w(t,\eta)}  = \left( \frac{\eta}{\xi}  \right)^{c(n-1) + C\kappa} 
  \left( \frac{1 + b_{n,\xi} |t-\frac{\xi}n | }{ 1 + b_{n,\eta} |t-\frac{\eta}n | }
  \right)^{C\kappa} = F_1 (F_2)^{C\kappa}. 
 \end{align}  
The first factor on the right-hand side of \eqref{Rat}  satisfies:   
\begin{align} 
  1 \leq F_1 =   \left( \frac{\eta}{\xi}  \right)^{c(n-1) + C\kappa} 
  \lesssim  \left( 1 +  \frac{\eta-\xi}{\xi}  \right)^{c \sqrt{\xi}}  
    \lesssim   e^{c \frac{\eta-\xi}{\sqrt{\xi}}} \leq  e^{c \sqrt{\eta-\xi}} \label{ineq:F1}
 \end{align}  
For the second factor in \eqref{Rat}, from \eqref{ineq:ab}, 
\begin{align*} 
 \max\left(F_2, \frac{1}{F_2}\right) & \leq  1 + \left|  b_{n,\xi} |t-\frac{\xi}n |  - b_{n,\eta} |t-\frac{\eta}n | 
     \right| \\  
& \lesssim  1 +  b_{n,\xi}\abs{\frac{\eta- \xi}n}  + |b_{n,\eta}  - b_{n,\xi}  | \frac{\eta}{n^2}
\\
& \lesssim  1 +  |\eta- \xi|  + \abs{\frac1{\xi}  - \frac1{\eta} } \eta 
 \lesssim 
 \jap{\eta- \xi} . 
 \end{align*}  
   
The case where   $t \in I_{n,\eta}^L \cap   I_{n,\xi}^L  $ can be treated in the same way. 

Assume now that $t \in I_{n,\eta}^L \cap   I_{n,\xi}^R  $, hence ${w(t,\xi)}$ 
is given by \eqref{txi}  and   ${w(t,\eta)}$ by (from \eqref{def:wNR}), 

\begin{align} 
 {w(t,\eta)}  &=\left(   \frac{1^2}{\eta}  \right)^{c}  \left(  
  \frac{2^2}{\eta}  \right)^{c} ...
  \left(    \frac{(n-1)^2}{\eta}  \right)^{c}  
 \left(    \frac{ n^2}{\eta}  \right)^{C\kappa} 
\Big(  [ 1 + a_{n,\eta} |t-\frac{\eta}n | ]   \Big)^{-1-C\kappa} \quad 
\hbox{and}  
 \label{tetaL}
\end{align}  
\begin{align*} 
 \frac{w(t,\xi)}{w(t,\eta)}  = \left( \frac{\eta}{\xi}  \right)^{c(n-1) + C\kappa} 
  \left( {1 + b_{n,\xi} |t-\frac{\xi}n | }
  \right)^{C\kappa}\Big(  [ 1 + a_{n,\eta} |t-\frac{\eta}n | ]   \Big)^{1+C\kappa}.
 \end{align*}  
Using that $ \frac{\xi}n \leq  t \leq \frac{\eta}n $ and that $b_{n,\xi}, a_{n,\eta} <4   $, we get from \eqref{ineq:F1} and \eqref{ineq:SobExp}, 
\begin{align*} 
1 \leq 
 \frac{w(t,\xi)}{w(t,\eta)}  \lesssim 
 e^{c \sqrt{\eta-\xi}} 
  \left( {1 + 4 \abs{\eta - \xi } }
  \right)^{1+ 2C\kappa} \lesssim e^{2c\sqrt{\eta - \xi}}. 
 \end{align*}  

\textit{Case $ j=n-1$:} \\ 
If 
 $t \in I_{n,\eta}^L $ then  $t_{n-1,\xi} \leq   \frac{\eta}n  $. 
 If  $ t \in    I_{n-1,\xi}^R  $, then  $  \frac{\xi}{n-1}  < t_{n-1,\eta}   $. 
In either one of these  cases, we deduce that 
 $   \frac{\xi}{n^2}  \lesssim    \frac{\eta-\xi}{n}     $  
and we conclude in a similar way to \eqref{Rat2c} below. 

Next assume that $t \in I_{n,\eta}^R \cap   I_{n-1,\xi}^L  $, which implies
\begin{align} 
 {w(t,\eta)}  &=\left(   \frac{1^2}{\eta}  \right)^{c}  \left(  
  \frac{2^2}{\eta}  \right)^{c} ...
  \left(    \frac{(n-1)^2}{\eta}  \right)^{c} 
\left( \frac{n^2}{\eta}   [ 1 + b_{n,\eta} |t-\frac{\eta}n | ]   \right)^{C\kappa} 
 \label{teta2},  \\
  {w(t,\xi)}  &=\left(   \frac{1^2}{\xi}  \right)^{c}  \left(  
  \frac{2^2}{\xi}  \right)^{c} ...
  \left(    \frac{(n-1)^2}{\xi}  \right)^{C\kappa} 
\left(   1 + a_{n-1,\xi} |t-\frac{\xi}{n-1}|    \right)^{-1-C\kappa} , \quad 
\hbox{and}
  \label{txi2} 
\end{align} 
\begin{align} \label{Rat2}   
 \frac{w(t,\xi)}{w(t,\eta)}  = \left( \frac{\eta}{\xi}  \right)^{c(n-2) + C\kappa} 
 \left(   \frac{(n-1)^2}{\eta}     [ 1 + a_{n-1,\xi} |t-\frac{\xi}{n-1}| ] 
    \right)^{-1-C\kappa} 
 \Big( \frac{n^2}{\eta}   [ 1 + b_{n,\eta} |t-\frac{\eta}n | ]   \Big)^{-C\kappa}.  
 \end{align}  
The result now follows from Lemma \ref{lem:wellsep}: if (b) holds then \eqref{Rat2} and $\eta \approx \xi$ imply
 \begin{align} \label{Rat2b}   
 \frac{w(t,\xi)}{w(t,\eta)}  \approx  \left( \frac{\eta}{\xi}  \right)^{c(n-2) + C\kappa} 
 \end{align}   
and we conclude as \eqref{ineq:F1}. If Lemma \ref{lem:wellsep} (c) holds then 
 \begin{align} \label{Rat2c}   
  \left( \frac{\eta}{\xi}  \right)^{c(n-2) + C\kappa}  
\lesssim  \frac{w(t,\xi)}{w(t,\eta)}  \lesssim
  \left( \frac{\eta}{\xi}  \right)^{c(n-2) + C\kappa}  {\jap{\eta-\xi}}^{1+ 2 C\kappa}, 
 \end{align} 
and we again apply \eqref{ineq:F1} and \eqref{ineq:SobExp} to deduce \eqref{ineq:WFreqpf}. 

\medskip

\textit{Case $j <  n-1 $: } \\ 
In  this case, it is easy to see that  
 $   \frac{\xi}{n^2}  \lesssim    \frac{\eta-\xi}{n}     $  
and we may conclude in a similar way to \eqref{Rat2c}. 
\end{proof}

A consequence of Lemma \ref{lem:WFreqCompare} is the following, 
which allows to easily exchange $J_k(\eta)$ for $J_l(\xi)$.

\begin{lemma}\label{lem:Jswap} 
In general we have 
\begin{align}
\frac{J_k(\eta)}{J_l(\xi)} \lesssim \frac{\abs{\eta}}{k^2\left( 1+ \abs{t - \frac{\eta}{k}} \right)} e^{9{\mu}\abs{k-l,\eta - \xi}^{1/2}}. \label{ineq:WFreqCompRes}
\end{align}  
If any one of the following holds: ($t \not\in I_{k,\eta}$) or ($k = l$) or ($t \in I_{k,\eta}$, $t \not\in I_{k,\xi}$ and $\frac{1}{\alpha}\abs{\xi} \leq \abs{\eta} \leq \alpha\abs{\xi}$ for some $\alpha \geq 1$) then we have the improved estimate  
\begin{align} 
\frac{J_k(\eta)}{J_l(\xi)} \lesssim e^{10 {\mu}\abs{k-l,\eta - \xi}^{1/2}}. \label{ineq:BasicJswap} 
\end{align} 
Finally if $t \in I_{l,\xi}$, $t \not\in I_{k,\eta}$ and $\frac{1}{\alpha}\abs{\xi} \leq \abs{\eta} \leq \alpha\abs{\xi}$ for some $\alpha> 0$ then 
\begin{align} 
\frac{J_k(\eta)}{J_l(\xi)} \lesssim \frac{l^2\left(1 + \abs{t - \frac{\xi}{l}}\right)}{\abs{\xi}}e^{11{\mu}\abs{k-l,\eta - \xi}^{1/2}}. \label{ineq:WFreqCompNRGain}
\end{align}
\begin{remark} \label{rmk:GainLoss}
The leading factors in \eqref{ineq:WFreqCompRes} and \eqref{ineq:WFreqCompNRGain} both are due to 
ratios of $w_{R}$ and $w_{NR}$.
Moreover, in the case $t \in I_{k,\eta} \cap I_{k,\xi}$, $k \neq l$,  the only case where \eqref{ineq:WFreqCompRes} is needed, we also have $\abs{\eta} \approx \abs{\xi}$ and hence from the definition \eqref{def:wR} and the proof of \eqref{dtw} it follows that 
\begin{align}
\frac{J_k(\eta)}{J_l(\xi)} \lesssim e^{20{\mu}\abs{\eta-\xi}^{1/2}} \frac{w_{NR}(t,\xi)}{w_{R}(t,\xi)}. \label{ineq:WLoss}
\end{align}
A version often used is if $t \in \mathbf I_{k,\eta} \cap \mathbf I_{k,\xi}, k\neq l$, then
by \eqref{ineq:WFreqCompRes}, \eqref{dtw}, Lemma \ref{lem:WtFreqCompare} and \eqref{ineq:SobExp}, 
\begin{align} 
\frac{J_k(\eta)}{J_l(\xi)} \lesssim \frac{\abs{\eta}}{k^2}\sqrt{\frac{\partial_t w_k(t,\eta)}{w_k(t,\eta)}}\sqrt{\frac{\partial_t w_l(t,\xi)}{w_l(t,\xi)}}e^{20\mu\abs{k-l,\eta-\xi}^{1/2}}. \label{ineq:RatJ2partt}
\end{align}
\end{remark} 
\end{lemma} 
\begin{proof} 
For $a,b,c,d > 0$, note the basic inequality: 
\begin{align} 
\frac{a+b}{c+d} \leq  \frac{a}{c} + \frac{b}{d},  \label{ineq:abcd}
\end{align}
which implies 
\begin{align*} 
\frac{J_k(\eta)}{J_l(\xi)} \leq \frac{w_l(t,\xi)}{w_k(t,\eta)}e^{\mu\abs{\eta-\xi}^{1/2}} + e^{\mu\abs{k-l}^{1/2}}. 
\end{align*} 
Hence, if $t \in I_{k,\eta}$ then \eqref{ineq:WFreqCompRes} follows from \eqref{ineq:WFreqCompNonRes}, the definition of $w_R$, \eqref{def:wR} and the definition of $w_k(t,\eta)$, \eqref{def:wk}. 
If $t \not\in I_{k,\eta}$ then \eqref{ineq:BasicJswap} follows from \eqref{ineq:WFreqCompNonRes}. 
Now consider the remaining cases. 

\textit{Proof of \eqref{ineq:BasicJswap} when $t \in I_{k,\eta}$, $t \not\in I_{k,\xi}$ and $\frac{1}{\alpha}\abs{\xi} \leq \abs{\eta} \leq \alpha\abs{\xi}$ for some $\alpha \geq 1$:}\\
In this case, \eqref{ineq:BasicJswap} follows from \eqref{ineq:WFreqCompRes} together with Lemma \ref{lem:wellsep} (b) or (c) and \eqref{ineq:SobExp}. Indeed, (a) is ruled out by $t \in I_{k,\eta}$, $t \not\in I_{k,\xi}$; if (b) holds then there is no loss and if (c) holds then \eqref{ineq:SobExp} can be used to absorb the loss. 

\textit{Proof of \eqref{ineq:BasicJswap} when $t \in I_{k,\eta}, t \in I_{k,\xi}$ and $k = l$:}\\
Inequality \eqref{ineq:BasicJswap} follows from \eqref{ineq:WFreqCompNonRes} and the definition of $w_k(t,\eta)$, \eqref{def:wk}. 

\textit{Proof of \eqref{ineq:WFreqCompNRGain}:} \\ 
First consider the case that $t \in I_{l,\eta}$, which also implies $k \neq l$ and $\eta \approx \xi$. 
By \eqref{def:wR},\eqref{def:wk} and \eqref{ineq:WFreqCompNonRes},  
\begin{align*} 
\frac{J_k(\eta)}{J_l(\xi)} & \leq \frac{w_l(t,\xi)}{w_k(t,\eta)}e^{\mu\abs{\eta-\xi}^{1/2}} + w_l(t,\xi)e^{\mu\abs{k}^{1/2} - \mu\abs{\xi}^{1/2}} \\
& \lesssim \frac{w_{NR}(t,\xi)}{w_{R}(t,\xi)}e^{10\mu\abs{\eta-\xi}^{1/2}} + e^{\mu\abs{k}^{1/2} - \mu\abs{\xi}^{1/2}}.  
\end{align*} 
Since $\abs{k}^2 \leq \frac{1}{4}\abs{\eta}$ then \eqref{ineq:IncExp} and \eqref{ineq:SobExp} deals with the second term and \eqref{ineq:WFreqCompNRGain} follows.   

In the case $t \not\in I_{l,\eta}$, \eqref{ineq:WFreqCompNRGain} follows from Lemma \ref{lem:wellsep}. 
Indeed, if (b) holds then 
\begin{align*} 
1 \lesssim_\delta  \frac{w_R(t,\xi)}{w_{NR}(t,\xi)}, 
\end{align*}
and \eqref{ineq:WFreqCompNRGain} follows from \eqref{ineq:BasicJswap} whereas if (c) holds then \eqref{ineq:WFreqCompNRGain} follows from \eqref{ineq:BasicJswap} and \eqref{ineq:SobExp}.  
\end{proof} 

The following variant of the previous lemmas is used in \S\ref{sec:Transport} to recover $1/2$ derivatives.
\begin{lemma} \label{lem:JTrans}
Let $t \leq \frac{1}{2}  \min(\sqrt{\abs{\eta}},\sqrt{\abs{\xi}})$. 
Then,  
\begin{align}  \label{JJ1}
\abs{\frac{J_k(\eta)}{J_l(\xi)} - 1} \lesssim 
\frac{\jap{\eta-\xi,k-l}  }{ \sqrt{ \abs{\xi} + \abs{\eta} +\abs{k} + \abs{l} }  } 
 e^{11\mu\abs{k-l,\eta-\xi}^{1/2}}.  
\end{align}

\end{lemma}

\begin{proof} 
Due to the assumption on $t$, we have that $ J_k(t, \eta) = J_k(0, \eta)$ and  $ J_l(t, \xi) = J_l(0, \xi)$. 
If $ \abs{\xi}^{1/2} + \abs{\eta}^{1/2} +\abs{k}^{1/2} + \abs{l}^{1/2}  \lesssim
     |\xi-\eta| + |k - l|    $, then \eqref{JJ1} follows from \eqref{ineq:BasicJswap}. 
Similarly, we may restrict to $ \abs{\xi}^{1/2} + \abs{\eta}^{1/2} +\abs{k}^{1/2} + \abs{l}^{1/2} \gtrsim 1$, 
as otherwise \eqref{JJ1} is weaker than \eqref{ineq:BasicJswap}.

From now on, we assume that 
\begin{align} \label{s100} 
 |\xi-\eta| + |k - l|    \leq \frac1{100} 
( \abs{\xi}^{1/2} + \abs{\eta}^{1/2} +\abs{k}^{1/2} + \abs{l}^{1/2} ) .  
 \end{align} 
\textit{Case  1: $  \frac1{10} (\abs{k}  + \abs{l} )  \leq \abs{\xi}  + \abs{\eta}  \leq 10 (\abs{k}  + \abs{l} )$:}\\
In this case, recalling the definition \eqref{def:tildeJB}
\begin{align} \label{JJkl} 
\abs{\frac{J_k(\eta)}{J_l(\xi)} - 1} \leq \frac{ \abs{\tilde J_k(\eta) 
- \tilde J_l(\xi)       } } { \tilde J_l(\xi)  +e^{ \mu |l|^{1/2}}  }   + \frac{ \abs{ 
 e^{ \mu |k|^{1/2}} -  e^{ \mu |l|^{1/2}}   } } { e^{ \mu |l|^{1/2}} +   \tilde J_l(\xi)    } .   
\end{align} 
The first term on the right-hand side of \eqref{JJkl} is controlled by 
\begin{align*}  
 \frac{ \abs{\tilde J_k(\eta) 
- \tilde J_l(\xi)       } } { \tilde J_l(\xi)}   \leq   
 \frac{w(0,\xi )}{w(0,\eta) }  \abs{ e^{ \mu  (|\eta|^{1/2}  -|\xi|^{1/2} )  }   -1}
 +  \abs{ \frac{w(0,\xi )}{w(0,\eta) }  -1}.  
\end{align*} 
To control the first term, we use  \eqref{ineq:WFreqCompNonRes} and $\abs{e^{x} - 1} \leq xe^x$, 
\begin{align*}  
 \abs{e^{ \mu  (|\eta|^{1/2}  -|\xi|^{1/2} )  }   -1  }  =  
\abs{ e^{ \mu   \frac { |\eta  -\xi| } {  \abs{\xi}^{1/2} + \abs{\eta}^{1/2}    }  }   -1 } \leq 
  \mu   \frac { |\eta  - \xi| }{  \abs{\xi}^{1/2} + \abs{\eta}^{1/2} }e^{\mu\abs{\eta - \xi}^{1/2}}   
 \lesssim   \frac{\jap{\eta-\xi,k-l}  }{ \sqrt{ \abs{\xi} + \abs{\eta} +\abs{k} + \abs{l} }  }e^{\mu\abs{\eta - \xi}^{1/2}}   .     
\end{align*}    
To control the second term, we  notice that the condition \eqref{s100} (together with our assumption $\abs{k} + \abs{l} \approx \abs{\eta} + \abs{\xi}$) implies that $  | E(\sqrt{\abs{\eta}})  -  E(\sqrt{\abs{\xi}})  | \leq 1    $.   
 We first  look at  the case 
  $ E(\sqrt{\abs{\eta}})  = E(\sqrt{\abs{\xi}})  $. 
Using the inequality: for all $ a, b$, $|a| < 1$ and $b > 1$,
$$  (1 + \frac{a}{b^2})^b - 1 \leq  e\frac{|a|}{b},     $$ 
we have (denoting $c = 1 + 2C\kappa$), 
\begin{align*}  
 \abs{   \frac{w(0,\xi )}{w(0,\eta) }  -1 }  =  
 \abs{  \left( \frac{ \abs{\eta}}{\abs{\xi}}  \right)^{c E(\sqrt{\abs{\eta}}) }  - 1 } 
\lesssim \frac{|\eta-\xi| } { \sqrt{\abs{\xi}}}. 
\end{align*} 
If $ E(\sqrt{\abs{\eta}})  = E(\sqrt{\abs{\xi}}) + 1  $, then $\sqrt{\abs{\xi}} 
<   E(\sqrt{\abs{\eta}}) \leq \sqrt{\abs{\eta}}  $ and 
\begin{align*}  
 \abs{   \frac{w(0,\xi )}{w(0,\eta) }  -1 }  =  
 \abs{  \left( \frac{ \abs{\eta}}{\abs{\xi}}  \right)^{c E(\sqrt{\abs{\xi}}) } 
\left(   \frac{|\eta|}{ E(\sqrt{\abs{\eta}})^2   }  \right)^{c}     - 1 } 
\lesssim \frac{|\eta-\xi| } { \sqrt{\abs{\xi}}}  +   \abs{ 
\left(   \frac{|\eta|}{ E(\sqrt{\abs{\eta}})^2   }  \right)^{c}     - 1 } 
\end{align*} 
and we conclude since 
\begin{align*}  
  \abs{ 
\left(   \frac{|\eta|}{ E(\sqrt{\abs{\eta}})^2   }  \right)^{c}     - 1 }  
\lesssim  \frac{\jap{\eta-\xi} } { {\abs{\xi}}} . 
\end{align*} 
 The case  $ E(\sqrt{\abs{\eta}})  = E(\sqrt{\abs{\xi}}) - 1  $ is treated in 
the same way. 

The second term  on the right-hand side of \eqref{JJkl} is controlled by \eqref{ineq:TrivDiff} and $\abs{e^x - 1} \leq xe^x$,
\begin{align*}   
 \abs{e^{ \mu  (|k|^{1/2}  -|l|^{1/2} )  }   -1} \lesssim  
  \mu   \frac { |k  - l| }{  \abs{k}^{1/2} + \abs{l}^{1/2} } e^{ \mu  \abs{k-l}^{1/2}  }      
 \lesssim   \frac{\jap{\eta-\xi,k-l}  }{ \sqrt{ \abs{\xi} + \abs{\eta} +\abs{k} + \abs{l} }  }e^{ \mu  \abs{k-l}^{1/2}  }.     
\end{align*} 

 \textit{Case  2: $\abs{\xi}  + \abs{\eta} \geq 10 (\abs{k}  + \abs{l} )$:} Here we can treat the 
  first  term on the right-hand side of \eqref{JJkl} as above 
 and use $\abs{\xi}  \geq 4 (\abs{k}  + \abs{l}  )   $ together with \eqref{ineq:SobExp} to treat the second term.

\textit{Case 3:$   \abs{k}  + \abs{l}  \geq 10 (\abs{\xi}  + \abs{\eta} )   $:} Here we can treat the 
  second   term on the right-hand side of \eqref{JJkl} as above 
 and use $\abs{l}  \geq 4 (\abs{\xi}  + \abs{\eta}  )   $  together with \eqref{ineq:SobExp} to treat the first term.
\end{proof}

\subsection{Product lemma and other basic properties of $A$}
Unlike the $\mathcal{G}^{\lambda,\sigma}$ norm (see \S\ref{apx:Gev}), the norm defined by $A$ is \emph{not} an algebra due to the discrepancy between resonant and non-resonant modes 
 which is as large as an entire derivative near the critical times.
 However, $A$ does define an algebra when restricted to the zero mode, as the zero mode is never resonant. 
Although $A$ defines an algebra on the zero modes, the multipliers that appear in the CK terms do not, hence more generally, we have the following product lemma.  
\begin{lemma}[Product lemma] \label{lem:ProdAlg}
For some $c \in (0,1)$, all $\sigma > 1$, all $\beta > -\sigma + 1$ and $\alpha \geq 0$, the following inequalities hold for all $f,g$ which depend only on $v$, 
\begin{subequations} \label{ineq:AProdProps}
\begin{align} 
\norm{\abs{\partial_v}^{\alpha}\jap{\partial_v}^\beta A(fg)}_2 & \lesssim \norm{f}_{\G^{c\lambda,\sigma}} \norm{\abs{\partial_v}^{\alpha}\jap{\partial_v}^\beta Ag}_2 + \norm{g}_{\G^{c\lambda,\sigma}} \norm{\abs{\partial_v}^{\alpha}\jap{\partial_v}^\beta Af}_2 \label{ineq:AProd} \\
\norm{\sqrt{\frac{\partial_tw}{w}}\jap{\partial_v}^\beta A(fg)}_2 & \lesssim \norm{g}_{\G^{c\lambda,\sigma}}\norm{\left(\sqrt{\frac{\partial_tw}{w}} + \frac{\abs{\partial_v}^{s/2}}{\jap{t}^s} \right)\jap{\partial_v}^\beta Af}_2 \nonumber \\ & \quad + \norm{f}_{\G^{c\lambda,\sigma}}\norm{\left(\sqrt{\frac{\partial_tw}{w}} + \frac{\abs{\partial_v}^{s/2}}{\jap{t}^s} \right)\jap{\partial_v}^\beta Ag}_2.   \label{ineq:wtAProd}
\end{align}
\end{subequations}
We also have for $\beta > -\sigma+1$ the algebra property,
\begin{align} 
\norm{\jap{\partial_v}^\beta A(fg)}_2 & \lesssim \norm{\jap{\partial_v}^\beta Af}_2 \norm{\jap{\partial_v}^\beta Ag}_2. \label{ineq:Aalg}
\end{align} 
Moreover, \eqref{ineq:AProdProps} and \eqref{ineq:Aalg} both hold for $A$ replaced by $A^R$. 
\end{lemma} 
\begin{remark} 
Writing $(v^\prime)^2 - 1 = (v^\prime - 1)^2 + 2(v^\prime - 1)$ and $v^{\prime\prime} = \partial_v(v^\prime - 1) + (v^\prime - 1)\partial_v(v^\prime - 1)$ (recall \eqref{def:vpp}), by the bootstrap hypotheses on $v^\prime - 1$ combined with \eqref{ineq:Aalg} we have,
\begin{subequations} \label{ineq:coefbds}
\begin{align}
\norm{A^R\left(1 - (v^\prime)^2\right)}_2 \lesssim \norm{A^R\left(1 - v^\prime\right)}_2 + \norm{A^R\left(1 - v^\prime\right)}_2^2 & \lesssim \epsilon \label{ineq:vp2m1bd} \\ 
\norm{\frac{A^R}{\jap{\partial_v}} v^{\prime\prime}}_2 = \norm{\frac{A^R}{\jap{\partial_v}} \left(v^\prime \partial_v v^\prime \right)}_2 \lesssim \norm{A^R\left(1 - v^\prime\right)}_2 +  \norm{A^R\left(1 - v^\prime\right)}_2^2 & \lesssim \epsilon. 
\end{align}
\end{subequations}
\end{remark}
\begin{proof}[Proof of Lemma \ref{lem:ProdAlg}] 
The proof of \eqref{ineq:AProd} follows from Lemmas \ref{lem:WFreqCompare} and \ref{lem:WtFreqCompare} combined with a paraproduct decomposition; 
the argument is similar to many used in the sequel so is omitted.  

Let us now focus on \eqref{ineq:wtAProd} which is more intricate. 
Let us also just treat the case $\beta = 0$; actually any $\beta > 1-\sigma$ can be treated by adjusting $c$ accordingly and using \eqref{ineq:SobExp}. 
Moreover, we focus on $t \geq 1$; the case $t \in (0,1)$ is easier and is not actually necessary for the proof of Theorem \ref{thm:Main}. 
Write 
\begin{align*}
\mathcal{M}(t,\xi) = \sqrt{\frac{\partial_t w(t,\xi)}{w(t,\xi)}}A_0(t,\xi),  
\end{align*}
and decompose $fg$ with a paraproduct (see \S\ref{Apx:LPProduct}):
\begin{align*} 
fg = T_f g + T_g f + \mathcal{R}(f,g). 
\end{align*} 
Consider the first term on the Fourier side: 
\begin{align*} 
\widehat{\mathcal{M} T_fg}(\xi)  = \frac{1}{2\pi}\sum_{M \geq 8}\mathcal{M}(t,\xi) \int_{\xi^\prime} \hat g(\xi^\prime)_{M}\hat f(\xi-\xi^\prime)_{<M/8} d\xi^\prime. 
\end{align*}  
The goal is to use \eqref{ineq:partialtw_endpt} to pass $\mathcal M$ onto $g$. 
On the support of the integrand (see \S\ref{Apx:LPProduct}), 
\begin{align}
\abs{\abs{\xi} - \abs{\xi^\prime}} \leq \abs{\xi - \xi^\prime} \leq 3/2 \left(\frac{M}{16}\right) \leq 6\abs{\xi^\prime}/32, \label{ineq:ProdRuleFreqLoc}
\end{align} 
and hence $26\abs{\xi^\prime}/32 \leq \abs{\xi} \leq 38\abs{\xi^\prime}/32$ and $26\abs{\xi^\prime}/32 \leq \abs{\xi} \leq 38\abs{\xi^\prime}/32$.
Therefore \eqref{lem:scon} implies for some $c^\prime \in (0,1)$ that $\abs{\xi}^s \leq \abs{\xi^\prime}^s + c^\prime\abs{\xi - \xi^\prime}^s$, and hence by \eqref{ineq:ProdRuleFreqLoc}, 
\begin{align*}
\abs{\widehat{\mathcal{M} T_fg}(\xi)} \lesssim \sum_{M \geq 8} \int_{\xi^\prime} \sqrt{\frac{\partial_t w(t,\xi)}{w(t,\xi)}}\jap{\xi^\prime}^{\sigma}e^{\lambda\abs{\xi^\prime}^s} J_0(\xi) \abs{\hat g(\xi^\prime)_{M}} e^{c^\prime\lambda\abs{\xi-\xi^\prime}^s}\abs{\hat f(\xi-\xi^\prime)_{<M/8}} d\xi^\prime.
\end{align*} 
By \eqref{ineq:BasicJswap} it follows 
\begin{align*} 
\abs{\widehat{\mathcal{M} T_fg}(\xi)} \lesssim \sum_{M \geq 8} \int_{\xi^\prime} \sqrt{\frac{\partial_t w(t,\xi)}{w(t,\xi)}}
A_0(\xi^\prime)\abs{\hat g(\xi^\prime)_{M}} e^{10\mu\abs{\eta-\xi}^{1/2} + c^\prime\lambda\abs{\xi-\xi^\prime}^s}\abs{\hat f(\xi-\xi^\prime)_{<M/8}} d\xi^\prime.
\end{align*} 
Then by \eqref{ineq:partialtw_endpt}, \eqref{ineq:IncExp} and \eqref{ineq:SobExp}, for any $c \in (c^\prime,1)$,  
\begin{align*} 
\abs{\widehat{\mathcal{M} T_fg}(\xi)} \lesssim \sum_{M \geq 8} \int_{\xi^\prime} \left[ \sqrt{\frac{\partial_t w(t,\xi^\prime)}{w(t,\xi^\prime)}} + \frac{\abs{\xi^\prime}^{s/2}}{\jap{t}^s}\right]
A_0(\xi^\prime)\abs{\hat g(\xi^\prime)_{M}}e^{c\lambda\abs{\xi-\xi^\prime}^s} \abs{\hat f(\xi-\xi^\prime)_{<M/8}} d\xi^\prime.
\end{align*} 
Since $\abs{\xi} \approx M$ by \eqref{ineq:ProdRuleFreqLoc}, \eqref{ineq:GeneralOrtho} and \eqref{ineq:L2L1} imply 
\begin{align*} 
\norm{\mathcal{M} T_fg}_2^2 & \lesssim \sum_{M \geq 8} \norm{\left(\sqrt{\frac{\partial_tw}{w}} + \frac{\abs{\partial_v}^{s/2}}{\jap{t}^s} \right) Ag_M}_{2}^2\norm{f_{<M/8}}_{\G^{c\lambda,\sigma}}^2 \\ 
   & \lesssim \norm{\left(\sqrt{\frac{\partial_tw}{w}} + \frac{\abs{\partial_v}^{s/2}}{\jap{t}^s} \right)Ag}_{2}^2\norm{f}_{\G^{c\lambda,\sigma}}^2. 
\end{align*} 
The contribution from $\mathcal{M}T_g f$ is analogous and yields the other term in \eqref{ineq:wtAProd}. 
Let us now turn to the remainder, 
\begin{align*} 
\widehat{\mathcal{M}\mathcal{R}(f,g)}(\xi) = \frac{1}{2\pi} \sum_{M \in \mathbb D} \sum_{M/8 \leq M^\prime \leq 8M} \int_{\xi^\prime}\mathcal{M}(t,\xi)g(\xi^\prime)_{M}f(\xi-\xi^\prime)_{M^\prime} d\xi^\prime.   
\end{align*} 
On the support of the integrand (see \S\ref{Apx:LPProduct}) $\frac{1}{24}\abs{\xi - \xi^\prime}  \leq \abs{\xi^\prime} \leq 24\abs{\xi-\xi^\prime}$,
hence \eqref{lem:strivial} implies for some $c^\prime$, $\abs{\xi}^{s} \leq c^\prime\abs{\xi^\prime}^s + c^\prime\abs{\xi - \xi^\prime}^s$, which gives 
\begin{align*}
\abs{\mathcal{M}\mathcal{R}(f,g)(\xi)} \lesssim \sum_{M \in \mathbb D} \sum_{M/8 \leq M^\prime \leq 8M} \int_{\xi^\prime}
 \sqrt{\frac{\partial_t w(t,\xi)}{w(t,\xi)}}\jap{\xi}^{\sigma}e^{c^\prime\lambda\abs{\xi^\prime}^s} J_0(\xi) \abs{g(\xi^\prime)_{M}} e^{c^\prime\lambda\abs{\xi-\xi^\prime}^s}\abs{f(\xi-\xi^\prime)_{M^\prime}} d\xi^\prime.
\end{align*}
Since $1 \leq t \leq 2\abs{\xi}$ on the support of the integrand, by Lemma \ref{basic}, \eqref{ineq:IncExp} and \eqref{ineq:SobExp}, 
\begin{align*}  
\abs{\mathcal{M}\mathcal{R}(f,g)(\xi)} & \lesssim \sum_{M \in \mathbb D} \sum_{M/8 \leq M^\prime \leq 8M} \int_{\xi^\prime} \frac{\abs{\xi}^{s/2}}{\jap{t}^s}\jap{\xi}^{\sigma + s/2}e^{c^\prime\lambda\abs{\xi^\prime}^s} J_0(\xi) \abs{g(\xi^\prime)_{M}} e^{c^\prime\lambda\abs{\xi-\xi^\prime}^s}\abs{f(\xi-\xi^\prime)_{M^\prime}} d\xi^\prime \\ 
& \lesssim \sum_{M \in \mathbb D} \sum_{M/8 \leq M^\prime \leq 8M} \int_{\xi^\prime} \left(\frac{\abs{\xi^\prime}^{s/2}}{\jap{t}^s} + \frac{\abs{\xi-\xi^\prime}^{s/2}}{\jap{t}^s} \right)e^{c\lambda\abs{\xi^\prime}^s}\frac{1}{\jap{\xi^\prime}}\abs{g(\xi^\prime)_{M}} e^{c\lambda\abs{\xi-\xi^\prime}^s}\abs{f(\xi-\xi^\prime)_{M^\prime}} d\xi^\prime, 
\end{align*}
for any $c \in (c^\prime,1)$. 
Therefore \eqref{ineq:L2L1} and implies 
\begin{align*} 
\norm{\mathcal{M}\mathcal{R}(f,g)(\xi)}_2 & \lesssim \sum_{M \in \mathbb D}\sum_{M/8 \leq M^\prime \leq 8M}\norm{\frac{\abs{\partial_v}^{s/2}}{\jap{t}^s}g_M}_{\G^{c\lambda,\sigma-1}}\norm{f_{M^\prime}}_{\G^{c\lambda}} + \norm{g_M}_{\G^{c\lambda,\sigma-1}}\norm{\frac{\abs{\partial_v}^{s/2}}{\jap{t}^s}f_{M^\prime}}_{\G^{c\lambda}} \\ 
& \lesssim \left(\sum_{M \in \mathbb D}\norm{\frac{\abs{\partial_v}^{s/2}}{\jap{t}^s}g_M}^2_{\G^{c\lambda,\sigma}}\right)^{1/2}\norm{f}_{\G^{c\lambda}} + \left(\sum_{M \in \mathbb D}\norm{g_M}^2_{\G^{c\lambda,\sigma}}\right)^{1/2}\norm{\frac{\abs{\partial_v}^{s/2}}{\jap{t}^s}f}_{\G^{c\lambda}}, 
\end{align*} 
which by \eqref{ineq:GeneralOrtho}, proves \eqref{ineq:wtAProd}. 

The proof in the case with $A$ replaced by $A^R$ proceeds the same. Indeed, from Lemma \ref{lem:Jswap} we see that it is not a matter of $w_{NR}$ vs $w_R$, it is only a matter of having either one or other, but not both. 
\end{proof} 

\section{Elliptic Estimates} \label{sec:Elliptic}
The purpose of this section is to provide a thorough analysis of $\Delta_t$. In particular, in this section we prove Proposition \ref{lem:PrecisionEstimateNeq0}.  

\subsection{Lossy estimate}
The following is the fundamental estimate on $\phi$ which allows to trade the regularity of $f$ in a high norm for decay of the streamfunction in a slightly lower norm. 
This estimate is clear for the elliptic operator that arises from the linearized problem, which we denote by 
\begin{align}
\Delta_L = \partial_{z}^2 + (\partial_v - t\partial_z)^2. \label{def:DeltaL}
\end{align} 
As can be easily seen from examining $\Delta_L^{-1}$ (e.g. \eqref{orr-cri}), we cannot expect to gain $O(t^{-2})$ decay without paying two derivatives. 
Notice that since the coefficient $v^{\prime\prime}$ effectively contains a derivative on $f$ (see \eqref{def:vPDE}), the estimate below loses \emph{three} derivatives. 
This loss can be treated with more precision, which is necessary in \S\ref{sec:Precf}. 

Due to the `lossy' nature of the lemma, this can only be used when $\phi$ is being measured in a low norm, however, this occurs in many places in the proof, most notably the treatment of transport in \S\ref{sec:Transport}, the treatment of $[\partial_t v]$ and $v^\prime \partial_v [\partial_t v]$ in Proposition \ref{prop:CoefControl_Real} and even the proof of the more precise elliptic estimate in \S\ref{sec:Precf}. 

\begin{lemma}[Lossy elliptic estimate] \label{lem:LossyElliptic}
Under the bootstrap hypotheses, for $\epsilon$ sufficiently small, 
\begin{align} 
\norm{P_{\neq 0}\phi(t)}_{\G^{\lambda(t), \sigma - 3}} \lesssim \frac{\norm{f(t)}_{\G^{\lambda(t),\sigma-1}}}{1+t^2}. \label{ineq:lossyelliptic}
\end{align} 
\end{lemma}
\begin{proof}
Omitting the time-dependence in $\lambda$, $\phi$ and $f$, first note
\begin{align} 
\norm{P_{\neq 0}\phi}^2_{\G^{\lambda,\sigma-3}} & = \sum_{k \neq 0}\int_\eta e^{2\lambda \abs{k,\eta}^s}\jap{k,\eta}^{2\sigma-6} \abs{\hat{\phi}(k,\eta)}^2 d\eta \nonumber \\ 
& = \sum_{k \neq 0}\int_\eta e^{2\lambda \abs{(k,\eta)}^s}\frac{\jap{k,\eta}^{2\sigma-2}}{\jap{k,\eta}^4(k^2 + \abs{\eta - kt}^2)^2} (k^2 + \abs{\eta - kt}^2)^2 \abs{\hat{\phi}(k,\eta)}^2 d\eta \nonumber \\ 
& \lesssim \frac{1}{\jap{t}^4}\norm{\Delta_L P_{\neq 0}\phi}^2_{\G^{\lambda,\sigma-1}}. \label{ineq:DeltaLphi}
\end{align}
We write $\Delta_t$ as a perturbation of $\Delta_L$ via (recall the definitions \eqref{def:Deltat},\eqref{def:DeltaL}),
\begin{align*}  
\Delta_LP_{\neq 0}\phi = P_{\neq 0}f + (1 - (v^{\prime})^2)(\partial_y - t\partial_z)^2P_{\neq 0}\phi - v^{\prime\prime}(\partial_y - t\partial_z)P_{\neq 0}\phi. 
\end{align*} 
By the algebra inequality \eqref{ineq:GAlg}, 
\begin{align*} 
\norm{\Delta_L P_{\neq 0}\phi}_{\G^{\lambda,\sigma-1}} \lesssim \norm{f}_{\G^{\lambda,\sigma-1}} + \norm{1 - (v^\prime)^2}_{\G^{\lambda,\sigma-1}}\norm{\Delta_L P_{\neq 0}\phi}_{\G^{\lambda,\sigma-1}} + \norm{v^{\prime\prime}}_{\G^{\lambda,\sigma-1}} \norm{(\partial_y - t\partial_z)P_{\neq 0}\phi}_{\G^{\lambda,\sigma-1}}. 
\end{align*}
Therefore, \eqref{ineq:coefbds} implies, 
\begin{align*} 
\norm{\Delta_L P_{\neq 0}\phi}_{\G^{\lambda,\sigma-1}} \lesssim \norm{f}_{\G^{\lambda,\sigma-1}} + \epsilon\norm{\Delta_L P_{\neq 0}\phi}_{\G^{\lambda,\sigma-1}}. 
\end{align*} 
 Together with \eqref{ineq:DeltaLphi}, this implies the a priori estimate \eqref{ineq:lossyelliptic} provided that $\epsilon$ is sufficiently small. 
\end{proof}

\subsection{Precision elliptic control} \label{sec:Precf}
Now we turn to the proof of Proposition \ref{lem:PrecisionEstimateNeq0}, announced in \S\ref{sec:MainEnergy}.
This is the main elliptic estimate which underlies the treatment of reaction in \S\ref{sec:RigReac} and the key estimates of \S\ref{sec:CoordControls}.
If $\Delta_t$ were simply $\Delta_L$ then the estimate would be trivial. 
However, the coefficients depend on the vorticity, which both couples all of the frequencies in the $v$ direction together and introduces the potential for losing regularity (it is key that the coefficients only depend on $v$).
The simplest effect one can see is the appearance of the CK multipliers on the coefficients collected in \eqref{def:QL}, which occur when `derivatives' taken on the LHS of \eqref{ineq:PrecisionPhiNeq0} land on the coefficients of $\Delta_t$. 
Notice that these `CCK' terms contain the more dangerous \emph{resonant} regularity (see \eqref{def:AR}).
This effect is controlled in Proposition \ref{prop:CoefControl_Real}. 
The other effect one sees is the $\jap{\partial_v(\partial_z t)^{-1}}^{-1}$ on the LHS, which is a precise way of treating the loss due to the fact that $v^{\prime\prime}$ effectively contains a derivative on $f$. 

\begin{proof}[Proof of Proposition \ref{lem:PrecisionEstimateNeq0}]
Since the coefficients only depend on $v$, $\Delta_t \phi = f$ decouples mode-by-mode in the $z$ frequencies. 
Hence, we essentially prove a mode-by-mode analogue of \eqref{ineq:PrecisionPhiNeq0} and then sum. 

As in Lemma \ref{lem:LossyElliptic} write (recall \eqref{def:Deltat},\eqref{def:DeltaL}), 
\begin{align*} 
\Delta_L \phi = f + (1-(v^\prime)^2)(\partial_v - t\partial_z)^2\phi - v^{\prime\prime}(\partial_v - t\partial_z)\phi.  
\end{align*}
Define the multipliers 
\begin{align*} 
\mathcal M_1(t,l,\xi) & = \jap{\frac{\xi}{lt}}^{-1}\frac{\abs{l,\xi}^{s/2}}{\jap{t}^s}A P_{\neq 0} \\
\mathcal M_2(t,l,\xi) & = \jap{\frac{\xi}{lt}}^{-1}\sqrt{\frac{\partial_t w}{w}} \tilde A P_{\neq 0}. 
\end{align*} 
Clearly, 
\begin{align*} 
\sum_{i=1,2}\norm{\mathcal{M}_i f}_2^2 \lesssim \frac{1}{\jap{t}^{2s}}\norm{\abs{\grad}^{s/2}P_{\neq 0}Af}_2^2 + \norm{\sqrt{\frac{\partial_t w}{w}}P_{\neq 0}\tilde Af}_2^2,
\end{align*} 
and hence the proposition would be trivial if $\Delta_L\phi = f$. 
Define, 
\begin{align*} 
T^1 & = (1-(v^\prime)^2)(\partial_v - t\partial_z)^2\phi \\ 
T^2 & = - v^{\prime\prime}(\partial_v - t\partial_z)\phi
\end{align*} 
and divide each via a paraproduct decomposition in the $v$ variable only 
\begin{align*} 
T^1 & = \sum_{M \geq 8} (1-(v^\prime)^2)_M(\partial_v - t\partial_z)^2\phi_{<M/8} + \sum_{M \geq 8} (1-(v^\prime)^2)_{<M/8}(\partial_v - t\partial_z)^2\phi_{M} \\  
& \quad + \sum_{M \in \mathbb D} \sum_{\frac{1}{8}M \leq M^\prime \leq 8M } (1-(v^\prime)^2)_{M^\prime}(\partial_v - t\partial_z)^2\phi_{M} \\ 
& = T^1_{HL} + T^1_{LH} + T^1_{\mathcal R} \\  
T^2 & = -\sum_{M \geq 8} (v^{\prime\prime})_M(\partial_v - t\partial_z)\phi_{<M/8} - \sum_{M \geq 8} (v^{\prime\prime})_{<M/8}(\partial_v - t\partial_z)\phi_{M} \\  
& \quad  - \sum_{M \in \mathbb D} \sum_{\frac{1}{8}M \leq M^\prime \leq 8M}(v^{\prime\prime})_{M^\prime}(\partial_v - t\partial_z)\phi_{M} \\  
& = T^2_{HL} + T^2_{LH} + T^2_{\mathcal R}. 
\end{align*}
The basic idea is to treat the HL terms by passing $\mathcal{M}_i$ onto the coefficients and to treat the $LH$ terms by passing the $\mathcal{M}_i$ onto $\phi$ and using $\epsilon$ sufficiently small to absorb these terms on the left-hand side of \eqref{ineq:PrecisionPhiNeq0}. 
Each step has several complications, dealt with and discussed below.

\subsubsection{Low-High interactions} \label{sec:lowhighElliptic}
Since $T^1_{LH}$ contains more derivatives on $\phi$ than $T^2_{LH}$, the former is strictly harder so we treat only  $T^1_{LH}$. 
In what follows we use the shorthand 
\begin{align*}
G(\xi) = \widehat{(1 - (v^\prime)^2)}(\xi). 
\end{align*}
Writing $T^1_{LH}$ on the frequency side with this convention gives 
\begin{align*}
\widehat{\mathcal{M}_1T^1_{LH}}(l,\xi) & = -\frac{1}{2\pi}\sum_{M \geq 8} \int_{\xi^\prime} \mathcal{M}_1(t,l,\xi) G(\xi-\xi^\prime)_{<M/8}(\xi^\prime-lt)^2 \hat{\phi}_l(\xi^\prime)_{M} d\xi^\prime. 
\end{align*} 
On the support of the integrand (see \S\ref{Apx:LPProduct}), 
\begin{align}
\abs{\abs{l,\xi} - \abs{l,\xi^\prime}} \leq \abs{\xi - \xi^\prime} \leq 3/2 \left(\frac{M}{16}\right) \leq 6\abs{\xi^\prime}/32 \leq 6\abs{l,\xi^\prime}/32, \label{ineq:TLHFreqLoc}
\end{align} 
and hence $26\abs{l,\xi^\prime}/32 \leq \abs{l,\xi} \leq 38\abs{l,\xi^\prime}/32$ and $26\abs{\xi^\prime}/32 \leq \abs{\xi} \leq 38\abs{\xi^\prime}/32$.   
Therefore, \eqref{lem:scon} implies that there exists a $c \in (0,1)$ such that 
\begin{align*} 
e^{\lambda\abs{l,\xi}^s} \leq e^{\lambda\abs{l,\xi^\prime}^s + c\lambda\abs{\xi-\xi^\prime}^s}.
\end{align*}  
Therefore (also using $\abs{\xi} \approx \abs{\xi^\prime}$): 
\begin{align*} 
\abs{\widehat{\mathcal{M}_1T^1_{LH}}(l,\xi)} & \lesssim \sum_{M \geq 8} \int_{\xi^\prime} \jap{\frac{\xi^\prime}{lt}}^{-1}\frac{\abs{l,\xi^\prime}^{s/2}}{\jap{t}^s}\jap{l,\xi^\prime}^\sigma J_l(\xi) \\ & \quad\quad \times e^{c\lambda\abs{\xi-\xi^\prime}^s}\abs{G(\xi-\xi^\prime)_{<M/8}} \abs{\xi^\prime-lt}^2 \abs{\hat{\phi}_l(\xi^\prime)_{M}}e^{\lambda\abs{l,\xi^\prime}^s} d\xi^\prime. 
\end{align*}
The goal is now to pass the multiplier $\mathcal{M}_1$ onto $\phi$. 
It is important here that we are not comparing different modes in $z$, and hence \eqref{ineq:BasicJswap} applies. Hence by \eqref{ineq:IncExp} (since $c < 1$),  
\begin{align*}
\abs{\widehat{\mathcal{M}_1T^{1}_{LH}}(l,\xi)} & \lesssim \sum_{M \geq 8} \int_{\xi^\prime}\jap{\frac{\xi^\prime}{lt}}^{-1}\frac{\abs{l,\xi^\prime}^{s/2}}{\jap{t}^s} e^{\lambda\abs{\xi-\xi^\prime}^s} \abs{G(\xi-\xi^\prime)_{<M/8}} \abs{\xi^\prime-lt}^2A_l(\xi^\prime)\abs{\hat{\phi}_l(\xi^\prime)_{M}} d\xi^\prime. 
\end{align*} 
Then \eqref{ineq:L2L1}, \eqref{ineq:GeneralOrtho} and \eqref{ineq:vp2m1bd} imply 
\begin{align*} 
\norm{\mathcal{M}_1T^{1}_{LH}}_2^2 = \sum_{l\neq 0}\norm{\mathcal{M}_1T^{1}_{LH}(l)}_2^2 & \lesssim \sum_{l\neq 0}\sum_{M \geq 8} \norm{(1-(v^\prime)^2)_{<M/8}}^2_{\G^{\lambda,2}} \norm{\mathcal{M}_1\Delta_L P_{\neq 0} (\phi_l)_{M}}^2_2  \\ & \lesssim \epsilon^2\norm{\mathcal{M}_1\Delta_L P_{\neq 0} \phi}^2_2, 
\end{align*} 
completing the treatment of $\mathcal{M}_1T^{1}_{LH}$, as this can be absorbed by the LHS of \eqref{ineq:PrecisionPhiNeq0}.

Now turn to $\mathcal M_2 T_{LH}^1$, which by the frequency localization on the support of the integrand \eqref{ineq:TLHFreqLoc} together with \eqref{lem:scon}, is bounded by
\begin{align*}
\abs{\widehat{\mathcal{M}_2T^1_{LH}}(l,\xi)} & \lesssim \sum_{M \geq 8} \int_{\xi^\prime} \jap{\frac{\xi^\prime}{lt}}^{-1}\sqrt{\frac{\partial_t w_l(\xi)}{w_l(\xi)}}\jap{l,\xi^\prime}^\sigma \tilde J_l(\xi) \\ & \quad\quad \times e^{c\lambda\abs{\xi-\xi^\prime}^s}G(\xi-\xi^\prime)_{<M/8} \abs{\xi^\prime-lt}^2\abs{\hat{\phi}_l(\xi^\prime)_{M}}e^{\lambda\abs{l,\xi^\prime}^s} d\xi^\prime.
\end{align*}
Now, by \eqref{ineq:partialtw_endpt}.  it follows that a proof similar to that used to treat $\mathcal{M}_1T^{1}_{LH}$ implies
\begin{align*} 
\norm{\mathcal{M}_2T^{1}_{LH}}_2^2 \lesssim \epsilon^2 \norm{\mathcal{M}_1 \Delta_L \phi}_2^2  + \epsilon^2 \norm{\mathcal{M}_2 \Delta_L \phi}_2^2,  
\end{align*} 
which completes the treatment of $\norm{\mathcal{M}_2T^{1}_{LH}}_2^2$. 

\subsubsection{High-Low interactions} \label{sec:HLElliptic}
Consider first $\mathcal{M}_1T^2_{HL}$. 
The notation is deceptive: the frequency in $z$ could be very large and hence more `derivatives' are appearing on $\phi$ and we will be in a situation like the LH terms. 
Hence we break into two cases: 
\begin{align*} 
\widehat{\mathcal{M}_1T^2_{HL}}(l,\xi) & = -\frac{1}{2\pi}\sum_{M \geq 8} \int_{\xi^\prime} \left[\mathbf{1}_{\abs{l} \geq \frac{1}{16}\abs{\xi}} + \mathbf{1}_{\abs{l} < \frac{1}{16}\abs{\xi}}  \right] \mathcal{M}_1(t,l,\xi)\widehat{v^{\prime\prime}}(\xi-\xi^\prime)_M i(\xi^\prime-lt)\hat{\phi}_l(\xi^\prime)_{<M/8} d\xi^\prime \\ 
& = \widehat{\mathcal{M}_1T^{2,z}_{HL}}(l,\xi) + \widehat{\mathcal{M}_1T^{2,v}_{HL}}(l,\xi). 
\end{align*} 
First consider $\mathcal{M}_1T^{2,z}_{HL}$, which we treat as a Low-High term.
On the support of the integrand, we claim that there is some $c \in (0,1)$ such that,
\begin{align} 
\abs{l,\xi}^s \leq \abs{l,\xi^\prime}^s + c\abs{\xi - \xi^\prime}^s. \label{ineq:lgainScon}
\end{align} 
To see this, one can consider separately the cases $\frac{1}{16}\abs{\xi} \leq \abs{l} \leq 16\abs{\xi}$ and $\abs{l} > 16\abs{\xi}$, applying \eqref{lem:strivial}  and \eqref{lem:scon} respectively. 
It follows that 
\begin{align*} 
\abs{\widehat{\mathcal{M}_1T^{2,z}_{HL}}(l,\xi)} & \lesssim \sum_{M \geq 8} \int_{\xi^\prime}\mathbf{1}_{\abs{l} > \frac{1}{16}\abs{\xi}} \frac{\abs{l}^{s/2}}{\jap{t}^s}\jap{l}^\sigma J_l(\xi) e^{c\lambda\abs{\xi-\xi^\prime}^s} \abs{\widehat{v^{\prime\prime}}(\xi-\xi^\prime)_M} \abs{\xi^\prime-lt}\abs{\hat{\phi}_l(\xi^\prime)_{<M/8}}e^{\lambda\abs{l,\xi^\prime}^s} d\xi^\prime,  
\end{align*} 
where we also used $\abs{\xi} \lesssim \abs{lt}$ to remove the leading factor $\jap{\xi/lt}^{-1}$.
As in the treatment of $\mathcal{M}_1T^{1}_{LH}$, we apply \eqref{ineq:BasicJswap} (as we are not comparing different $z$ modes) and \eqref{ineq:IncExp} to deduce 
\begin{align*} 
\abs{\widehat{\mathcal{M}_1T^{2,z}_{HL}}(l,\xi)} & \lesssim \sum_{M \geq 8} \int_{\xi^\prime}\mathbf{1}_{\abs{l} > \frac{1}{16}\abs{\xi}}  \frac{\abs{l}^{s/2}}{\jap{t}^s}e^{\lambda\abs{\xi-\xi^\prime}^s}\abs{\widehat{v^{\prime\prime}}(\xi-\xi^\prime)_M} \abs{\xi^\prime-lt}A_l(\xi^\prime)\abs{\hat{\phi}_l(\xi^\prime)_{<M/8}}d\xi^\prime. 
\end{align*} 
Since $\abs{\xi^\prime - lt} \leq \abs{l}^2 + \abs{\xi^\prime -lt}^2$, applying \eqref{ineq:L2L1}, \eqref{ineq:GeneralOrtho} and \eqref{ineq:coefbds}, 
\begin{align} 
\norm{\mathcal{M}_1T^{2,z}_{HL}}_2^2 = \sum_{l\neq 0}\norm{\mathcal{M}_1T^{2,z}_{HL}(l)}_2^2 & \lesssim \sum_{M \geq 8} \norm{v^{\prime\prime}_M}^2_{\G^{\lambda,3}} \norm{\mathcal{M}_1\Delta_L P_{\neq 0} \phi}^2_2 \nonumber  \\ & \lesssim \epsilon^2\norm{\mathcal{M}_1\Delta_L P_{\neq 0} \phi}^2_2.  \label{ineq:M1T2z}
\end{align} 

Next consider $\mathcal{M}_1T^{2,v}_{HL}$.
On the support of the integrand, the `derivatives' are all landing on $v^{\prime\prime}$ and since 
this function essentially contains a derivative of $f$ it is here where we need the $\jap{\xi/lt}^{-1}$ (see \eqref{def:vPDE}).
 In this case, using $\abs{\xi^\prime} \leq \frac{3}{16}\abs{\xi-\xi^\prime}$ analogous to \eqref{ineq:TLHFreqLoc},
\begin{align}
 \abs{\abs{\xi - \xi^\prime}-\abs{l,\xi}} \leq \abs{l,\xi^\prime} \leq \abs{\xi}/16 + \abs{\xi^\prime} \leq \abs{\xi-\xi^\prime}/16 + 17\abs{\xi^\prime}/16 \leq \frac{67}{256}\abs{\xi-\xi^\prime}. \label{ineq:THLFreqLoc}
\end{align} 
Therefore by \eqref{lem:scon}, there exists some $c \in (0,1)$ such that 
\begin{align*} 
\abs{\widehat{\mathcal{M}_1T^{2,v}_{HL}}(l,\xi)} & \lesssim \sum_{M \geq 8} \int_{\xi^\prime}\mathbf{1}_{\abs{l} < \frac{1}{16}\abs{\xi}} \jap{\frac{\xi}{lt}}^{-1}\frac{\abs{l,\xi}^{s/2}}{\jap{t}^s}\jap{l,\xi}^\sigma J_l(\xi)e^{\lambda\abs{\xi-\xi^\prime}^s} \\ & \quad\quad \times 
\abs{\widehat{v^{\prime\prime}}(\xi-\xi^\prime)_M} \abs{\xi^\prime-lt}\abs{\hat{\phi}_l(\xi^\prime)_{<M/8}}e^{c\lambda\abs{l,\xi^\prime}^s} d\xi^\prime. 
\end{align*}  
Since we will pass $\mathcal{M}_1$ onto $v^{\prime\prime}$, Lemma \ref{lem:Jswap} could imply a loss.
However, from Lemma \ref{lem:Jswap} and the definitions \eqref{def:AR}, \eqref{def:wR}, we see that on the support of the integrand we have (using that $\abs{\xi} \gtrsim \abs{l}$): 
\begin{align} 
J_l(\xi) \lesssim J^R(\xi-\xi^\prime)e^{20\mu \abs{\xi^\prime}^{1/2}}. \label{ineq:JRswap}
\end{align}
Since $c < 1$ and $s > 1/2$, then \eqref{ineq:JRswap} and \eqref{ineq:IncExp} (also $\abs{l} \lesssim \abs{\xi} \approx \abs{\xi-\xi^\prime}$) imply
\begin{align*} 
\abs{\widehat{\mathcal{M}_1T^{2,v}_{HL}}(l,\xi)} & \lesssim \sum_{M \geq 8} \int_{\xi^\prime}\jap{\frac{\xi-\xi^\prime}{lt}}^{-1}\frac{\abs{\xi-\xi^\prime}^{s/2}}{\jap{t}^s}A^R(\xi-\xi^\prime)\abs{\widehat{v^{\prime\prime}}(\xi-\xi^\prime)_M} \abs{\xi^\prime-lt}\abs{\hat{\phi}_l(\xi^\prime)_{<M/8}}e^{\lambda\abs{l,\xi^\prime}^s} d\xi^\prime. 
\end{align*}
Therefore, 
\begin{align*} 
\abs{\widehat{\mathcal{M}_1T^{2,v}_{HL}}(l,\xi)}& \lesssim \sum_{M \geq 8} \int_{\xi^\prime} \frac{1}{\jap{\frac{\xi-\xi^\prime}{lt}}\jap{lt}} \frac{\abs{\xi-\xi^\prime}^{s/2}}{\jap{t}^s} A^R(\xi-\xi^\prime)\abs{\widehat{v^{\prime\prime}}(\xi-\xi^\prime)_M} \jap{lt}\abs{\xi^\prime-lt}\abs{\hat{\phi}_l(\xi^\prime)_{<M/8}}e^{\lambda\abs{l,\xi^\prime}^s} d\xi^\prime \\ 
& \lesssim \sum_{M \geq 8} \int_{\xi^\prime} \frac{\abs{\xi-\xi^\prime}^{s/2}}{\jap{t}^s}\frac{ A^R(\xi-\xi^\prime)}{\jap{\xi-\xi^\prime}}\abs{\widehat{v^{\prime\prime}}(\xi-\xi^\prime)_M} \jap{lt}\abs{\xi^\prime-lt}\abs{\hat{\phi}_l(\xi^\prime)_{<M/8}}e^{\lambda\abs{l,\xi^\prime}^s} d\xi^\prime. 
\end{align*}
Finally, by \eqref{ineq:L2L1} ($\sigma > 6$), \eqref{ineq:GeneralOrtho}
Lemma \ref{lem:LossyElliptic} and the bootstrap hypotheses we have
\begin{align} 
 \norm{\mathcal{M}_1T^{2,v}_{HL}}_2^2 & \lesssim \sum_{M \geq 8}\frac{1}{\jap{t}^{2s}}\norm{\abs{\partial_v}^{s/2}\frac{A^R}{\jap{\partial_v}} v^{\prime\prime}_M }_2^2 \norm{t^2 \phi_{<M/8}}_{\G^{\lambda,\sigma-3}}^2 \nonumber \\ 
& \lesssim  \frac{\epsilon^2}{\jap{t}^{2s}}\norm{\abs{\partial_v}^{s/2}\frac{A^R}{\jap{\partial_v}} v^{\prime\prime} }_2^2 \lesssim \epsilon^2 CCK_\lambda^2. \label{ineq:M1T2vR}
\end{align}
This is sufficient to treat $\mathcal{M}_1T^{2}_{HL}$. 

The corresponding argument to show 
\begin{align*} 
\norm{\mathcal{M}_1T^1_{HL}}_2^2 \lesssim \epsilon^2 \norm{\mathcal{M}_1\Delta_L \phi}^2_2 + \epsilon^2 CCK^1_\lambda. 
\end{align*}
is similar and hence omitted. Note that in the case of $T^1_{HL}$ no derivative needs to be recovered on the coefficients (and would be impossible due to lack of time-decay). This completes the treatment of the High-Low terms involving $\mathcal{M}_1$. 

The argument for $\mathcal{M}_2$ is only slightly different. 
We treat $\mathcal{M}_2T^{2}_{HL}$; the case $\mathcal{M}_2T^{1}_{HL}$ is analogous.
As in $\mathcal{M}_1T^2_{HL}$, we divide into separate cases: 
\begin{align*} 
\widehat{\mathcal{M}_2T^2_{HL}}(l,\xi) & = -\frac{1}{2\pi}\sum_{M \geq 8} \int_{\xi^\prime} \left[\mathbf{1}_{\abs{l} \geq \frac{1}{16}\abs{\xi}} + \mathbf{1}_{\abs{l} < \frac{1}{16}\abs{\xi}}  \right] \mathcal{M}_2(t,l,\xi)\widehat{v^{\prime\prime}}(\xi-\xi^\prime)_M i(\xi^\prime-lt)\hat{\phi}_l(\xi^\prime)_{<M/8} d\xi^\prime \\ 
& = \widehat{\mathcal{M}_2T^{2,z}_{HL}}(l,\xi) + \widehat{\mathcal{M}_2T^{2,v}_{HL}}(l,\xi). 
\end{align*} 
First consider $\mathcal{M}_2T^{2,z}_{HL}$, which like $\mathcal{M}_1T^{2,z}_{HL}$, we treat as a Low-High term.
As there, \eqref{ineq:lgainScon} applies on the support of the integrand and hence by \eqref{ineq:BasicJswap} (since we are not comparing different $z$ frequencies) and \eqref{ineq:IncExp}: 
\begin{align*} 
\abs{\widehat{\mathcal{M}_2T^{2,z}_{HL}}(l,\xi)} & \lesssim \sum_{M \geq 8} \int_{\xi^\prime}\mathbf{1}_{\abs{l} \geq \frac{1}{16}\abs{\xi}}\sqrt{\frac{\partial_tw_l(\xi)}{w_l(\xi)}}e^{\lambda\abs{\xi-\xi^\prime}^s} \abs{\widehat{v^{\prime\prime}}(\xi-\xi^\prime)_M} \abs{\xi^\prime-lt}\tilde A_l(\xi^\prime)\abs{\hat{\phi}_l(\xi^\prime)_{<M/8}}d\xi^\prime.
\end{align*} 
As in the treatment of $\mathcal{M}_2T^{1}_{LH}$ above in \S\ref{sec:lowhighElliptic}, we apply \eqref{ineq:partialtw_endpt}, and then a proof similar to that used to treat $\mathcal{M}_1T^{2,z}_{HL}$ in \eqref{ineq:M1T2z} implies 
\begin{align} 
\norm{\mathcal{M}_2T^{2,z}_{HL}}_2^2 \lesssim \epsilon^2 \norm{\mathcal{M}_2 \Delta_L \phi}_2^2 + \epsilon^2 \norm{\mathcal{M}_1 \Delta_L \phi}_2^2, \label{ineq:M2T2z}
\end{align}
which completes the treatment of $\mathcal{M}_2T^{2,z}_{HL}$. 

Now turn to $\mathcal{M}_2T^{2,v}_{HL}$. 
As in the treatment of $\mathcal{M}_1T^{2,v}_{HL}$, \eqref{ineq:THLFreqLoc} holds on the support of the integrand, and hence so does \eqref{ineq:JRswap} with $J_l(\xi)$ replaced by $\tilde J_l(\xi)$ (recall \eqref{def:tildeJB}).
Therefore, by \eqref{lem:scon} for some $c \in (0,1)$ followed by \eqref{ineq:JRswap} and \eqref{ineq:IncExp} implies
\begin{align*} 
\abs{\widehat{\mathcal{M}_2T^{2,v}_{HL}}(l,\xi)} & \lesssim \sum_{M \geq 8} \int_{\xi^\prime}\mathbf{1}_{\abs{l} < \frac{1}{16}\abs{\xi}} \jap{\frac{\xi}{lt}}^{-1}\sqrt{\frac{\partial_t w_l(\xi)}{w_l(\xi)}}\jap{\xi-\xi^\prime}^\sigma \tilde J_l(\xi) e^{\lambda\abs{\xi-\xi^\prime}^s} \\ & \quad\quad\quad \times \abs{\widehat{v^{\prime\prime}}(\xi-\xi^\prime)_M} \abs{\xi^\prime-lt}\abs{\hat{\phi}_l(\xi^\prime)_{<M/8}}e^{c\lambda\abs{l,\xi^\prime}^s} d\xi^\prime \\ 
& \lesssim \sum_{M \geq 8} \int_{\xi^\prime}\frac{\mathbf{1}_{\abs{l} < \frac{1}{16}\abs{\xi}}}{\jap{\frac{\xi}{lt}} \abs{lt}}\sqrt{\frac{\partial_t w_l(\xi)}{w_l(\xi)}}A^R(\xi-\xi^\prime) \abs{\widehat{v^{\prime\prime}}(\xi-\xi^\prime)_M} \\ & \quad\quad \times \abs{lt}\abs{\xi^\prime-lt}\abs{\hat{\phi}_l(\xi^\prime)_{<M/8}}e^{\lambda\abs{l,\xi^\prime}^s} d\xi^\prime.
\end{align*}  
Then by $\abs{l} \lesssim \abs{\xi} \approx \abs{\xi-\xi^\prime}$ and \eqref{ineq:partialtw_endpt} we have
\begin{align*} 
\abs{\widehat{\mathcal{M}_2T^{2,v}_{HL}}(l,\xi)} & \lesssim \sum_{M \geq 8} \int_{\xi^\prime}\left(\sqrt{\frac{\partial_t w_0(\xi - \xi^\prime)}{w_0(\xi-\xi^\prime)}} + \frac{\abs{\xi-\xi^\prime}^{s/2}}{\jap{t}^{s}} \right) \frac{A^R(\xi-\xi^\prime)}{\jap{\xi-\xi^\prime}} \abs{\widehat{v^{\prime\prime}}(\xi-\xi^\prime)_M} \\ & \quad\quad \times \abs{lt}\abs{\xi^\prime-lt}\jap{\xi^\prime}\abs{\hat{\phi}_l(\xi^\prime)_{<M/8}}e^{\lambda\abs{l,\xi^\prime}^s} d\xi^\prime.
\end{align*} 
Hence, by \eqref{ineq:L2L1} ($\sigma > 7$), \eqref{ineq:GeneralOrtho}, 
Lemma \ref{lem:LossyElliptic} and the bootstrap hypotheses we have
\begin{align*} 
 \norm{\mathcal{M}_2T^{2,v}_{HL}}_2^2 & \lesssim \sum_{M \geq 8}\left(\frac{1}{\jap{t}^{2s}}\norm{\abs{\partial_v}^{s/2}\frac{A^R}{\jap{\partial_v}} v^{\prime\prime}_M }_2^2 + \norm{\sqrt{\frac{\partial_t w}{w}} \frac{A^R}{\jap{\partial_v}} v^{\prime\prime}_M}_2^2 \right) \norm{t^2 \phi_{<M/8}}_{\G^{\lambda,\sigma-3}}^2 \nonumber \\ 
& \lesssim \epsilon^2 CCK_{\lambda}^2 + \epsilon^2 CCK_w^2. 
\end{align*}
Together with \eqref{ineq:M2T2z}, this completes the treatment of $\mathcal{M}_2T^{2}_{HL}$. The treatment of $\mathcal{M}_2T^{1}_{HL}$ is analogous and omitted.

\subsubsection{Remainders} \label{sec:EllipticRemainder}
The last terms to consider are $T^1_{\mathcal{R}}$ and $T^2_{\mathcal{R}}$.
In these terms powers of $\partial_v$  can be split evenly between the two factors.  
However, the same is not true of $l$. For this reason, we treat both remainders as Low-High terms. 
The difference between $T^1$ and $T^2$ here is insignificant since it is straightforward to gain a power of $\jap{\partial_v}^{-1}$ for $v^{\prime\prime}$. 
Hence we focus only on $T^1_{\mathcal{R}}$. 

Begin with $\mathcal{M}_1T^1_{\mathcal{R}}$ and divide into two cases based on the relative size of $\xi^\prime$ and $l$, 
\begin{align*} 
\abs{\widehat{\mathcal{M}_1T^1_{\mathcal R}}(l,\xi)}
 & \lesssim \sum_{M \in \mathbb D}\sum_{M/8 \leq M^\prime \leq 8M} \int \left[ \mathbf{1}_{\abs{l} > 100\abs{\xi^\prime}} + \mathbf{1}_{\abs{l} \leq 100\abs{\xi^\prime}}\right] \mathcal{M}_1(t,l,\xi) \abs{G(\xi-\xi^\prime)_{M^\prime}} \abs{\xi^\prime-lt}^2\abs{\hat{\phi}_l(\xi^\prime)_{M}} d\xi^\prime \\ 
& = \abs{\widehat{\mathcal{M}_1T^{1,z}_{\mathcal R}}(l,\xi)} + \abs{\widehat{\mathcal{M}_1T^{1,v}_{\mathcal R}}(l,\xi)}. 
\end{align*} 
Consider first $\mathcal{M}_1T^{1,z}_{\mathcal R}$. 
Since on the support of the integrand,
\begin{align} 
\abs{\abs{l,\xi} - \abs{l,\xi^\prime}} \leq \abs{\xi - \xi^\prime} \leq \frac{3M^\prime}{2} \leq 12 M \leq 24\abs{\xi^\prime} \leq \frac{24}{100}\abs{l,\xi^\prime},\label{ineq:M1TR}
\end{align}
inequality \eqref{lem:scon} implies, 
\begin{align*} 
\abs{\widehat{\mathcal{M}_1T^{1,z}_{\mathcal{R}}(l,\xi)}} & \lesssim \sum_{M \in \mathbb D}\sum_{M/8 \leq M^\prime \leq 8M} \int \mathbf{1}_{\abs{l} > 100\abs{\xi^\prime}} 
\jap{\frac{\xi}{lt}}^{-1}\frac{\abs{l,\xi^\prime}^{s/2}}{\jap{t}^s}\jap{l,\xi^\prime}^\sigma J_l(\xi)e^{c\lambda\abs{\xi-\xi^\prime}^s} \\ & \quad\quad \times \abs{G(\xi-\xi^\prime)_{M^\prime}} \abs{\xi^\prime-lt}^2\abs{\hat{\phi}_l(\xi^\prime)_{M}}e^{\lambda\abs{l,\xi^\prime}^s} d\xi^\prime.
\end{align*}  
By $\frac{1}{24}\abs{\xi^\prime} \leq \abs{\xi-\xi^\prime} \leq 24\abs{\xi^\prime}$ we have the rough bound from \eqref{w-grwth}, 
\begin{align} 
\frac{J_l(\xi)}{J_l(\xi^\prime)} & \lesssim e^{\frac{3\mu}{2}\abs{\xi}^{1/2}} \lesssim e^{50\mu\abs{\xi-\xi^\prime}^{1/2}}. \label{ineq:RemainderRough}
\end{align}
Therefore, \eqref{ineq:IncExp} implies 
\begin{align*} 
\abs{\widehat{\mathcal{M}_1T^{1,z}_{\mathcal{R}}}(l,\xi)} & \lesssim \sum_{M \in \mathbb D}\sum_{M/8 \leq M^\prime \leq 8M} \int \mathbf{1}_{\abs{l} > 100\abs{\xi^\prime}}\jap{\frac{\xi}{lt}}^{-1}\frac{\abs{l,\xi}^{s/2}}{\jap{t}^s} \\ & \quad\quad \times  e^{\lambda\abs{\xi-\xi^\prime}^s}\abs{G(\xi-\xi^\prime)_{M^\prime}} \abs{\xi^\prime-lt}^2A_l(\xi^\prime) \abs{\hat{\phi}_l(\xi^\prime)_{M}}d\xi^\prime.
\end{align*} 
Since $t \geq 1$, and $\jap{\xi^\prime/lt} \approx 1$ on the support of the integrand, by \eqref{ineq:M1TR},
\begin{align*} 
\abs{\widehat{\mathcal{M}_1T^{1,z}_{\mathcal{R}}}(l,\xi)} & \lesssim \sum_{M \in \mathbb D}\sum_{M/8 \leq M^\prime \leq 8M} \int \mathbf{1}_{\abs{l} > 100\abs{\xi^\prime}} \mathcal{M}_1(t,l,\xi^\prime)e^{\lambda\abs{\xi-\xi^\prime}^s}\abs{G(\xi-\xi^\prime)_{M^\prime}} \abs{\xi^\prime-lt}^2\abs{\hat{\phi}_l(\xi^\prime)_{M}} d\xi^\prime.
\end{align*}
Taking the $L^2$ norm in $\xi$, applying \eqref{ineq:L2L1} and \eqref{ineq:GeneralOrtho} (note the $M^\prime$ sum only contains 7 terms), 
\begin{align*} 
\norm{\mathcal{M}_1T^{1,z}_{\mathcal{R}}(l)}_2  & \lesssim \sum_{M^\prime \in \mathbb D} \norm{(1-(v^\prime)^2)_{M^\prime}}_{\G^{\lambda,2}} \sum_{M^\prime/8 \leq M \leq 8M^\prime} \norm{\mathcal{M}_1\Delta_L (\phi_l)_{M}}_2 \\ 
 & \lesssim \sum_{M^\prime \in \mathbb D} \norm{(1-(v^\prime)^2)_{M^\prime}}_{\G^{\lambda,2}}\norm{\mathcal{M}_1\Delta_L (\phi_l)_{\sim M^\prime}}_2 \\ 
& \lesssim \left(\sum_{M^\prime \in \mathbb D} \norm{(1-(v^\prime)^2)_{M^\prime}}^2_{\G^{\lambda,2}}\right)^{1/2} \left(\sum_{M^\prime \in \mathbb D}\norm{\mathcal{M}_1\Delta_L (\phi_l)_{\sim M^\prime}}_2^2 \right)^{1/2} \\ 
& \lesssim \epsilon\norm{\mathcal{M}_1\Delta_L \phi_l}_2,  
\end{align*} 
where the last line followed from \eqref{ineq:vp2m1bd}.
Taking squares and summing over $l \neq 0$ implies
\begin{align} 
\norm{\mathcal{M}_1T^{1,z}_{\mathcal{R}}}_2 & \lesssim \epsilon^2\norm{\mathcal{M}_1\Delta_L P_{\neq 0}\phi}^2_2.  \label{ineq:M1T1zR}
\end{align}
Turn now to $\mathcal{M}_1T^{1,v}_{\mathcal R}$. 
On the support of the integrand in this case, 
\begin{align*} 
\abs{\xi - \xi^\prime} & \leq 24\abs{\xi^\prime} \leq 24\abs{l,\xi^\prime} \\ 
\abs{l,\xi^\prime} & \leq 101\abs{\xi^\prime} \leq 2424\abs{\xi - \xi^\prime}. 
\end{align*}
Therefore by \eqref{lem:strivial} there exists a $c \in (0,1)$ such that
\begin{align*} 
\abs{l,\xi}^{s} \leq \abs{\abs{l,\xi^\prime} + \abs{\xi-\xi^\prime}}^s \leq c\abs{l,\xi^\prime}^s + c\abs{\xi-\xi^\prime}^s. 
\end{align*}
Hence, 
\begin{align*} 
\abs{\widehat{\mathcal{M}_1T^{1,v}_{\mathcal{R}}}(l,\xi)} & \lesssim \sum_{M \in \mathbb D} \sum_{M/8 \leq M^\prime \leq 8M} \int \mathbf{1}_{\abs{l} \leq 100\abs{\xi^\prime}} \jap{\frac{\xi}{lt}}^{-1}\frac{\abs{l,\xi}^{s/2}}{\jap{t}^s}\jap{l,\xi^\prime}^{\sigma/2}\jap{\xi-\xi^\prime}^{\sigma/2} \\ & \quad\quad \times J_l(\xi)e^{c\lambda\abs{\xi-\xi^\prime}^s}\abs{G(\xi-\xi^\prime)_{M^\prime}} \abs{\xi^\prime-lt}^2\abs{\hat{\phi}_l(\xi^\prime)_{M}}e^{c\lambda\abs{l,\xi^\prime}^s} d\xi^\prime.
\end{align*} 
Using again \eqref{ineq:RemainderRough} and \eqref{ineq:IncExp} with $\abs{l,\xi^\prime} \approx \abs{\xi-\xi^\prime}$, and $\jap{\frac{\xi^\prime}{lt}}\jap{\frac{\xi}{lt}}^{-1} \lesssim \jap{\xi^\prime}$ implies 
\begin{align*} 
\abs{\widehat{\mathcal{M}_1T^{1,v}_{\mathcal{R}}}(l,\xi)} & \lesssim \sum_{M \in \mathbb D}\sum_{M/8 \leq M^\prime \leq 8M} \int \mathbf{1}_{\abs{l} \leq 100\abs{\xi^\prime}} 
\jap{\frac{\xi^\prime}{lt}}^{-1}\frac{\abs{l,\xi^\prime}^{s/2}}{\jap{t}^s} e^{\lambda\abs{\xi-\xi^\prime}^s} 
\\ & \quad\quad \times \jap{\xi-\xi^\prime}^{\sigma/2+1}
\abs{G(\xi-\xi^\prime)_{M^\prime}}\abs{\xi^\prime-lt}^2A_l(\xi^\prime)\abs{\hat{\phi}_l(\xi^\prime)_{M}} d\xi^\prime.
\end{align*}
An argument similar to that used to complete the proof of $\norm{\mathcal{M}_1T^{1,z}_{\mathcal{R}}}_2$ in \eqref{ineq:M1T1zR} implies 
\begin{align*} 
\norm{\mathcal{M}_1T^{1,v}_{\mathcal{R}}}_2^2 
 & \lesssim \epsilon^2 \norm{\mathcal{M}_1\Delta_L\phi}_2^2.
\end{align*} 
This completes the treatment of $\mathcal{M}_1T^1_{\mathcal{R}}$; as discussed above, $\mathcal{M}_1T^2_{\mathcal{R}}$ is treated in a similar manner and is hence omitted.
Combining this argument with those used to treat $\mathcal{M}_2T^1_{LH}$ we may also easily treat $\mathcal{M}_2T^{1}_{\mathcal{R}}$ and $\mathcal{M}_2T^{2}_{\mathcal{R}}$.  
The proof is omitted and the result is 
\begin{align*} 
\norm{\mathcal{M}_2T^1_{\mathcal R}}_2^2 + \norm{\mathcal{M}_2T^2_{\mathcal R}}_2^2  & \lesssim  \epsilon^2 \norm{\mathcal{M}_1\Delta_L\phi}_2^2 + \epsilon^2 \norm{\mathcal{M}_2\Delta_L\phi}_2^2.  
\end{align*}
This completes the treatment of the remainder terms and hence the proof of Proposition \ref{lem:PrecisionEstimateNeq0}.
\end{proof}

\section{Transport} \label{sec:Transport}
In this section we prove Proposition \ref{prop:Transport}. 
As discussed above, we adapt methods similar to \cite{FoiasTemam89,LevermoreOliver97,KukavicaVicol09} for this purpose. 
This adaptation is not completely straightforward since $J_k(\eta)$ assigns slightly different
regularities to modes which are near the critical time. 
Dealing with this will require special attention and all of the available time decay from the velocity field.

In the methods of \cite{FoiasTemam89,LevermoreOliver97,KukavicaVicol09} the goal is to gain $1-s$ derivatives from the difference $A_k(\eta) - A_l(\xi)$, and hence be able to absorb 
the leading contributions of $T_N$ with $CK_\lambda$. 
Decompose this difference:
\begin{align*} 
A_k(\eta) - A_l(\xi) & = A_l(\xi)\left[e^{\lambda\abs{k,\eta}^s - \lambda \abs{l,\xi}^s} - 1\right] + A_l(\xi)e^{\lambda\abs{k,\eta}^s - \lambda \abs{l,\xi}^s}\left[\frac{J_k(\eta)}{J_l(\xi)} - 1\right] \frac{\jap{k,\eta}^\sigma}{\jap{l,\xi}^\sigma} \\ &  \quad + A_l(\xi)e^{\lambda\abs{k,\eta}^s - \lambda \abs{l,\xi}^s}\left[ \frac{\jap{k,\eta}^\sigma}{\jap{l,\xi}^\sigma} - 1\right] 
\end{align*} 
In what follows we write
\begin{align*} 
T_{N} & = i\sum_{k,l}\int_{\eta,\xi} A_k(\eta)\bar{\hat{f}}_k(\eta) \hat{u}_{k-l}(\eta-\xi)_{<N/8} \cdot (l,\xi)A_l(\xi)\hat{f}_l(\xi)_{N}\left[e^{\lambda\abs{k,\eta}^s - \lambda \abs{l,\xi}^s} - 1\right] d\eta d\xi \\ 
&\quad + i\sum_{k,l}\int_{\eta,\xi} A_k(\eta)\bar{\hat{f}}_k(\eta) \hat{u}_{k-l}(\eta-\xi)_{<N/8} \cdot (l,\xi)A_l(\xi)\hat{f}_l(\xi)_{N} e^{\lambda\abs{k,\eta}^s - \lambda \abs{l,\xi}^s}\left[\frac{J_k(\eta)}{J_l(\xi)} - 1\right] \frac{\jap{k,\eta}^\sigma}{\jap{l,\xi}^\sigma} d\eta d\xi \\ 
& \quad + i\sum_{k,l}\int_{\eta,\xi} A_k(\eta)\bar{\hat{f}}_k(\eta) \hat{u}_{k-l}(\eta-\xi)_{<N/8} \cdot (l,\xi)A_l(\xi)\hat{f}_l(\xi)_{N} e^{\lambda\abs{k,\eta}^s - \lambda \abs{l,\xi}^s}\left[ \frac{\jap{k,\eta}^\sigma}{\jap{l,\xi}^\sigma} - 1\right]  d\eta d\xi \\
& = T_{N;1} + T_{N;2} + T_{N;3}
\end{align*}

\subsection{Term $T_{N;1}$: exponential regularity}
First we treat $T_{N;1}$ for which the methods of \cite{FoiasTemam89,LevermoreOliver97,KukavicaVicol09} easily adapt. 
By $\abs{e^{x}-1} \leq xe^{x}$ and \eqref{ineq:TrivDiff},
\begin{align*} 
\abs{T_{N;1}} & \leq \sum_{k,l}\int_{\eta,\xi} \abs{A\hat{f}_k(\eta)} \abs{\hat{u}_{k-l}(\eta-\xi)_{<N/8}} \abs{l,\xi} A_l(\xi)\abs{\hat{f}_l(\xi)_{N}}\lambda\abs{\abs{k,\eta}^s - \abs{l,\xi}^s}e^{\lambda \abs{\abs{k,\eta}^s - \abs{l,\xi}^s}} d\eta d\xi \\ 
& \lesssim \lambda\sum_{k,l}\int_{\eta,\xi} \abs{A\hat{f}_k(\eta)} \abs{\hat{u}_{k-l}(\eta-\xi)_{<N/8}} \abs{l,\xi} A_l(\xi)\abs{\hat{f}_l(\xi)_{N}}\frac{\abs{\abs{k,\eta} - \abs{l,\xi}}}{\abs{k,\eta}^{1-s} + \abs{l,\xi}^{1-s}} e^{\lambda \abs{\abs{k,\eta}^s - \abs{l,\xi}^s}} d\eta d\xi.   
\end{align*} 
Since on the support of the integrand (see \S\ref{Apx:LPProduct}),  
\begin{subequations} \label{ineq:TransFreqCon}
\begin{align} 
\abs{\abs{k,\eta} - \abs{l,\xi}} & \leq \abs{k-l,\eta-\xi} \leq \frac{6}{32}\abs{l,\xi}, \\
(26/32)\abs{l,\xi} &\leq \abs{k,\eta} \leq (38/32)\abs{l,\xi},  
\end{align} 
\end{subequations}
inequalities \eqref{lem:scon} and \eqref{ineq:SobExp} imply, for some $c \in (0,1)$, 
\begin{align*} 
\abs{T_{N;1}} & \lesssim \lambda\sum_{k,l}\int_{\eta,\xi} \abs{A\hat{f}_k(\eta)} \abs{\hat{u}_{k-l}(\eta-\xi)_{<N/8}} \abs{l,\xi}^{s/2}\abs{k,\eta}^{s/2} A_l(\xi)\abs{\hat{f}_l(\xi)_{N}}e^{c\lambda\abs{k-l,\eta - \xi}^s} d\eta d\xi. 
\end{align*} 
Hence \eqref{ineq:L2L2L1} implies (since $\sigma > 6$), 
\begin{align*} 
\abs{T_{N;1}} & \lesssim \lambda\norm{\abs{\grad}^{s/2}Af_{\sim N}}_2\norm{\abs{\grad}^{s/2} Af_{N}}_2\norm{u_{<N/8}}_{\G^{\lambda;\sigma-4}}.
\end{align*} 
Therefore by Lemma \ref{lem:LossyElliptic} and the bootstrap hypotheses,
\begin{align} 
\abs{T_{N;1}} & \lesssim \epsilon\frac{\lambda}{\jap{t}^{2-K_D\epsilon/2}}\norm{\abs{\grad}^{s/2}Af_{\sim N}}_2\norm{\abs{\grad}^{s/2} Af_{N}}_2. \label{ineq:TN1}
\end{align} 

\subsection{Term $T_{N;2}$: effect of $J$}
The most difficult of the three terms in $T_{N}$ is $T_{N;2}$ since $J$ is sensitive to where it is being evaluated in $(t,k,\eta)$.
We divide the integral as follows
\begin{align*} 
T_{N;2} &= i\sum_{k,l}\int_{\eta,\xi} \left[\chi^S +\chi^{L}\right]A\bar{\hat{f}}_k(\eta) \hat{u}_{k-l}(\eta-\xi)_{<N/8} \cdot (l,\xi)A_l(\xi)\hat{f}_l(\xi)_{N} \\ 
& \quad\quad\quad \times e^{\lambda\abs{k,\eta}^s - \lambda \abs{l,\xi}^s}\left[\frac{J_k(\eta)}{J_l(\xi)} - 1\right]\frac{\jap{k,\eta}^\sigma}{\jap{l,\xi}^{\sigma}} d\eta d\xi \\ 
& = T_{N;2}^{S} + T_{N;2}^{L}, 
\end{align*}
where $\chi^{S} = \mathbf{1}_{t \leq \frac{1}{2}\min(\sqrt{\abs{\xi}},\sqrt{\abs{\eta}})}$ and $\chi^L =  1- \chi^S$. 

Focus first on $T_{N;2}^{S}$. 
In this term we apply Lemma \ref{lem:JTrans} to gain $1/2$ derivatives. 
Indeed, on the support of the integrand, \eqref{ineq:TransFreqCon} holds and hence by \eqref{lem:scon} and Lemma \ref{lem:JTrans} we deduce 
\begin{align*} 
\abs{T_{N;2}^{S}} & \lesssim \sum_{k,l}\int_{\eta,\xi} \chi^{S}\abs{A\hat{f}_k(\eta)} \abs{l,\xi}^{1/2} \abs{\hat{u}_{k-l}(\eta-\xi)_{<N/8}} A_l(\xi)\abs{\hat{f}_l(\xi)_{N}} \\ 
& \quad \quad \times \jap{k-l,\eta - \xi}e^{c\lambda\abs{k-l,\eta - \xi}^{s}} d\xi d\eta.   
\end{align*}
Since $c < 1$ it follows by \eqref{ineq:TransFreqCon}, \eqref{ineq:IncExp} and \eqref{ineq:SobExp} that 
\begin{align*} 
\abs{T_{N;2}^{S}} & \lesssim \sum_{k,l}\int_{\eta,\xi} \chi^{S}\abs{A\hat{f}_k(\eta)} \abs{l,\xi}^{1/2} \abs{\hat{u}_{k-l}(\eta-\xi)_{<N/8}} A_l(\xi)\abs{\hat{f}_l(\xi)_{N}} e^{\lambda\abs{k-l,\eta - \xi}^{s}} d\xi d\eta \\ 
 & \lesssim \sum_{k,l}\int_{\eta,\xi} \chi^{S} \abs{A\hat{f}_k(\eta)} \left(1 +  \abs{l,\xi}^{s/2}\abs{k,\eta}^{s/2}\right) \abs{\hat{u}_{k-l}(\eta-\xi)_{<N/8}} A_l(\xi)\abs{\hat{f}_l(\xi)_{N}} e^{\lambda\abs{k-l,\eta - \xi}^{s}} d\xi d\eta. 
\end{align*}
Hence by \eqref{ineq:L2L2L1} followed by the bootstrap hypotheses and Lemma \ref{lem:LossyElliptic},   
\begin{align} 
\abs{T_{N;2}^{S}} & \lesssim \norm{u_{<N/8}}_{\G^{\lambda, \sigma - 4}}\norm{\abs{\grad}^{s/2}Af_{\sim N}}_2\norm{\abs{\grad}^{s/2}Af_{N}}_2 + \norm{u_{<N/8}}_{\G^{\lambda, \sigma - 4}}\norm{Af_{\sim N}}_2\norm{Af_{N}}_2 
\nonumber \\ 
 & \lesssim \frac{\epsilon}{\jap{t}^{2-K_D\epsilon/2}}\norm{\abs{\grad}^{s/2}Af_{\sim N}}_2^2 + \frac{\epsilon}{\jap{t}^{2-K_D\epsilon/2}}\norm{Af_{\sim N}}_2^2. \label{ineq:TNs}
\end{align} 

Now focus on the more difficult $T_{N;2}^{L}$, where the resonant and non-resonant modes are being assigned slightly different regularities. 
There is a potential problem if Lemma \ref{lem:Jswap} incurs a loss. 
Hence we divide into the two natural cases:
\begin{align*} 
T_{N;2}^{L} & =  i\sum_{k,l}\int_{\eta,\xi}\chi^L A_k(\eta)\bar{\hat{f}}_k(\eta) \hat{u}_{k-l}(\eta-\xi)_{<N/8} \cdot (l,\xi) \left[\chi^D + \chi^{\ast}\right] A_l(\xi) \\ &\quad\quad\quad \times \hat{f}_l(\xi)_{N} e^{\lambda\abs{k,\eta}^s - \lambda \abs{l,\xi}^s}\left[\frac{J_k(\eta)}{J_l(\xi)} - 1\right] d\eta d\xi \\ 
& =  T_{N;2}^{D} + T_{N;2}^{\ast}, 
\end{align*} 
where $\chi^{D} = \mathbf{1}_{t \in \I_{k,\eta}}\mathbf{1}_{t \in \I_{k,\xi}}\mathbf{1}_{k \neq l}$ and $\chi^{\ast} = 1 - \chi^D$ (`D' for `difficult'). 
  
First focus on $T_{N;2}^{D}$ which is expected to be challenging.
We will throw away any possible gain from the $-1$ and apply \eqref{ineq:RatJ2partt}; the time decay will make this unimportant.
By \eqref{ineq:TransFreqCon} and \eqref{lem:scon}, 
\begin{align*} 
\abs{T_{N;2}^{D}} & \lesssim \sum_{k,l\neq 0}\int_{\eta,\xi} \abs{A\hat{f}_k(\eta)}\abs{\hat{u}_{k-l}(\eta-\xi)_{<N/8}} \abs{l,\xi} \chi^{D} A_l(\xi)\abs{\hat{f}_l(\xi)_{N}} \\ 
& \quad\quad \times e^{c\lambda\abs{k-l,\eta-\xi}^s}\frac{\abs{\eta}}{k^2}\sqrt{\frac{\partial_t w_k(t,\eta)}{w_k(t,\eta)}}\sqrt{\frac{\partial_t w_l(t,\xi)}{w_l(t,\xi)}}e^{20\mu\abs{k-l,\eta - \xi}^{1/2}} d\eta d\xi. 
\end{align*} 
Applying \eqref{ineq:IncExp} implies 
\begin{align*} 
\abs{T_{N;2}^{D}} & \lesssim \sum_{k,l\neq 0}\int_{\eta,\xi} \abs{A\hat{f}_k(\eta)} \abs{\hat{u}_{k-l}(\eta-\xi)_{<N/8}} \chi^{D} A_l(\xi)\abs{\hat{f}_l(\xi)_{N}} \abs{l,\xi}\\ 
&\quad\quad \times e^{\lambda\abs{k-l,\eta-\xi}^s} \frac{\abs{\eta}}{k^2}\sqrt{\frac{\partial_t w_k(t,\eta)}{w_k(t,\eta)}}\sqrt{\frac{\partial_t w_l(t,\xi)}{w_l(t,\xi)}} d\eta d\xi.  
\end{align*}
On the support of the integrand $A_l(\xi) \lesssim \tilde A_l(\xi)$ and $A_k(\eta) \lesssim \tilde A_k(\eta)$ and since $1 \leq t \approx \frac{\eta}{k}$, 
\begin{align*} 
\abs{T_{N;2}^{D}} & \lesssim \sum_{k,l\neq 0}\int_{\eta,\xi} \abs{\tilde A\hat{f}_k(\eta)} t^2\abs{\hat{u}_{k-l}(\eta-\xi)_{<N/8}} \chi^{D} \tilde A_l(\xi)\abs{\hat{f}_l(\xi)_{N}} \\ 
& \quad\quad\quad \times \frac{\abs{l,\xi}}{\abs{\eta}}e^{\lambda\abs{k-l,\eta-\xi}^s}\sqrt{\frac{\partial_t w_k(t,\eta)}{w_k(t,\eta)}}\sqrt{\frac{\partial_t w_l(t,\xi)}{w_l(t,\xi)}}d\eta d\xi.  
\end{align*}
By the definition of $\chi^D$ and \eqref{ineq:TransFreqCon}, we have $\abs{l,\xi} \lesssim \abs{\eta}$, hence \eqref{ineq:L2L2L1} implies: 
\begin{align*} 
\abs{T_{N;2}^{D}} & \lesssim t^2\norm{P_{\neq 0}u_{<N/8}}_{\G^{\lambda,\sigma-4}}\norm{\sqrt{\frac{\partial_t w}{w}}\tilde Af_{\sim N}}_2\norm{\sqrt{\frac{\partial_t w}{w}}\tilde Af_N}_2, 
\end{align*}  
where note that since $k \neq l$ by the definition of $\chi^D$, we may restrict to non-zero modes in $u$, crucial to get the full $O(t^{-2})$ decay. 
Therefore by Lemma \ref{lem:LossyElliptic} and the bootstrap hypotheses, 
\begin{align} 
\abs{T_{N;2}^{D}} & \lesssim \epsilon\norm{\sqrt{\frac{\partial_t w}{w}}\tilde Af_{\sim N}}_2\norm{\sqrt{\frac{\partial_t w}{w}}\tilde Af_N}_2. \label{ineq:TNRNR}
\end{align} 
This completes the treatment of $T_{N;2}^{D}$. 

It remains to treat $T_{N;2}^{\ast}$. 
We divide into two cases based on the relative size of $\abs{l}$ and $\abs{\xi}$:
\begin{align*} 
T_{N;2}^{\ast} & =  i\sum_{k,l}\int_{\eta,\xi}\chi^L A_k(\eta)\bar{\hat{f}}_k(\eta) \hat{u}_{k-l}(\eta-\xi)_{<N/8} \cdot (l,\xi)\chi^{\ast}
 \left[\mathbf{1}_{\abs{l} > 100\abs{\xi}} + \mathbf{1}_{\abs{l} \leq 100\abs{\xi}}\right]
A_l(\xi)\hat{f}_l(\xi)_{N} \\ 
& \quad\quad\quad \times e^{\lambda\abs{k,\eta}^s - \lambda \abs{l,\xi}^s}\left[\frac{J_k(\eta)}{J_l(\xi)} - 1\right] d\eta d\xi \\ 
& =  T_{N;2}^{\ast,z} + T_{N;2}^{\ast,v}.  
\end{align*}
First consider $T_{N;2}^{\ast,z}$.
On the support of the integrand, note $\abs{\eta} < \frac{313}{1000}\abs{l}$ and hence by \eqref{ineq:TrivDiff},
 \eqref{w-grwth} and \eqref{ineq:SobExp}, 
\begin{align*} 
\abs{\frac{J_k(\eta)}{J_l(\xi)} - 1} = 
\abs{\frac{w_k(t,\eta)^{-1} e^{\mu\abs{\eta}^{1/2}} + e^{\mu\abs{k}^{1/2}}}{w_l(t,\xi)^{-1}e^{\mu\abs{\xi}^{1/2}} +  e^{\mu\abs{l}^{1/2}}} - 1}
 & \lesssim e^{\frac{3}{2}\mu\abs{\eta}^{1/2} - \mu\abs{l}^{1/2}} + \abs{e^{\mu\abs{k}^{1/2} - \mu\abs{l}^{1/2}} - 1} \\
& \lesssim_{\mu} \frac{1}{\abs{l}^{1/2}} +  \frac{\abs{k - l}}{\abs{k}^{1/2} + \abs{l}^{1/2}} e^{\mu\abs{k - l}^{1/2}}.
\end{align*}
Therefore by \eqref{ineq:IncExp}, \eqref{ineq:SobExp} and $\abs{l,\xi} \lesssim \abs{l}$,
\begin{align*} 
\abs{T_{N;2}^{\ast,z}} & \lesssim \sum_{k,l}\int_{\eta,\xi} \chi^L\chi^{\ast} \abs{A\hat{f}_k(\eta)} \mathbf{1}_{\abs{l} > 100\abs{\xi}} \abs{l}^{1/2} \abs{\hat{u}_{k-l}(\eta-\xi)_{<N/8}} A_l(\xi)\abs{\hat{f}_l(\xi)_{N}} e^{\lambda\abs{k-l,\eta-\xi}^s}d\eta d\xi \\ 
& \lesssim \sum_{k,l}\int_{\eta,\xi} \chi^L\chi^{\ast} \abs{A\hat{f}_k(\eta)} \mathbf{1}_{\abs{l} > 100\abs{\xi}} \abs{l}^{s} \abs{\hat{u}_{k-l}(\eta-\xi)_{<N/8}} A_l(\xi)\abs{\hat{f}_l(\xi)_{N}} e^{\lambda\abs{k-l,\eta-\xi}^s}d\eta d\xi.
\end{align*}
Applying \eqref{ineq:L2L2L1}, the bootstrap hypotheses and Lemma \ref{lem:LossyElliptic} as in the treatment of $T_{N;2}^{S}$,
\begin{align}
\abs{T_{N;2}^{\ast,z}} & \lesssim \frac{\epsilon}{\jap{t}^{2-K_D\epsilon/2}}\norm{\abs{\grad}^{s/2}Af_{\sim N}}_2^2. \label{ineq:TNastl}
\end{align}
Turn now to $T_{N;2}^{\ast,v}$. Note that on the support of the integral, $\abs{\eta} \approx \abs{\xi}$.  
By definition of $\chi^{\ast}$, we may apply \eqref{ineq:BasicJswap}, which together with \eqref{lem:scon} and \eqref{ineq:IncExp} implies 
\begin{align*} 
\abs{T_{N;2}^{\ast,v}} & \lesssim \sum_{k,l}\int_{\eta,\xi}\chi^L \chi^{\ast} \mathbf{1}_{\abs{l} \leq 100\abs{\xi}} \abs{A\hat{f}_k(\eta)} \abs{\hat{u}_{k-l}(\eta-\xi)_{<N/8}} \abs{l,\xi} A_l(\xi)\abs{\hat{f}_l(\xi)_{N}} e^{\lambda\abs{k-l,\eta-\xi}^s} d\eta d\xi.  
\end{align*} 
Since $\abs{l,\xi} \lesssim \abs{\xi} \lesssim t^2$, we have $\abs{l,\xi} \lesssim \abs{k,\eta}^{s/2}\abs{l,\xi}^{s/2} t^{2-2s}$
and therefore, 
\begin{align*} 
\abs{T_{N;2}^{\ast,v}} & \lesssim \sum_{k,l}\int_{\eta,\xi} \chi^\ast\chi^L \abs{k,\eta}^{s/2}\abs{A\hat{f}_k(\eta)} t^{2-2s}\abs{\hat{u}_{k-l}(\eta-\xi)_{<N/8}} \abs{l,\xi}^{s/2} A_l(\xi)\abs{\hat{f}_l(\xi)_{N}} e^{\lambda\abs{k-l,\eta-\xi}^s} d\eta d\xi. 
\end{align*}
By \eqref{ineq:L2L2L1}, Lemma \ref{lem:LossyElliptic} and the bootstrap hypotheses, 
\begin{align}
\abs{T_{N;2}^{\ast,v}} \lesssim \frac{\epsilon}{\jap{t}^{2s-K_D\epsilon/2}}\norm{\abs{\grad}^{s/2}Af_{\sim N}}_2^2, \label{ineq:TNastv}
\end{align}
where note $2s - K_D\epsilon/2 \geq s + 1/2$ for $\epsilon$ sufficiently small. 

Combining \eqref{ineq:TNs},\eqref{ineq:TNRNR},\eqref{ineq:TNastl} and \eqref{ineq:TNastv} 
we have for $\epsilon$ sufficiently small,  
\begin{align}
\abs{T_{N;2}} \lesssim \frac{\epsilon}{\jap{t}^{s+1/2}}\norm{\abs{\grad}^{s/2}Af_{\sim N}}_2^2 + \epsilon\norm{\sqrt{\frac{\partial_t w}{w}}\tilde Af_{\sim N}}_2^2 + \frac{\epsilon}{\jap{t}^{2-K_D\epsilon/2}}\norm{Af_{\sim N}}_2^2, \label{ineq:TN2a}
\end{align}  
which completes the treatment of $T_{N;2}$. 

\subsection{Term $T_{N;3}$: Sobolev correction}
Next, turn to $T_{N;3}$ which is the easiest to treat. 
By the mean value theorem and  \eqref{ineq:TransFreqCon},
\begin{align*} 
\abs{\frac{\jap{k,\eta}^\sigma}{\jap{l,\xi}^\sigma} - 1 } \lesssim \frac{\abs{k-l,\eta - \xi}}{\jap{l,\xi}},  
\end{align*} 
which implies 
arguments similar to those applied above can deduce
\begin{align} 
\abs{T_{N;3}} & \lesssim \frac{\epsilon}{\jap{t}^{2 - K_D \epsilon/2}}\norm{Af_{\sim N}}_2\norm{Af_N}_2.  \label{ineq:TN2b}
\end{align} 
Indeed, putting \eqref{ineq:TN1}, \eqref{ineq:TN2a} and \eqref{ineq:TN2b} together with \eqref{ineq:GeneralOrtho} proves Proposition \ref{prop:Transport}.

\section{Reaction}  \label{sec:RigReac}
Focus first on an individual frequency shell and divide each one into several natural pieces 
\begin{align*} 
R_N & = R_N^1 + R_N^{\epsilon,1} + R_N^2 + R_N^3   
\end{align*}
where 
\begin{align*} 
R_N^1 & = \sum_{k,l\neq 0}\int_{\eta,\xi} A \bar{\hat f}_k(\eta) A_k(\eta) (\eta l - \xi k) \hat{\phi}_l(\xi)_N \hat{f}_{k-l}(\eta - \xi)_{< N/8} d\eta d\xi \\ 
R_N^{\epsilon,1} & = -\sum_{k,l\neq 0}\int_{\eta,\xi} A \bar{\hat f}_k(\eta) A_k(\eta)(\eta l - \xi k) \left[\widehat{(1-v^\prime)\phi}_l\right](\xi)_N \hat{f}_{k-l}(\eta - \xi)_{< N/8} d\eta d\xi \\ 
R_N^2 & = \sum_{k}\int_{\eta,\xi} A \bar{\hat f}_k(\eta) A_k(\eta)\widehat{[\partial_t v]}(\xi)_N \widehat{\partial_v f}_{k}(\eta - \xi)_{< N/8} d\eta d\xi \\ 
R_N^3 & = -\sum_{k,l}\int_{\eta,\xi} A \bar{\hat f}_k(\eta) A_{k-l}(\eta-\xi)\hat{u}_l(\xi)_N\widehat{\grad f}_{k-l}(\eta - \xi)_{<N/8} d\eta d\xi. 
\end{align*} 

\subsection{Main contribution} \label{sec:RN1} 
The main contribution comes from $R_N^1$. 
We subdivide this integral depending on whether or not $(l,\xi)$ and/or $(k,\eta)$ are resonant as each combination requires a slightly different treatment.  
Define the partition: 
\begin{align*} 
1 &= \mathbf{1}_{t \not\in\I_{k,\eta}, t \not\in\I_{l,\xi}} +\mathbf{1}_{t \not\in\I_{k,\eta}, t \in\I_{l,\xi}} + \mathbf{1}_{t \in\I_{k,\eta}, t \not\in\I_{l,\xi}} + \mathbf{1}_{t \in\I_{k,\eta}, t \in\I_{l,\xi}} \\ 
& = \chi^{NR,NR} + \chi^{NR,R} + \chi^{R,NR} + \chi^{R,R}, 
\end{align*}
where the NR and R denotes `non-resonant' and `resonant' respectively referring to $(k,\eta)$ and $(l,\xi)$. 
Correspondingly, denote
\begin{align*} 
R_N^{1} & = \sum_{k,l\neq 0}\int_{\eta,\xi} \left[\chi^{NR,NR} + \chi^{NR,R} + \chi^{R,NR} + \chi^{R,R}\right] A \bar{\hat f}_k(\eta) A_k(\eta) (\eta l - \xi k) \hat{\phi}_l(\xi)_N \hat{f}_{k-l}(\eta - \xi)_{< N/8} d\eta d\xi \\ 
& = R_N^{NR,NR} + R_N^{NR,R} + R_N^{R,NR} + R_N^{R,R}.
\end{align*}

\subsubsection{Treatment of $R_N^{NR,NR}$}
Since on the support of the integrand of $R_N^{1}$,
\begin{align} 
\abs{\abs{l,\xi} - \abs{k,\eta}} \leq \abs{k-l,\eta - \xi} \leq \frac{6}{32}\abs{l,\xi}, \label{ineq:RN1lxictrl} 
\end{align}
it follows from \eqref{lem:scon} that for some $c \in (0,1)$,
\begin{align*} 
\abs{R_N^{NR,NR}} & \leq \sum_{k,l\neq 0}\int_{\eta,\xi} \abs{\chi^{NR,NR} A \bar{\hat f}_k(\eta) J_{k}(\eta) e^{\lambda\abs{l,\xi}^s}e^{c\lambda\abs{k-l,\eta-\xi}^s}} \\ & \quad\quad \times \jap{k,\eta}^\sigma\abs{\eta l - \xi k} \abs{\hat{\phi}_l(\xi)_N \hat{f}_{k-l}(\eta - \xi)_{< N/8}} d\eta d\xi. 
\end{align*} 
Moreover, by \eqref{ineq:BasicJswap}, \eqref{ineq:RN1lxictrl} (which implies $\abs{k,\eta} \approx \abs{l,\xi}$)  and \eqref{ineq:IncExp},
\begin{align*} 
\abs{R_N^{NR,NR}} & \lesssim  \sum_{k,l\neq 0}\int_{\eta,\xi}\chi^{NR,NR} \abs{A\hat f_k(\eta)} e^{\lambda\abs{k-l,\eta-\xi}^s} \abs{\eta l - \xi k} A_l(\xi) \abs{\hat{\phi}_l(\xi)_N \hat{f}_{k-l}(\eta - \xi)_{< N/8}} d\eta d\xi.  
\end{align*}
Again by \eqref{ineq:RN1lxictrl}, $\abs{\eta l - \xi k} \lesssim \abs{l,\xi}^{1-s/2}\abs{k,\eta}^{s/2}\abs{k-l,\eta-\xi}$ which implies by \eqref{ineq:L2L2L1}, 
\begin{align} 
\abs{R_N^{NR,NR}} & \lesssim \sum_{k,l\neq 0}\int_{\eta,\xi} \abs{k,\eta}^{s/2} \chi^{NR,NR}\abs{A \hat f_k(\eta)}  \abs{l,\xi}^{1-s/2}A_l(\xi)\abs{\hat{\phi}_l(\xi)_N} \\ & \quad\quad \times \abs{k-l,\eta-\xi} \abs{\hat{f}_{k-l}(\eta - \xi)_{< N/8}} e^{\lambda\abs{k-l,\eta-\xi}^s} d\eta d\xi \nonumber \\
& \lesssim \norm{\abs{\grad}^{s/2}Af_{\sim N}}_2 \norm{\abs{\grad}^{1-s/2}AP_{\neq 0}\chi^{NR}\phi_N}_2\norm{f}_{\G^{\lambda,\sigma}} \nonumber \\
& \lesssim \epsilon\norm{\abs{\grad}^{s/2}Af_{\sim N}}_2 \norm{\abs{\grad}^{1-s/2}AP_{\neq 0}\chi^{NR}\phi_N}_2, \label{ineq:RNNRNR}
\end{align} 
where the last line followed from the bootstrap hypotheses.
Here we are denoting $\chi^{NR}f$ the multiplier $\widehat{\chi^{NR}f}(t,l,\xi) = \mathbf{1}_{t \not\in\I_{l,\xi}} \hat{f}_{l}(t,\xi)$.  

\subsubsection{Treatment of $R_N^{R,NR}$}
Next we turn to $R_N^{R,NR}$ which is one of the terms $w$ was designed to treat. 
Physically, it describes the action of the non-resonant modes on the resonant modes. 
By \eqref{ineq:RN1lxictrl} and \eqref{lem:scon}, for some $c \in (0,1)$, 
\begin{align*} 
\abs{R_N^{R,NR}} \lesssim \sum_{k,l\neq 0}\int_{\eta,\xi}\chi^{R,NR} \abs{A \hat f_k(\eta)} J_{k}(\eta) e^{\lambda\abs{l,\xi}^s}e^{c\lambda\abs{k-l,\eta-\xi}^s}\jap{l,\xi}^\sigma\abs{l,\xi}\abs{\hat{\phi}_l(\xi)_N \widehat{\grad f}_{k-l}(\eta - \xi)_{< N/8}} d\eta d\xi.
\end{align*}  
Consider separately the following cases: 
\begin{align*} 
\abs{R_N^{R,NR}} & \lesssim \sum_{k,l\neq 0}\int_{\eta,\xi}\chi^{R,NR}\left[\mathbf{1}_{t \in\I_{k,\xi}} + \mathbf{1}_{t \not\in\I_{k,\xi}} \right] \abs{A \hat f_k(\eta)} J_{k}(\eta) \\ & \quad\quad \times e^{\lambda\abs{l,\xi}^s}e^{c\lambda\abs{k-l,\eta-\xi}^s}\jap{l,\xi}^\sigma\abs{l,\xi}\abs{\hat{\phi}_l(\xi)_N \widehat{\grad f}_{k-l}(\eta - \xi)_{< N/8}} d\eta d\xi \\ 
& = R_N^{R,NR;D} + R_N^{R,NR;\ast}. 
\end{align*}  
The toy model is adapted to treat $R_N^{R,NR;D}$, so consider this first. 
Note that on the support of the integrand in this case, we have $\abs{\eta} \approx \abs{\xi}$. 
Therefore, applying \eqref{ineq:RatJ2partt} and \eqref{ineq:IncExp}, 
\begin{align*} 
R_N^{R,NR;D} & \lesssim \sum_{k,l\neq 0}\int_{\eta,\xi}\chi^{R,NR} \mathbf{1}_{t \in\I_{k,\xi}} \abs{A \hat f_k(\eta)} J_{l}(\xi)\frac{\abs{\eta}}{k^2}\sqrt{\frac{\partial_t w_k(t,\eta)}{w_k(t,\eta)}}\sqrt{\frac{\partial_t w_l(t,\xi)}{w_l(t,\xi)}} \\
& \quad\quad \times e^{\lambda\abs{l,\xi}^s}e^{\lambda\abs{k-l,\eta-\xi}^s}\jap{l,\xi}^\sigma\abs{l,\xi}\abs{\hat{\phi}_l(\xi)_N \widehat{\grad f}_{k-l}(\eta - \xi)_{< N/8}} d\eta d\xi.
\end{align*}
Since $l^2k^{-2} \leq \jap{l-k}^2$, $\abs{\eta} \approx \abs{\xi}$ and $\abs{l} < \frac{1}{4}\abs{\xi}$ (hence $J_l(\xi)\lesssim \tilde J_l(\xi)$), 
\begin{align*} 
R_N^{R,NR;D} & \lesssim \sum_{k,l\neq 0}\int_{\eta,\xi}\chi^{R,NR} \mathbf{1}_{t \in\I_{k,\xi}}  \abs{\tilde A\hat f_k(\eta)} \tilde J_{l}(\xi) \frac{\abs{\xi}^2}{l^2\jap{k-l}^2}\sqrt{\frac{\partial_t w_k(t,\eta)}{w_k(t,\eta)}}\sqrt{\frac{\partial_t w_l(t,\xi)}{w_l(t,\xi)}} \\
& \quad\quad \times e^{\lambda\abs{l,\xi}^s}e^{\lambda\abs{k-l,\eta-\xi}^s}\jap{l,\xi}^\sigma\abs{\hat{\phi}_l(\xi)_N \jap{k-l}^4 \widehat{\grad f}_{k-l}(\eta - \xi)_{< N/8}} d\eta d\xi.
\end{align*}
Applying $\abs{k - t^{-1}\eta} \leq 1$, \eqref{ineq:L2L2L1} and the bootstrap hypotheses (denoting $\chi^r(t,\eta) = \mathbf{1}_{2\sqrt{\abs{\eta}} \leq t \leq 2\abs{\eta}}$): 
\begin{align} 
R_N^{R,NR;D}   & \lesssim \epsilon\norm{\sqrt{\frac{\partial_t w}{w}}Af_{\sim N}}_2\norm{\sqrt{\frac{\partial_t w}{w}}\frac{\abs{\partial_v}^2}{\partial_z^2\jap{t^{-1}\partial_v-\partial_z}^2}\chi^{r}\chi^{NR}\tilde AP_{\neq 0}\phi_N}_2. \label{ineq:RNRNRD} 
\end{align}

Turn now to $R_N^{R,NR;\ast}$.  
In this case we may apply \eqref{ineq:BasicJswap} (as opposed to \eqref{ineq:WFreqCompRes}) and hence we can exchange $J_k(\eta)$ for $J_l(\xi)$ without incurring a major cost.  
Therefore, by applying the same argument used to treat $R_N^{NR,NR}$ we deduce: 
\begin{align} 
R_N^{R,NR;\ast} \lesssim \epsilon \norm{\abs{\grad}^{s/2}Af_{\sim N}}_2\norm{\abs{\grad}^{1-s/2}\chi^{NR}AP_{\neq 0}\phi_N}_2. \label{ineq:RNRNRast}
\end{align}

\subsubsection{Treatment of $R_N^{NR,R}$}
The next term we treat is $R_N^{NR,R}$, in which case $(k,\eta)$ is non-resonant and $(l,\xi)$ is resonant. 
It follows that $4\abs{l}^2 \leq \abs{\xi}$ and since $N \geq 8$, \eqref{ineq:RN1lxictrl} implies 
$N/4 \leq \abs{\xi} \leq 3N/2$ and  $\abs{\eta} \approx \abs{\xi}$. 
By \eqref{ineq:partialtw_endpt}, 
\begin{align} 
1 & \lesssim \sqrt{\frac{w_l(t,\xi)}{\partial_t w_l(t,\xi)}}\left[\sqrt{\frac{\partial_t w_k(t,\eta)}{w_k(t,\eta)}} + \frac{\abs{k,\eta}^{s/2}}{\jap{t}^{s}}\right]\jap{\eta-\xi}. \label{ineq:NRRpartialtw}
\end{align}
Applying \eqref{ineq:NRRpartialtw}, \eqref{ineq:WFreqCompNRGain}, \eqref{lem:scon} (using \eqref{ineq:RN1lxictrl}) \eqref{ineq:SobExp} and \eqref{ineq:IncExp} ($s > 1/2$), 
\begin{align*} 
\abs{R_N^{NR,R}} & \lesssim \sum_{k,l\neq 0}\int_{\eta,\xi}\chi^{NR,R}\left[\sqrt{\frac{\partial_t w_k(\eta)}{w_k(\eta)}} + \frac{\abs{k,\eta}^{s/2}}{\jap{t}^{s}}\right]\abs{A \hat f_k(\eta)} J_{l}(\xi) \frac{w_R(\xi)}{w_{NR}(\xi)}\sqrt{\frac{w_l(t,\xi)}{\partial_t w_l(t,\xi)}} \\
& \quad\quad \times e^{\lambda\abs{l,\xi}^s}e^{\lambda\abs{k-l,\eta-\xi}^s}\jap{l,\xi}^\sigma\abs{l,\xi}\abs{\hat{\phi}_l(\xi)_N \widehat{\grad f}_{k-l}(\eta - \xi)_{< N/8}} d\eta d\xi. 
\end{align*} 
On the support of the integrand, it follows from \eqref{ineq:RN1lxictrl} that $\abs{k} < \abs{\eta}$ and hence $A_k(\eta) \lesssim \tilde A_k(\eta)$. Similarly, $A_l(\xi) \lesssim \tilde A_l(\xi)$. Therefore \eqref{ineq:L2L2L1} and the bootstrap hypotheses imply
\begin{align} 
\abs{R_N^{NR,R}} & \lesssim \epsilon\left(\norm{\sqrt{\frac{\partial_t w}{w}} \tilde Af_{\sim N}}_2 + \frac{1}{\jap{t}^{s}}\norm{\abs{\grad}^{s/2}Af_{\sim N}}_2\right)\norm{\sqrt{\frac{w}{\partial_t w}}\abs{\grad}\frac{w_{R}}{w_{NR}} \chi^{R} \tilde A \phi_N}_2. \label{ineq:RNNRR}
\end{align}

\subsubsection{Treatment of $R_N^{R,R}$}
In this case both $(k,\eta)$ and $(l,\xi)$ are resonant, an interaction that was neglected in the derivation of the toy model.  
We claim that on the support of the integrand of $R_N^{R,R}$: 
\begin{align}
\abs{\eta l - \xi k}\frac{J_k(\eta)}{J_l(\xi)} \lesssim \sqrt{\frac{\partial_t w_k(t,\eta)}{w_k(t,\eta)}} \left[ \abs{l,\xi}\frac{w_R(t,\xi)}{w_{NR}(t,\xi)}  + \abs{l}\right] \sqrt{\frac{w_l(t,\xi)}{\partial_t w_l(t,\xi)}} e^{12\mu\abs{k-l,\eta-\xi}^{1/2}}. \label{ineq:JRRCtrl}
\end{align}
Indeed, if $k = l$ then \eqref{ineq:BasicJswap} and Lemma \ref{lem:WtFreqCompare} imply, 
\begin{align*} 
\abs{l}\abs{\eta - \xi}\frac{J_l(\eta)}{J_l(\xi)} \lesssim \abs{l}\jap{\eta - \xi}^2\sqrt{\frac{\partial_t w_k(t,\eta)}{w_k(t,\eta)}}\sqrt{\frac{w_l(t,\xi)}{\partial_t w_l(t,\xi)}} e^{10\mu\abs{k-l,\eta-\xi}^{1/2}}, 
\end{align*}
from which \eqref{ineq:JRRCtrl} follows by \eqref{ineq:SobExp}. 
If $k \neq l$ then as in the proof of \eqref{ineq:WFreqCompNRGain} we apply Lemma \ref{lem:wellsep}.
If Lemma \ref{lem:wellsep} (b) holds then by Lemma \ref{lem:WtFreqCompare}, \eqref{ineq:BasicJswap} (note that on the support of the integrand $\abs{\eta} \approx \abs{\xi}$ by \eqref{ineq:RN1lxictrl} with $k^2 < \frac{1}{4}\abs{\eta}$, $l^2 < \frac{1}{4}\abs{\xi}$) and the definitions \eqref{def:wNR}, \eqref{def:wR}: 
\begin{align*} 
\abs{\eta l - \xi k}\frac{J_k(\eta)}{J_l(\xi)} \lesssim \abs{l,\xi}\jap{k-l,\eta-\xi}^2\sqrt{\frac{\partial_t w_k(t,\eta)}{w_k(t,\eta)}}\sqrt{\frac{w_l(t,\xi)}{\partial_t w_l(t,\xi)}}\frac{w_R(t,\xi)}{w_{NR}(t,\xi)}e^{10\mu\abs{k-l,\eta-\xi}^{1/2}},
\end{align*} 
which again implies \eqref{ineq:JRRCtrl} by \eqref{ineq:SobExp}. 
Finally, if Lemma \ref{lem:wellsep} (c) holds then by Lemma \ref{lem:WtFreqCompare}, \eqref{ineq:BasicJswap} and \eqref{ineq:SobExp}, 
\begin{align*} 
\abs{\eta l - \xi k}\frac{J_k(\eta)}{J_l(\xi)} \lesssim \abs{l,\xi}\abs{k-l,\eta-\xi}e^{9\mu\abs{k-l,\eta-\xi}^{1/2}} \lesssim \abs{l}\sqrt{\frac{\partial_t w_k(t,\eta)}{w_k(t,\eta)}}\sqrt{\frac{w_l(t,\xi)}{\partial_t w_l(t,\xi)}} e^{10\mu\abs{k-l,\eta-\xi}^{1/2}},   
\end{align*} 
which proves \eqref{ineq:JRRCtrl} in the remaining case. 
Applying \eqref{ineq:JRRCtrl}, \eqref{lem:scon} (using \eqref{ineq:RN1lxictrl}) and \eqref{ineq:IncExp} implies
\begin{align*}
\abs{R_N^{R,R}} & \lesssim \sum_{k,l\neq 0}\int_{\eta,\xi} \chi^{R,R}\sqrt{\frac{\partial_t w_k(t,\eta)}{w_k(t,\eta)}} \abs{A \hat f_k(\eta)} \\ & \quad\quad \times \left[ \abs{l,\xi}\frac{w_R(t,\xi)}{w_{NR}(t,\xi)}  + \abs{l}\right] \sqrt{\frac{w_l(t,\xi)}{\partial_t w_l(t,\xi)}} A_l(\xi)\abs{\hat{\phi}_l(\xi)_N} e^{\lambda\abs{k-l,\eta-\xi}^s}\abs{\hat{f}_{k-l}(\eta - \xi)_{< N/8}} d\eta d\xi. 
\end{align*} 
Since $(k,\eta)$ and $(l,\xi)$ are both resonant, $A_k(\eta) \lesssim \tilde A_k(\eta)$ and $A_l(\xi) \lesssim \tilde A_l(\xi)$. 
Then by \eqref{ineq:L2L2L1} and the bootstrap hypotheses, 
\begin{align} 
\abs{R_N^{R,R}} & \lesssim \epsilon \norm{\sqrt{\frac{\partial_t w}{w}} \tilde Af_{\sim N}}_2\left(\norm{\sqrt{\frac{w}{\partial_t w}}\abs{\grad}\frac{w_{R}}{w_{NR}} \chi^{R} \tilde A \phi_N}_2 +  \norm{\sqrt{\frac{w}{\partial_t w}}\abs{\partial_z}\chi^{R} \tilde A \phi_N}_2\right), \label{ineq:RNRRF}
\end{align} 
which completes the treatment of $R_N^{R,R}$. 

\subsubsection{Contribution to \eqref{ineq:ReacIntro}} 
Combining \eqref{ineq:RNNRNR}, \eqref{ineq:RNRNRD}, \eqref{ineq:RNRNRast}, \eqref{ineq:RNNRR}, \eqref{ineq:RNRRF} and Cauchy-Schwarz we deduce, 
\begin{align} 
\abs{R_N^1} & \lesssim \frac{\epsilon}{\jap{t}^{2s}}\norm{\abs{\grad}^{s/2}Af_{\sim N}}^2_2 + \epsilon \norm{\sqrt{\frac{\partial_t w}{w}} Af_{\sim N}}^2_2 \nonumber \\
& \quad + \epsilon \jap{t}^{2s} \norm{\abs{\grad}^{1-s/2}\chi^{NR} AP_{\neq 0}\phi_N}^2_2 
+ \epsilon \norm{\sqrt{\frac{\partial_t w}{w}}\frac{\abs{\partial_v}^2}{\partial_z^2\jap{t^{-1}\partial_v-\partial_z}^2}\chi^r\chi^{NR}\tilde{A}\phi_N}^2_2 \nonumber \\
& \quad + \epsilon\norm{\sqrt{\frac{w}{\partial_t w}}\abs{\grad}\frac{w_{R}}{w_{NR}} \chi^{R} \tilde A \phi_N}^2_2 + \epsilon\norm{\sqrt{\frac{w}{\partial_t w}}\abs{\partial_z}\chi^{R} \tilde A \phi_N}^2_2,\label{ineq:RN1ctrl}
\end{align}
where $\chi^r(t,\eta) = \mathbf{1}_{2\sqrt{\abs{\eta}} \leq t \leq 2\abs{\eta}}$. 
The treatment of $R_N^1$ will then be complete once we have the following lemma to relate the latter four terms to those in \eqref{ineq:ReacIntro}, on which Proposition \ref{lem:PrecisionEstimateNeq0} can be applied.  
The primary complication in \eqref{ineq:ReactionControl} below is the leading factor $\jap{\partial_v(t\partial_z)^{-1}}^{-1}$ which will require some additional care to include.
Recall the presence of this multiplier arises from the $\partial_v$ in the expression of $v^{\prime\prime}$ (see \eqref{def:vpp}) which appears as a coefficient in $\Delta_t$; see \S\ref{sec:Precf} for more information.
This lemma expresses something important: that the multipliers coming from $w$ exactly `match' the loss of ellipticity in $\Delta_L$.

\begin{lemma}
Under the bootstrap hypotheses, 
\begin{subequations} \label{ineq:ReactionControl}
\begin{align} 
\jap{t}^{2s}\norm{\abs{\grad}^{1-s/2}\chi^{NR}AP_{\neq 0}\phi}^2_2 & \lesssim \norm{\jap{\frac{\partial_v}{t\partial_z}}^{-1}\frac{\abs{\grad}^{s/2}}{\jap{t}^s}\Delta_LAP_{\neq 0} \phi}_2^2 \label{ineq:ReacC4} \\  
\norm{\sqrt{\frac{\partial_t w}{w}}\frac{\abs{\partial_v}^2}{\partial_z^2\jap{t^{-1}\partial_v - \partial_z}^2}\chi^r\chi^{NR} \tilde AP_{\neq 0}\phi}^2_2 & \lesssim \norm{\jap{\frac{\partial_v}{t\partial_z}}^{-1}\sqrt{\frac{\partial_t w}{w}}\Delta_L \tilde A P_{\neq 0}\phi}_2^2 \label{ineq:ReacC5} \\    
\norm{\sqrt{\frac{w}{\partial_t w}}\abs{\grad}\frac{w_{R}}{w_{NR}}\chi^{R} \tilde A P_{\neq 0}\phi}^2_2 
& \lesssim \norm{\jap{\frac{\partial_v}{t\partial_z}}^{-1}\sqrt{\frac{\partial_t w}{w}} \Delta_L \tilde A P_{\neq 0}\phi}_2^2 \label{ineq:ReacC6} \\    
\norm{\sqrt{\frac{w}{\partial_t w}} \abs{\partial_z} \chi^{R} \tilde A P_{\neq 0} \phi}^2_2 & \lesssim \norm{\jap{\frac{\partial_v}{t\partial_z}}^{-1}\sqrt{\frac{\partial_t w}{w}}\Delta_L \tilde A P_{\neq 0}\phi}_2^2. \label{ineq:ReacC7}   
\end{align} 
\end{subequations} 
\end{lemma} 
\begin{proof}
Note that \eqref{ineq:ReactionControl} is only a statement about the Fourier multipliers and has nothing really to do with $A\phi$. 
Indeed, \eqref{ineq:ReacC4} follows from the pointwise inequality: for all $t \geq 1$ $l \neq 0$ and $\xi \in \Real$, 
\begin{align}
\jap{t}^{s}\abs{l,\xi}^{1-s/2}\mathbf{1}_{t \not\in\I_{l,\xi}} \lesssim \jap{\frac{\xi}{lt}}^{-1}\left(l^2 + \abs{\xi - lt}^2\right)\frac{\abs{l,\xi}^{s/2}}{\jap{t}^s}. \label{ineq:ptwiseReac}
\end{align}  

\textit{Proof of \eqref{ineq:ptwiseReac}:}\\ 
Consider the case $\frac{1}{2}\abs{lt} \leq \abs{\xi} \leq 2\abs{lt}$. 
By the presence of $\mathbf{1}_{t \not\in\I_{l,\xi}}$, either $\abs{\xi - tl} \gtrsim \abs{\xi/l} \approx t$ (if $t^2 \gtrsim \abs{\xi}$) or $l^2 \gtrsim t$ (if $t^2 \lesssim \abs{\xi}$). 
In either case, 
\begin{align*} 
\jap{t}^{2s}\abs{l,\xi}^{1-s} \lesssim \abs{l}^{1-s}t^{1+s} \leq l^2 + t^2 \lesssim l^2 + \abs{\xi-lt}^2, 
\end{align*}
which implies \eqref{ineq:ptwiseReac}. 
Next consider the case $\abs{\xi} < \abs{lt}/2$, which implies 
\begin{align*} 
\jap{t}^{2s}\abs{l,\xi}^{1-s} \lesssim t^{2s}\abs{lt}^{1-s} \lesssim \abs{lt}^{1+s} \lesssim l^2 + \abs{\xi - lt}^2, 
\end{align*}
which again implies \eqref{ineq:ptwiseReac}. 
Finally consider the case $\abs{\xi} \geq 2\abs{lt}$, in which the leading $\jap{\xi/lt}^{-1}$ plays a role. 
In this case (note since $t \geq 1$, $\abs{\xi} \geq 2\abs{l}$), 
\begin{align*} 
\jap{t}^{2s}\abs{l,\xi}^{1-s} \lesssim \abs{\xi}^{2-s} \jap{t}^{2s-1} \frac{\abs{lt}}{\abs{\xi}} \lesssim \abs{\xi}^{1 + s} \frac{\abs{lt}}{\abs{\xi}} \lesssim \left(l^2 + \abs{\xi-lt}^2\right)\frac{\abs{lt}}{\abs{\xi}}, 
\end{align*}  
which implies \eqref{ineq:ptwiseReac}. As all cases have been covered, this proves \eqref{ineq:ReacC4}.

\textit{Proof of \eqref{ineq:ReacC5}:} \\ 
Inequality \eqref{ineq:ReacC5} follows from the pointwise inequality: Let $j$ be such that $t \in\I_{j,\xi}$, then for all $t \geq 1$, $l \neq 0$ and $\xi \in \Real$ with $2\sqrt{\abs{\xi}} \leq t \leq 2\abs{\xi}$: 
\begin{align} 
\frac{\abs{\xi}^2}{\abs{l}^2\jap{j-l}^2}\mathbf{1}_{t \not\in\I_{l,\xi}} \lesssim \jap{\frac{\xi}{lt}}^{-1}\left(l^2 + \abs{\xi-lt}^2\right). \label{ineq:ptwiseReac5}
\end{align} 
Since $t \not\in\I_{l,\xi}$, $\abs{\xi}^2 \lesssim l^2 \abs{\xi-lt}^2$. 
 If $2\abs{lt} \geq \abs{\xi}$ then the factor $\jap{\frac{\xi}{lt}}^{-1}$ does not play a role and \eqref{ineq:ptwiseReac5} follows immediately. 
Next consider the case $2\abs{lt} \leq \abs{\xi}$, which implies $\abs{\xi-lt} \gtrsim \abs{\xi}$. 
In this case, since $\xi \approx jt$ and $\abs{j} \leq \abs{l}\abs{l-j}$, 
\begin{align*} 
\frac{\abs{\xi}^2}{\abs{l}^2\jap{j-l}^2} \lesssim \frac{\abs{\xi}}{\jap{j-l}^2\abs{lt}}  \abs{\frac{lt}{\xi}}\abs{\xi-lt}^2 \lesssim \frac{\abs{j}}{\abs{l}\jap{j-l}^2}  \abs{\frac{lt}{\xi}}\abs{\xi-lt}^2 \lesssim \abs{\frac{lt}{\xi}}\abs{\xi-lt}^2. 
\end{align*}
This verifies \eqref{ineq:ptwiseReac5} in every case and hence \eqref{ineq:ReacC5}. 

\textit{Proof of \eqref{ineq:ReacC6}:}
By the definitions of $w_{R}$ and $w_{NR}$ \eqref{def:wNR},\eqref{def:wR}  and \eqref{dtw}, 
\begin{align*} 
\sqrt{\frac{w(t,\xi)}{\partial_t w(t,\xi)}}\abs{l,\xi}\frac{w_{R}(t,\xi)}{w_{NR}(t,\xi)}\mathbf{1}_{t \in\I_{l,\xi}} & \lesssim \abs{\xi}\left(1 + \abs{t - \frac{\xi}{l}}\right)^{1/2}\frac{w_{R}(t,\xi)}{w_{NR}(t,\xi)}\mathbf{1}_{t \in\I_{l,\xi}}  \lesssim \left(l^2 + \abs{\xi - tl}^2\right)\sqrt{\frac{\partial_t w(t,\xi)}{w(t,\xi)}}, 
\end{align*}
which proves \eqref{ineq:ReacC6}. 

\textit{Proof of \eqref{ineq:ReacC7}:}
Similarly, by \eqref{dtw}
\begin{align*}
\sqrt{\frac{w(t,\xi)}{\partial_t w(t,\xi)}}\abs{l}\mathbf{1}_{t \in\I_{l,\xi}} & \lesssim \abs{l}\sqrt{1 + \abs{t - \frac{\xi}{l}}}\mathbf{1}_{t \in\I_{l,\xi}}  \lesssim \abs{l}\left(1 + \abs{t - \frac{\xi}{l}}\right)\sqrt{\frac{\partial_t w(t,\xi)}{w(t,\xi)}}, 
\end{align*}
which proves \eqref{ineq:ReacC7}.
\end{proof} 

Finally, Lemma \ref{ineq:ReactionControl} completes the treatment of $R_N^1$; in particular, after summing in $N$, \eqref{ineq:GeneralOrtho}, \eqref{ineq:RN1ctrl} and Lemma \ref{ineq:ReactionControl} yield only terms appearing on the RHS of \eqref{ineq:ReacIntro}. 

\subsection{Corrections} 
\subsubsection{Term $R_N^{\epsilon,1}$: $O(\epsilon)$ correction to $R_N^1$} \label{sec:RNeps}
In this section we treat $R_N^{\epsilon,1}$ which is higher order in $\epsilon$ than $R_N^1$. 
We expand $(1-v^\prime)\phi_l$ with a paraproduct \emph{only in $v$}: 
\begin{align*} 
R_N^{\epsilon,1} & = -\frac{1}{2\pi}\sum_{M \geq 8} \sum_{k,l\neq 0}\int_{\eta,\xi,\xi^\prime} A \bar{\hat f}_k(\eta) A_k(\eta)(\eta l - \xi^\prime k)\rho_N(l,\xi)\left[\widehat{(1-v^\prime)}(\xi^\prime - \xi)\right]_{<M/8} \widehat{\phi}_l(\xi^\prime)_M \hat{f}_{k-l}(\eta - \xi)_{< N/8} d\eta d\xi d\xi^\prime \\ 
&  \quad -\frac{1}{2\pi} \sum_{M \geq 8} \sum_{k,l\neq 0}\int_{\eta,\xi,\xi^\prime} A \bar{\hat f}_k(\eta) A_k(\eta)(\eta l - \xi^\prime k)\rho_N(l,\xi)\left[\widehat{(1-v^\prime)}(\xi^\prime - \xi)\right]_{M} \widehat{\phi}_l(\xi^\prime)_{<M/8} \hat{f}_{k-l}(\eta - \xi)_{< N/8} d\eta d\xi d\xi^\prime \\ 
&  \quad - \frac{1}{2\pi}\sum_{M \in \mathbb D}\sum_{\frac{1}{8}M \leq M^\prime \leq 8M} \sum_{k,l\neq 0}\int_{\eta,\xi,\xi^\prime} A \bar{\hat f}_k(\eta) A_k(\eta)(\eta l - \xi^\prime  k)\rho_N(l,\xi) \\ & \hspace{5cm} \times \left[\widehat{(1-v^\prime)}(\xi^\prime - \xi)\right]_{M^\prime}\widehat{\phi}_l(\xi^\prime)_M \hat{f}_{k-l}(\eta - \xi)_{< N/8} d\eta d\xi d\xi^\prime \\
& = R_{N;LH}^{\epsilon,1} + R_{N;HL}^{\epsilon,1} + R_{N;HH}^{\epsilon,1}. 
\end{align*} 
We recall that $\rho_N$ denotes the Littlewood-Paley cut-off to the $N$-th dyadic shell in $\Integer \times \Real$; see \eqref{Apx:LPProduct}. 
The intuition is as follows: $R_{N;LH}^{\epsilon,1}$ can be treated in a manner very similar to $R_N^{1}$ as here $(1-v^\prime)$ appears essentially as part of the background and $R_{N;HL}^{\epsilon,1}$ should be manageable since $(1-v^\prime)$ is controlled by the bootstrap hypotheses and $\phi_l$ provides decay in time.

Begin first with $R_{N;LH}^{\epsilon,1}$. 
On the support of the integrand, 
\begin{subequations} \label{ineq:ketalxiRepsilon}
\begin{align} 
\abs{ \abs{k,\eta} - \abs{l,\xi} } & \leq \abs{k-l,\eta-\xi} \leq \frac{3}{16}\abs{l,\xi}, \\
\abs{ \abs{l,\xi^\prime} - \abs{l,\xi} } & \leq \abs{\xi^\prime-\xi} \leq \frac{3}{16}\abs{\xi^\prime} \leq  \frac{3}{16}\abs{l,\xi^\prime}, 
\end{align}
\end{subequations}
and hence by two applications of \eqref{lem:scon}, there is some $c \in (0,1)$ such that  
\begin{align*} 
e^{\lambda\abs{k,\eta}^s} \leq e^{\lambda\abs{l,\xi^\prime}^s + c\lambda\abs{k-l,\eta-\xi}^s  + c\lambda\abs{\xi^\prime-\xi}^s}. 
\end{align*} 
Therefore, (using that $\abs{k,\eta} \approx \abs{l,\xi^\prime}$ from \eqref{ineq:ketalxiRepsilon}), 
\begin{align*}
\abs{ R_{N;LH}^{\epsilon,1}} & \lesssim \sum_{M \geq 8} \sum_{k,l\neq 0}\int_{\eta,\xi,\xi^\prime} \abs{A\hat f_k(\eta)} \abs{\eta l - \xi^\prime k}\rho_N(l,\xi) J_k(\eta)\jap{l,\xi^\prime}^\sigma e^{\lambda\abs{l,\xi^\prime}^s}\abs{\widehat{\phi}_l(\xi^\prime)_M} \\ 
& \quad\quad \times e^{c\lambda\abs{k-l,\eta-\xi}^s  + c\lambda\abs{\xi^\prime-\xi}^s} \abs{\widehat{(1-v^\prime)}(\xi^\prime - \xi)_{<M/8}} \abs{\hat{f}_{k-l}(\eta - \xi)_{< N/8}} d\eta d\xi d\xi^\prime. 
\end{align*}
From here we may proceed analogous to the treatment of $R_N^{1}$ with $(l,\xi^\prime)$ playing the role of $(l,\xi)$ and using \eqref{ineq:L2L2L1L1} (instead of \eqref{ineq:L2L2L1}) together with the bootstrap hypotheses to deal with the low-frequency factors. We omit the details and simply conclude that the result is analogous to \eqref{ineq:RN1ctrl}, except with an additional power of $\epsilon$.  

Turn now to $R_{N;HL}^{\epsilon,1}$. 
Similar to what occurs in \S\ref{sec:HLElliptic}, $l$ could be large relative to $\xi - \xi^\prime$ hence more `derivatives' are appearing on $\phi$ than $(1-v^\prime)$ and we are again in a situation similar to $R_{N;LH}^{\epsilon,1}$. 
As in \S\ref{sec:HLElliptic} we divide the integral based on the relative size of $l$ and $\xi$: 
\begin{align*} 
 R_{N;HL}^{\epsilon,1} & = -\frac{1}{2\pi}\sum_{M \geq 8} \sum_{k,l\neq 0}\int_{\eta,\xi,\xi^\prime} A \bar{\hat f}_k(\eta) \left[\mathbf{1}_{16\abs{l} \geq \abs{\xi}} + \mathbf{1}_{16\abs{l} < \abs{\xi}}  \right] A_k(\eta)(\eta l - \xi^\prime k)\rho_N(l,\xi) \\ & \quad\quad \times \widehat{(1-v^\prime)}(\xi^\prime - \xi)_{M} \widehat{\phi}_l(\xi^\prime)_{<M/8} \hat{f}_{k-l}(\eta - \xi)_{< N/8} d\eta d\xi d\xi^\prime \\ 
& = R_{N;HL}^{\epsilon,1;z} + R_{N;HL}^{\epsilon,1;v}.
\end{align*}

First consider $R_{N;HL}^{\epsilon,1;z}$, where on the support of the integrand, $16\abs{l} \geq \abs{\xi}$.
\begin{subequations} \label{ineq:ketalxiRepsilonHL}
\begin{align} 
\abs{ \abs{k,\eta} - \abs{l,\xi} } & \leq \abs{k-l,\eta-\xi} \leq 3\abs{l,\xi}/16, \\
\abs{\abs{l,\xi} - \abs{l,\xi^\prime}} & \leq \abs{\xi-\xi^\prime} \leq 38\abs{\xi}/32 \lesssim \abs{l}.
\end{align} 
\end{subequations}
If $\abs{l} \geq 16\abs{\xi}$, then in fact $38\abs{\xi}/32 < \abs{l}/4$, therefore by applying twice \eqref{lem:scon}, for some $c \in (0,1)$, 
\begin{align*} 
e^{\lambda\abs{k,\eta}^s} \leq e^{\lambda\abs{l,\xi}^s + c\lambda\abs{k-l,\eta-\xi}^s} \leq e^{\lambda\abs{l,\xi^\prime}^s + c\lambda\abs{\xi-\xi^\prime}^s + c\lambda\abs{k-l,\eta-\xi}^s}. 
\end{align*}
Alternatively, if  $\frac{1}{16}\abs{\xi} \leq \abs{l} \leq 16\abs{\xi}$ then $\abs{\xi - \xi^\prime} \approx \abs{l,\xi}$ and hence
\eqref{lem:scon} and \eqref{lem:strivial} imply for some (different) $c \in (0,1)$ we have,
\begin{align*} 
e^{\lambda\abs{k,\eta}^s} \leq e^{\lambda\abs{l,\xi}^s + c\lambda\abs{k-l,\eta-\xi}^s} \leq e^{c\lambda\abs{l,\xi^\prime}^s + c\lambda\abs{\xi-\xi^\prime}^s + c\lambda\abs{k-l,\eta-\xi}^s}. 
\end{align*}
In both cases, it follows that (using also $\jap{k,\eta} \approx \jap{l,\xi} \lesssim \jap{l}$ from \eqref{ineq:ketalxiRepsilonHL}), 
\begin{align*} 
 \abs{R_{N;HL}^{\epsilon,1;z}} & \lesssim \sum_{M \geq 8} \sum_{k,l\neq 0}\int_{\eta,\xi,\xi^\prime}\mathbf{1}_{16\abs{l} \geq \abs{\xi}}  \abs{A\hat f_k(\eta)} \abs{\eta l - \xi^\prime k}\rho_N(l,\xi) J_k(\eta)\jap{l}^\sigma e^{\lambda\abs{l,\xi^\prime}^s}\abs{\widehat{\phi}_l(\xi^\prime)_{<M/8}} \\ 
& \quad\quad \times e^{c\lambda\abs{\xi-\xi^\prime}^s + c\lambda\abs{k-l,\eta-\xi}^s} \abs{\widehat{(1-v^\prime)}(\xi^\prime - \xi)_{M}} \abs{\hat{f}_{k-l}(\eta - \xi)_{< N/8}} d\eta d\xi d\xi^\prime. 
\end{align*} 
For minor technical reasons divide into low and high frequencies: for some $M_0 \geq 8$,  
\begin{align*}
 \abs{R_{N;HL}^{\epsilon,1;z}} & \lesssim \left(\sum_{M \leq M_0} + \sum_{M \geq M_0} \right) \sum_{k,l\neq 0}\int_{\eta,\xi,\xi^\prime}\mathbf{1}_{16\abs{l} \geq \abs{\xi}}  \abs{A\hat f_k(\eta)} \abs{\eta l - \xi^\prime k}\rho_N(l,\xi) J_k(\eta)\jap{l}^\sigma e^{\lambda\abs{l,\xi^\prime}^s} \\ 
& \quad\quad \times \abs{\widehat{\phi}_l(\xi^\prime)_{<M/8}}e^{c\lambda\abs{\xi-\xi^\prime}^s + c\lambda\abs{k-l,\eta-\xi}^s} \abs{\widehat{(1-v^\prime)}(\xi^\prime - \xi)_{M}} \abs{\hat{f}_{k-l}(\eta - \xi)_{< N/8}} d\eta d\xi d\xi^\prime \\ 
& = R_{N;HL;L}^{\epsilon,1;z} +  R_{N;HL;H}^{\epsilon,1;z}. 
\end{align*} 
Consider next $R_{N;HL;H}^{\epsilon,1;z}$.  
By choosing $M_0$ sufficiently large (relative to our $O(1)$ arithmetic conventions), on the support of the integrand $(k,\eta)$ and $(l,\xi^\prime)$ are both non-resonant by \eqref{ineq:ketalxiRepsilonHL}. 
Therefore, by \eqref{ineq:BasicJswap} followed by \eqref{ineq:IncExp}, 
\begin{align*} 
 R_{N;HL;H}^{\epsilon,1;z} & \lesssim \sum_{M \geq 8} \sum_{k,l\neq 0}\int_{\eta,\xi,\xi^\prime}\mathbf{1}_{16\abs{l} \geq \abs{\xi}}  \abs{A \hat f_k(\eta)} \abs{\eta l - \xi^\prime k}\rho_N(l,\xi) J_l(\xi^\prime) \jap{l}^\sigma e^{\lambda\abs{l,\xi^\prime}^s}
\\ & \quad\quad \times \mathbf{1}_{t \not\in\I_{l,\xi^\prime}}\abs{\widehat{\phi}_l(\xi^\prime)_{<M/8}} e^{\lambda\abs{\xi-\xi^\prime}^s + \lambda\abs{k-l,\eta-\xi}^s} \abs{\widehat{(1-v^\prime)}(\xi^\prime - \xi)_{M}} \abs{\hat{f}_{k-l}(\eta - \xi)_{< N/8}} d\eta d\xi d\xi^\prime.
\end{align*} 
Since $l$ cannot be zero and 
by \eqref{ineq:ketalxiRepsilonHL}, we have $\abs{l,\xi} \lesssim \abs{l}^{1-s/2}\abs{k,\eta}^{s/2}$, which implies
\begin{align*} 
 R_{N;HL;H}^{\epsilon,1;z} & \lesssim \sum_{M \geq 8} \sum_{k,l\neq 0}\int_{\eta,\xi,\xi^\prime}\mathbf{1}_{16\abs{l} \geq \abs{\xi}} \abs{k,\eta}^{s/2} \abs{A \hat f_k(\eta)}\rho_N(l,\xi)\abs{l}^{1-s/2} J_l(\xi^\prime)\jap{l}^\sigma e^{\lambda\abs{l,\xi^\prime}^s}\mathbf{1}_{t \not\in\I_{l,\xi^\prime}}
\\ & \quad\quad \times \abs{\widehat{\phi}_l(\xi^\prime)_{<M/8}} e^{\lambda\abs{\xi-\xi^\prime}^s + \lambda\abs{k-l,\eta-\xi}^s} \abs{\widehat{(1-v^\prime)}(\xi^\prime - \xi)_{M}} \abs{k-l,\eta-\xi} \abs{\hat{f}_{k-l}(\eta - \xi)_{< N/8}} d\eta d\xi d\xi^\prime.
\end{align*} 
Therefore, since $\abs{k,\eta} \approx \abs{l,\xi} \approx \abs{l,\xi^\prime}$ from \eqref{ineq:ketalxiRepsilonHL}, by \eqref{ineq:L2L2L1L1} and \eqref{ineq:GeneralOrtho} (denoting $\chi^{NR}_l(t,\xi) = \mathbf{1}_{t \not\in\I_{l,\xi}}$),  
\begin{align*} 
R_{N;HL;H}^{\epsilon,1;z} & \lesssim \sum_{M \geq 8} \norm{\abs{\grad}^{s/2}Af_{\sim N}}_2 \norm{\abs{\grad}^{1-s/2}\chi^{NR}AP_{\neq 0}\phi_{\sim N}}_2 \norm{(1-v^\prime)_M}_{\G^{\lambda, \sigma-1}} \norm{f_{<N/8}}_{\G^{\lambda,\sigma}} \\ 
 & \lesssim \sum_{M \geq 8} \norm{\abs{\grad}^{s/2}Af_{\sim N}}_2 \norm{\abs{\grad}^{1-s/2}\chi^{NR} AP_{\neq 0}\phi_{\sim N}}_2 \frac{1}{M}\norm{(1-v^\prime)_M}_{\G^{\lambda, \sigma}} \norm{f_{<N/8}}_{\G^{\lambda,\sigma}} \\
  & \lesssim \norm{\abs{\grad}^{s/2}Af_{\sim N}}_2 \norm{\abs{\grad}^{1-s/2}\chi^{NR}AP_{\neq 0}\phi_{\sim N}}_2 \norm{f_{<N/8}}_{\G^{\lambda,\sigma}} \left(\sum_{M \geq 8}\norm{(1-v^\prime)_M}^2_{\G^{\lambda, \sigma}}\right)^{1/2}, 
\end{align*} 
where the last line followed by Cauchy-Schwarz. By \eqref{ineq:LPOrthoProject} and the bootstrap hypotheses,  
\begin{align} 
R_{N;HL;H}^{\epsilon,1;z} & \lesssim \epsilon^2\norm{\abs{\grad}^{s/2}Af_{\sim N}}_2 \norm{\abs{\grad}^{1-s/2} \chi^{NR}AP_{\neq 0}\phi_{\sim N}}_2.  \label{ineq:RNepsLHz}
\end{align}
The treatment of $R_{N;HL;L}^{\epsilon,1;z}$ is straightforward by a similar argument but now simply using Lemma \ref{lem:LossyElliptic} to handle $\phi$. We omit the details and state the result 
\begin{align}
\sum_{N \geq 8}R_{N;HL;L}^{\epsilon,1;z} \lesssim \frac{\epsilon^4}{\jap{t}^2}, \label{ineq:RNepsLHzL}
\end{align}  
completing the treatment of $R_{N;HL}^{\epsilon,1;z}$. 

Next we turn to $R_{N;HL}^{\epsilon,1;v}$, in which case we can consider all of the `derivatives' to be landing on $1-v^\prime$. 
On the support of the integrand, 
\begin{subequations} \label{ineq:RnHLeps_FreqLoc} 
\begin{align}
\abs{ \abs{k,\eta} - \abs{l,\xi} } & \leq \abs{k-l,\eta-\xi} \leq 3\abs{l,\xi}/16, \\
\abs{\abs{\xi - \xi^\prime}-\abs{l,\xi}} & \leq \abs{l,\xi^\prime} \leq \abs{\xi}/16 + \abs{\xi^\prime} \leq \abs{\xi-\xi^\prime}/16 + 17\abs{\xi^\prime}/16 \leq 67\abs{\xi-\xi^\prime}/100, 
\end{align}
\end{subequations}
which implies by two applications of \eqref{lem:scon} there exists some $c \in (0,1)$ such that 
\begin{align*} 
e^{\lambda\abs{k,\eta}^s} \leq e^{\lambda\abs{l,\xi}^s + c\lambda\abs{k-l,\eta-\xi}^s} \leq e^{\lambda\abs{\xi-\xi^\prime}^s + c\lambda\abs{l,\xi^\prime}^s + c\lambda\abs{k-l,\eta-\xi}^s}. 
\end{align*} 
Therefore, since also $\abs{l,\xi} \approx \abs{\xi-\xi^\prime}$ by \eqref{ineq:RnHLeps_FreqLoc},  
\begin{align*} 
\abs{R_{N;HL}^{\epsilon,1;v}} 
& \lesssim \sum_{M \geq 8} \sum_{k,l\neq 0}\int_{\eta,\xi,\xi^\prime} \abs{A\hat f_k(\eta)}\mathbf{1}_{16\abs{l} \leq \abs{\xi}} J_k(\eta) \jap{\xi-\xi^\prime}^\sigma \rho_N(l,\xi)e^{\lambda\abs{\xi-\xi^\prime}^s}\abs{\widehat{(1-v^\prime)}(\xi^\prime - \xi)_{M}} \\
  & \quad\quad \times \abs{k-l,\eta-\xi} \abs{l,\xi^\prime} e^{c\lambda\abs{l,\xi^\prime}^s + c\lambda\abs{k-l,\eta-\xi}^s}\abs{\widehat{\phi}_l(\xi^\prime)_{<M/8} \hat{f}_{k-l}(\eta - \xi)_{< N/8}} d\eta d\xi d\xi^\prime. 
\end{align*} 
We will now use the following analogue of \eqref{ineq:JRswap}, which applies on the support of the integrand due to the frequency localizations  \eqref{ineq:RnHLeps_FreqLoc}: 
\begin{align} 
J_k(\eta) \lesssim J^R(\xi-\xi^\prime)e^{20\mu \abs{\xi^\prime}^{1/2} + 20\mu \abs{\eta - \xi}^{1/2}}. \label{ineq:JRswap2}
\end{align} 
Applying this together with \eqref{ineq:IncExp} and \eqref{ineq:SobExp} implies 
\begin{align*}  
\abs{R_{N;HL}^{\epsilon,1}} & \lesssim \sum_{M \geq 8} \sum_{k,l\neq 0}\int_{\eta,\xi,\xi^\prime} \abs{A \hat f_k(\eta)} A^R(\xi-\xi^\prime) \abs{\widehat{(1-v^\prime)}(\xi^\prime - \xi)_{M}} \\ & \quad\quad \times \rho_N(l,\xi) \mathbf{1}_{16\abs{l} \leq \abs{\xi}}
e^{\lambda\abs{l,\xi^\prime}^s + \lambda\abs{k-l,\eta-\xi}^s}\abs{\widehat{\phi}_l(\xi^\prime)_{<M/8} \hat{f}_{k-l}(\eta - \xi)_{< N/8}} d\eta d\xi d\xi^\prime. 
\end{align*}
Since $\abs{l,\xi} \approx \abs{\xi-\xi^\prime}$ by \eqref{ineq:RnHLeps_FreqLoc}, the sum only includes boundedly many terms.
Therefore, by \eqref{ineq:L2L2L1L1}, Lemma \ref{lem:LossyElliptic} and the bootstrap hypotheses,
\begin{align} 
\abs{R_{N;HL}^{\epsilon,1}} & \lesssim \norm{Af_{\sim N}}_2\norm{A^R(1-v^\prime)_{\sim N}}_2\norm{f}_{\G^{\lambda,\sigma-4}}\norm{P_{\neq}\phi}_{\G^{\lambda,\sigma- 4}}\nonumber \\ 
& \lesssim \frac{\epsilon}{\jap{t}^2} \norm{Af_{\sim N}}_2^2 + \frac{\epsilon^3}{\jap{t}^2}\norm{A^R(1-v^\prime)_{\sim N}}^2_2, \label{ineq:RNepsHLD}
\end{align} 
which completes our treatment of $R_{N;HL}^{\epsilon,1}$.

We turn to the remainder term $R_{N;HH}^{\epsilon,1}$. 
Analogous to \S\ref{sec:EllipticRemainder}, there is a problem when $l$ is large compared to $\xi$. 
The situation here is only slightly more subtle. 
We divide into two cases: 
\begin{align*}   
R_{N;HH}^{\epsilon,1} & = -\frac{1}{2\pi}\sum_{M \in \mathbb D}\sum_{\frac{1}{8}M \leq M^\prime  \leq  8M} \sum_{k,l\neq 0}\int_{\eta,\xi,\xi^\prime} A \bar{\hat f}_k(\eta) \left[\mathbf{1}_{\abs{l} > 100\abs{\xi^\prime}} + \mathbf{1}_{\abs{l} \leq 100\abs{\xi^\prime}}\right]
\\ & \quad\quad \times  A_k(\eta)(\eta l - \xi^\prime k)\rho_N(l,\xi)\widehat{(1-v^\prime)}(\xi^\prime - \xi)_{M^\prime} \widehat{\phi}_l(\xi^\prime)_M \hat{f}_{k-l}(\eta - \xi)_{< N/8} d\eta d\xi d\xi^\prime \\
& = R_{N;HH}^{\epsilon,1;z} + R_{N;HH}^{\epsilon,1;v}.  
\end{align*} 
First consider $R_{N;HH}^{\epsilon,1;z}$. On the support of the integrand we have 
\begin{subequations} \label{ineq:RNHHzFreqLoc} 
\begin{align} 
\abs{ \abs{k,\eta} - \abs{l,\xi} } & \leq \abs{k-l,\eta-\xi} \leq 6\abs{l,\xi}/32, \\
\abs{\abs{l,\xi} - \abs{l,\xi^\prime}} & \leq \abs{\xi - \xi^\prime} \leq 24\abs{\xi^\prime} \leq \frac{24}{100}\abs{l,\xi^\prime}. 
\end{align}
\end{subequations} 
Therefore by two applications of \eqref{lem:scon}, 
\begin{align*}
 e^{\lambda\abs{k,\eta}^s} \leq e^{\lambda\abs{l,\xi}^s + c\lambda\abs{k-l,\eta-\xi}^s} \leq e^{\lambda\abs{l,\xi^\prime}^s + c\lambda\abs{\xi-\xi^\prime}^s + c\lambda\abs{k-l,\eta-\xi}^s}.
\end{align*} 
By $\abs{l} > 100\abs{\xi^\prime}$ and \eqref{ineq:RNHHzFreqLoc}, it follows that $\abs{\xi} \leq \frac{2524}{10000}\abs{l}$ and hence 
$(l,\xi)$ cannot be resonant and $\abs{\eta} \leq 1.531\abs{k}$, which implies by $N \geq 8$ that $(k,\eta)$ cannot be resonant.  
Also using
\begin{align*} 
\abs{\eta l - \xi^\prime k} \leq \abs{l,\xi^\prime}\abs{k-l,\eta-\xi^\prime} \leq  \abs{l,\xi^\prime}\left(\abs{k-l,\eta-\xi} + \abs{\xi - \xi^\prime}\right),   
\end{align*}
Then \eqref{ineq:BasicJswap}, \eqref{ineq:IncExp} and \eqref{ineq:SobExp} together imply
\begin{align*} 
\abs{R_{N;HH}^{\epsilon,1;z}} & \lesssim \sum_{M \in \mathbb D}\sum_{M^\prime \approx M} \sum_{k,l\neq 0}\int_{\eta,\xi,\xi^\prime} \abs{A \hat f_k(\eta)}\mathbf{1}_{\abs{l} > 100\abs{\xi^\prime}}\rho_N(l,\xi)A_l(\xi^\prime)\abs{l,\xi^\prime}\widehat{\phi}_l(\xi^\prime)_M \\ & \quad\quad \times e^{\lambda\abs{\xi-\xi^\prime}^s}\abs{\widehat{(1-v^\prime)}(\xi^\prime - \xi)_{M^\prime}} e^{\lambda\abs{k-l,\eta-\xi}^s}\abs{\hat{f}_{k-l}(\eta - \xi)_{< N/8}} d\eta d\xi d\xi^\prime.
\end{align*} 
By \eqref{ineq:L2L2L1L1}, $\abs{l,\xi^\prime} \approx N$ (by \eqref{ineq:RNHHzFreqLoc}) and \eqref{ineq:GeneralOrtho}, 
\begin{align} 
\abs{R_{N;HH}^{\epsilon,1;z}} & \lesssim \norm{\abs{\grad}^{s/2}Af_{\sim N}}\norm{f_{<N/8}}_{\G^{\lambda,\sigma}}\norm{\abs{\grad}^{1-s/2} \chi^{NR} A\phi_{\sim N}}_2 \sum_{M^\prime \in \mathbb D} \norm{(1-v^\prime)_{M^\prime}}_{\G^{\lambda, \sigma-1}} \nonumber \\
& \lesssim \norm{\abs{\grad}^{s/2}Af_{\sim N}}\norm{f_{<N/8}}_{\G^{\lambda,\sigma}}\norm{\abs{\grad}^{1-s/2} \chi^{NR} A\phi_{\sim N}}_2 \left(\sum_{M^\prime \in \mathbb D} \norm{(1-v^\prime)_{M^\prime}}^2_{\G^{\lambda, \sigma}}\right)^{1/2} \nonumber \\
& \lesssim \epsilon^2\norm{\abs{\grad}^{s/2}Af_{\sim N}}_2 \norm{\abs{\grad}^{1-s/2}\chi^{NR} AP_{\neq 0}\phi_{\sim N}}_2, \label{ineq:RNepsHH}
\end{align}
where the last line followed from the bootstrap hypotheses. 

Turn to $R_{N;HH}^{\epsilon,1;v}$. 
On the support of the integrand there holds 
\begin{subequations} \label{ineq:RNepsHHvFreqLoc}
\begin{align} 
\abs{ \abs{k,\eta} - \abs{l,\xi} } & \leq \abs{k-l,\eta-\xi} \leq 3\abs{l,\xi}/16, \\
\abs{\xi - \xi^\prime} & \leq 24\abs{\xi^\prime} \leq 24\abs{l,\xi^\prime} \\ 
\abs{l,\xi^\prime} & \leq 101\abs{\xi^\prime} \leq 2424\abs{\xi - \xi^\prime},
\end{align}
\end{subequations}
and hence by \eqref{lem:scon} followed by \eqref{lem:strivial} for some $c \in (0,1)$, 
\begin{align*} 
e^{\lambda\abs{k,\eta}^s} \leq e^{\lambda\abs{l,\xi}^s + c\lambda\abs{k-l,\eta-\xi}^s} \leq e^{c\lambda\abs{l,\xi^\prime}^s + c\lambda\abs{\xi-\xi^\prime}^s + c\lambda\abs{k-l,\eta-\xi}^s}.
\end{align*} 
Notice that here $N \lesssim \abs{\xi,l} \lesssim M$. 
By Lemma \ref{basic} and \eqref{ineq:RNepsHHvFreqLoc}, 
\begin{align*} 
J_k(\eta) & \lesssim e^{2\mu\abs{k,\eta}^{1/2}} \lesssim e^{2 \mu\abs{k-l,\eta-\xi}^{1/2} + 2\mu\abs{l,\xi^\prime}^{1/2} + 2\mu\abs{\xi^\prime-\xi}^{1/2}}. 
\end{align*}
The previous two estimates together with \eqref{ineq:IncExp} and \eqref{ineq:SobExp} imply 
\begin{align*} 
\abs{R_{N;HH}^{\epsilon,1;v}} & \lesssim \sum_{M \in \mathbb D} \sum_{M^\prime \approx M} \sum_{k,l\neq 0}\int_{\eta,\xi,\xi^\prime} \abs{A\hat f_k(\eta)}\mathbf{1}_{\abs{l} \leq 100\abs{\xi^\prime}}\rho_N(l,\xi) \\ & \quad\quad \times e^{\lambda\abs{\xi - \xi^\prime}^s}\abs{\widehat{(1-v^\prime)}(\xi^\prime - \xi)_{M^\prime}} e^{\lambda\abs{l,\xi^\prime}^s}\abs{\widehat{\phi}_l(\xi^\prime)_M} e^{\lambda\abs{k-l,\eta - \xi}^s}\abs{\hat{f}_{k-l}(\eta - \xi)_{< N/8}} d\eta d\xi d\xi^\prime. 
\end{align*} 
By applying \eqref{ineq:L2L2L1L1}, Lemma \ref{lem:LossyElliptic} and the bootstrap hypotheses similar to above,
\begin{align*} 
\abs{R_{N;HH}^{\epsilon,1;v}} & \lesssim   \sum_{M \in \mathbb D} \sum_{M^\prime \approx M} \norm{Af_{\sim N}}_{2}\norm{P_{\neq 0}\phi_{M}}_{\G^{\lambda,\sigma-3}}\norm{(1-v^\prime)_{M^\prime}}_{\G^{\lambda,\sigma - 2}}\norm{f_{<N/8}}_{\G^{\lambda,\sigma}} \\ 
& \lesssim \epsilon  \sum_{M \in \mathbb D} \sum_{M^\prime \approx M}\norm{Af_{\sim N}}_{2} \frac{1}{M^2} \norm{P_{\neq 0}\phi_{M}}_{\G^{\lambda,\sigma-3}}\norm{(1-v^\prime)_{M^\prime}}_{\G^{\lambda,\sigma}} \\
& \lesssim \frac{\epsilon^3}{N\jap{t}^2}\norm{Af_{\sim N}}_{2}, 
\end{align*} 
where the last two lines followed from $N \lesssim M \approx M^\prime$, Cauchy-Schwarz and \eqref{ineq:GeneralOrtho}. 
Together with \eqref{ineq:RNepsHH}, \eqref{ineq:RNepsHLD} \eqref{ineq:RNepsLHz} and \eqref{ineq:RNepsLHzL}, this completes the treatment of $R_{N}^{\epsilon,1}$ after summing $N$ and applying \eqref{ineq:GeneralOrtho}, 
hence  reducing to terms appearing on the RHS of \eqref{ineq:ReacIntro}. 

\subsubsection{Term $R_N^3$: remainder from commutator} \label{sec:RN3}
Now we treat $R_N^3$, which arose from the integration by parts intended for treating transport.
By \eqref{ineq:L2L2L1} (for $\sigma > 6$), the bootstrap hypotheses, 
\begin{align*}
\abs{R_N^3} & \lesssim \sum_{k,l}\int_{\eta,\xi} \abs{A \hat f_k(\eta)} \abs{k-l,\eta-\xi}\abs{\hat{u}_l(\xi)_N} A_{k-l}(\eta-\xi)\abs{\widehat{f}_{k-l}(\eta - \xi)_{<N/8}} d\eta d\xi \\ 
& \lesssim \norm{Af_{\sim N}}_2\norm{u_{N}}_{H^{\sigma-4}}\norm{Af_{<N/8}}_2 \\ 
&  \lesssim \frac{\epsilon}{\jap{t}^{2-K_D\epsilon/2}}\norm{Af_{\sim N}}^2_2 + \epsilon\jap{t}^{2-K_D\epsilon/2}\norm{u_{N}}^2_{H^{\sigma-4}}, 
\end{align*} 
where we also used that $\abs{k-l,\eta-\xi} \lesssim \abs{l,\xi}$ on the support of the integrand. 
By \eqref{ineq:GeneralOrtho}, Lemma \ref{lem:LossyElliptic} and the bootstrap hypotheses, 
\begin{align*} 
\sum_{N \geq 8} \abs{R_N^3} & \lesssim \frac{\epsilon^3}{\jap{t}^{2-K_D\epsilon/2}}. 
\end{align*}
This completes the treatment of $R_N^3$, as this appears on the RHS of \eqref{ineq:ReacIntro}. 

\subsection{Zero mode reaction term} \label{sec:ZeroMdReac}
Next we turn to $R_N^2$, which is the part of the reaction term involving $[\partial_t v]$. 
Here we need to make sure that assigning slightly less regularity to $[\partial_t v]$ than $f$ is consistent with $R_N^2$.
Also note that $[\partial_t v]$ has non-resonant regularity, but forces resonant frequencies here, expressed in the loss that Lemma \ref{lem:Jswap} could incur. 
Write $R_N^2$ on the frequency-side and divide into the two natural cases
\begin{align*} 
R_N^{2} & =-\sum_{k}\int_{\eta,\xi} A \bar{\hat f}_k(\eta) \left[\chi^{D} + \chi^{\ast}\right] A_k(\eta)\widehat{[\partial_t v]}(\xi)_Ni(\eta-\xi)\hat{f}_{k}(\eta - \xi)_{< N/8} d\eta d\xi \\
& = R_N^{2;D} + R_N^{2;\ast},  
\end{align*}
where $\chi^{D} = \mathbf{1}_{t \in\I_{k,\eta}}\mathbf{1}_{t \in\I_{k,\xi}}$ and $\chi^{\ast} = 1 - \chi^{D}$.  
Next note that on the support of the integrand we have 
\begin{align}
\abs{k} + \abs{\eta - \xi} \leq 3N/32 \leq 3\abs{\xi}/16 \label{ineq:zeroReacFreqLoc}
\end{align}
which implies $\abs{k} \lesssim \abs{\eta} \approx \abs{\xi}$. 
Also note $\abs{\xi} \gtrsim N$. 

First treat the term $R_N^{2;\ast}$. 
 By \eqref{lem:scon}, \eqref{ineq:BasicJswap}, \eqref{ineq:IncExp} and \eqref{ineq:SobExp}, 
\begin{align*} 
\abs{R_N^{2;\ast}} & \lesssim \sum_{k}\int_{\eta,\xi} \chi^\ast \abs{A \hat f_k(\eta)} A_0(\xi)\abs{\widehat{[\partial_t v]}(\xi)_N} e^{\lambda\abs{k,\eta-\xi}^{s}}\abs{\hat{f}_{k}(\eta - \xi)_{< N/8}} d\eta d\xi. 
\end{align*} 
To deal with the norm imbalance between $[\partial_t v]$ and $f$, by $\abs{\xi} \gtrsim 1$ and $\abs{\eta} \approx \abs{\xi}$, 
\begin{align*} 
\abs{R_N^{2;\ast}}& \lesssim \sum_{k}\int_{\eta,\xi} \chi^\ast \abs{A \hat f_k(\eta)} \frac{\abs{\eta}^{s/2}\abs{\xi}^{s/2}}{\jap{\xi}^s} A_0(\xi)\abs{\widehat{[\partial_t v]}(\xi)_N} e^{\lambda\abs{k,\eta-\xi}^{s}}\abs{\hat{f}_{k}(\eta - \xi)_{< N/8}} d\eta d\xi. 
\end{align*} 
It follows from \eqref{ineq:L2L2L1} and the bootstrap hypotheses that
\begin{align} 
\abs{R_{N}^{2;\ast}} & \lesssim \epsilon\norm{\abs{\grad}^{s/2}Af_{\sim N}}_2 \norm{\abs{\partial_v}^{s/2}\frac{A}{\jap{\partial_v}^s}[\partial_t v]_N}_2 \nonumber \\
& \lesssim  \frac{\epsilon}{\jap{t}^{2s}}\norm{\abs{\grad}^{s/2}Af_{\sim N}}_2^2 + \epsilon\jap{t}^{2s}\norm{\abs{\partial_v}^{s/2}\frac{A}{\jap{\partial_v}^s}[\partial_t v]_N}_2. 
 \label{ineq:RN2ast}
\end{align}
The former is absorbed by $CK_\lambda$ and the latter is controlled by $CK_\lambda^{v,1}$.

Next turn to $R_N^{2;D}$. Applying \eqref{lem:scon} and \eqref{ineq:zeroReacFreqLoc}, there exists a $c \in (0,1)$ such that, 
\begin{align*} 
\abs{R_N^{2;D}} & \lesssim \sum_{k}\int_{\eta,\xi} \chi^{D} \abs{A \hat f_k(\eta)} A_0(\xi)\frac{J_k(\eta)}{J_0(\xi)}\abs{\widehat{[\partial_t v]}(\xi)_N} e^{c\lambda\abs{k,\eta-\xi}^{s}}\abs{\eta-\xi}\abs{\hat{f}_{k}(\eta - \xi)_{< N/8}} d\eta d\xi. 
\end{align*}
Since $t \in\I_{k,\eta}$, $A_k(\eta) \lesssim \tilde A_k(\eta)$. 
Moreover, by \eqref{ineq:RatJ2partt}, \eqref{ineq:IncExp} and \eqref{ineq:SobExp}, 
\begin{align*} 
\abs{R_N^{2;D}} & \lesssim \sum_{k}\int_{\eta,\xi} \chi^{D} \abs{\tilde A \hat f_k(\eta)} A_0(\xi)\abs{\widehat{[\partial_t v]}(\xi)_N}\frac{\abs{\eta}}{k^2}\sqrt{\frac{\partial_tw_k(\eta)}{w_k(\eta)}}\sqrt{\frac{\partial_tw_0(\xi)}{w_0(\xi)}}e^{\lambda\abs{k,\eta-\xi}^{s}}\abs{\hat{f}_{k}(\eta - \xi)_{< N/8}} d\eta d\xi. 
\end{align*} 
Since $t \approx \frac{\eta}{k}$, by \eqref{ineq:zeroReacFreqLoc}, \eqref{ineq:L2L2L1} and the bootstrap hypotheses, 
\begin{align} 
\abs{R_{N}^{2;D}} & \lesssim \sum_{k}\int_{\eta,\xi} \chi^{D} \abs{\tilde A \hat f_k(\eta)} A_0(\xi)\abs{\widehat{[\partial_t v]}(\xi)_N} \frac{t^{1+s}}{\abs{k}^{1-s}\jap{\xi}^{s}}\sqrt{\frac{\partial_tw_k(\eta)}{w_k(\eta)}}\sqrt{\frac{\partial_tw_0(\xi)}{w_0(\xi)}} e^{\lambda\abs{k,\eta-\xi}^{s}}\abs{\hat{f}_{k}(\eta - \xi)_{< N/8}} d\eta d\xi \nonumber \\ 
 & \lesssim \epsilon t^{1+s}\norm{\sqrt{\frac{\partial_t w}{w}}\tilde Af_{\sim N}}_2 \norm{\sqrt{\frac{\partial_t w}{w}}\frac{A}{\jap{\partial_v}^s}[\partial_t v]_N}_2 \nonumber \\ 
& \lesssim \epsilon\norm{\sqrt{\frac{\partial_t w}{w}}\tilde Af_{\sim N}}_2^2 + \epsilon t^{2+2s}\norm{\sqrt{\frac{\partial_t w}{w}}\frac{A}{\jap{\partial_v}^s}[\partial_t v]_N}^2_2.  \label{ineq:RN2D}
\end{align} 
The first term is absorbed by $CK_w$ and the latter is controlled by $CK_w^{v,1}$. 
This completes the treatment of $R^2_N$.
Combining the results of \eqref{ineq:RN2ast}, \eqref{ineq:RN2D} with \S\ref{sec:RNeps}, \S\ref{sec:RNeps} and \S\ref{sec:RN3} 
and applying \eqref{ineq:GeneralOrtho} completes the treatment of $R_N$, proving Proposition \ref{prop:ReactionIntro}. 

\section{Remainder} \label{sec:RemainderEnergy}
In this section we prove Proposition \ref{prop:RemainderIntro}. 
The commutator cannot gain us anything so we may as well treat each term separately, 
\begin{align*} 
\mathcal{R} & = 2\pi\sum_{N \in \mathbb{D}}\sum_{\frac{1}{8}N \leq N^\prime \leq 8N} \int Af\left[A(u_{N}\grad f_{N^\prime})\right] dx dv - 2\pi\sum_{N \in \mathbb{D}}\sum_{\frac{1}{8}N \leq N^\prime \leq 8N} \int Af u_{N}\grad A f_{N^\prime} dx dv \\  
& = \mathcal{R}_{a} + \mathcal{R}_{b}. 
\end{align*}
Consider first $\mathcal{R}_{a}$, written on the Fourier side: 
\begin{align*} 
\mathcal{R}_{a} & = \sum_{N \in \mathbb D}\sum_{N^\prime \approx N}\sum_{k,l}\int_{\eta,\xi} A \bar{\hat f}_k(\eta)A_k(\eta) \hat{u}_l(\xi)_{N} \cdot \widehat{\grad f}_{k-l}(\eta-\xi)_{N^\prime} d\xi d\eta.   
\end{align*}
On the support of the integrand, $\abs{l,\xi} \approx \abs{k-l,\eta-\xi}$ and hence by \eqref{lem:strivial} for some $c \in (0,1)$,
\begin{align*}
\abs{k,\eta}^s \leq c\abs{k-l,\eta-\xi}^s + c\abs{l,\xi}^s.
\end{align*}
Hence, 
\begin{align*} 
\abs{\mathcal{R}_{a}} & \lesssim \sum_{N \in \mathbb D}\sum_{N^\prime \approx N}\sum_{k,l}\int_{\eta,\xi} \abs{A\hat f_k(\eta)} J_k(\eta)\jap{l,\xi}^{\sigma/2+1}e^{c\lambda\abs{l,\xi}^{s}} \abs{\hat{u}_l(\xi)_{N}}  \\ & \quad\quad \times \jap{k-l,\eta-\xi}^{\sigma/2 - 1}e^{c\lambda\abs{k-l,\eta-\xi}^{s}}\abs{\widehat{\grad f}_{k-l}(\eta-\xi)_{N^\prime}} d\xi d\eta.
\end{align*} 
By Lemma \ref{basic}, \eqref{ineq:IncExp} and \eqref{ineq:SobExp} (since $c < 1$ and $s > 1/2$),
\begin{align*} 
\abs{\mathcal{R}_{a}} & \lesssim \sum_{N \in \mathbb D}\sum_{N^\prime \approx N}\sum_{k,l}\int_{\eta,\xi} \abs{A\hat f_k(\eta)}e^{\lambda\abs{l,\xi}^{s}} \abs{\hat{u}_l(\xi)_{N}} e^{\lambda\abs{k-l,\eta-\xi}^{s}}\abs{\widehat{\grad f}_{k-l}(\eta-\xi)_{N^\prime}} d\xi d\eta.
\end{align*} 
Hence by \eqref{ineq:L2L2L1}, Lemma \ref{lem:LossyElliptic} and the bootstrap hypotheses, 
\begin{align*} 
\abs{\mathcal{R}_{a}} & \lesssim \sum_{N \in \mathbb D} \sum_{N^\prime \approx N} \norm{Af}_2\norm{u_N}_{\G^{\lambda,\sigma-4}}\norm{f_{N^\prime}}_{\G^{\lambda,\sigma-1}}  \lesssim \sum_{N^\prime \in \mathbb D}\frac{\epsilon}{t^{2-K_D\epsilon/2} N^\prime} \norm{Af}_2\norm{f_{N^\prime}}_{\G^{\lambda,\sigma}} \lesssim  \frac{\epsilon^3}{t^{2-K_D\epsilon/2}}. 
\end{align*}
This completes the treatment of $\mathcal{R}_{a}$. 
The term $\mathcal{R}_{b}$ is similar and hence omitted. 
This completes the proof of Proposition \ref{prop:RemainderIntro}. 

\section{Coordinate system controls} \label{sec:CoordControls}
In this section we detail the controls on \eqref{def:vPDE} and prove Proposition \ref{prop:CoefControl_Real}. 

\subsection{Derivation of \eqref{def:vPDE}} \label{sec:CordDeriv}
As in \S\ref{sec:CoordinateTrans}, denote  $v^\prime(t,v(t,y)) = \partial_y v(t,y)$, $v^{\prime\prime}(t,v(t,y)) = \partial_{yy} v(t,y)$ 
and $[\partial_t v](t,v(t,y)) = \partial_t v(t,y)$. 
Since by \eqref{def:v}, $v(t,y)$ satisfies
\begin{align*}
\frac{d}{dt}(t(v_y(t,y) - 1)) = -\omega_0(t,y), 
\end{align*} 
we can directly derive \eqref{def:pdevp} via the chain rule.  
Similarly, we may also derive \eqref{def:vpp} via the chain rule using the definitions of $v^\prime$ and $v^{\prime\prime}$.  

Deriving \eqref{def:ddtv} requires more work. 
Notice that
\begin{align}
\frac{d}{dt}\left(t (v(t,y)  - y)  \right) = U^x_0(t,y),  \label{def:psidiff}
\end{align} 
and denote $C(t,v(t,y)) = v(t,y) - y$.  
From \eqref{def:psidiff} we get (recalling the definitions $[\partial_t v](t,v(t,y)) = \partial_t v(t,y)$ and $\tilde u_0(t,v(t,y)) = U_0^x(t,y)$):  
\begin{align} 
\partial_t v(t,y) & = \frac{1}{t}U_0^x(t,y) - \frac{1}{t}\left(v(t,y) - y\right) \nonumber \\ 
[\partial_t v](t,v) & = \frac{1}{t}\tilde u_0(t,v) - \frac{1}{t}C(t,v). \label{ineq:dtvCdef}
\end{align}
Via the chain rule, 
\begin{align} 
\partial_t C(t,v(t,y)) = \partial_t v(t,y) - \partial_v C(t,v(t,y)) \partial_t v(t,y). \label{def:dtc}
\end{align}
Differentiating \eqref{ineq:dtvCdef} in time and using \eqref{def:dtc}, \eqref{ineq:dtvCdef} and \eqref{eq:tildeu0_moment} eliminates $C$ entirely and derives \eqref{def:ddtv}.  

Moreover,
\begin{align*} 
v^\prime(t,v(t,y)) - 1 = \partial_y v(t,y) - 1 = \partial_v C(t,v(t,y)) \partial_y v(t,y),
\end{align*}
which in particular implies 
\begin{align} 
\partial_v C(t,v) & = \frac{v^\prime(t,v) - 1}{v^\prime(t,v)}. \label{ineq:vC}
\end{align}
Finally, notice that \eqref{ineq:dtvCdef} together with \eqref{ineq:vC} and \eqref{eq:tildeu0_simple} implies
\begin{align} 
v^\prime \partial_v [\partial_t v](t,v) & =  -\frac{1}{t}f_0(t,v) - \frac{1}{t}(v^\prime - 1)(t,v). \label{def:barh2}
\end{align} 

\begin{remark} \label{rmk:CoefWeak}
From \eqref{def:pdevp}, \eqref{def:vpp} and the bootstrap hypotheses, one can show with relative ease that 
\begin{align}
\norm{A(1-v^\prime)(t)}_2 + \norm{A \left( (v^{\prime})^2 - 1\right)(t) }_2 + \norm{\jap{\partial_v}^{-1}Av^{\prime\prime}(t)}_2 \lesssim \epsilon^2.  \label{ineq:CoeffBd_R}
\end{align} 
The estimates on $(1-v^\prime)$ follow from energy estimates on \eqref{def:pdevp} performed with the unknown $t(v^\prime(t,v) - 1)$ in a manner 
that is very similar to, but easier than, techniques applied in \S\ref{sec:Transport}, \S\ref{sec:ZeroMdReac} and \S\ref{sec:RemainderEnergy}.
Similar control on $(v^\prime)^2 - 1$ and $v^{\prime\prime}$ then follows from \eqref{ineq:Aalg} (recalling \eqref{def:vpp}).  

The estimates \eqref{ineq:CoeffBd_R} suffice for most purposes, however they do not obtain control on the CCK terms in \eqref{def:QL}. 
This is because the estimates just described are insensitive to whether the background shear flow is converging and so cannot ensure that the CCK integrals are convergent. 
Hence, in order to do better, in \S\ref{sec:PropCtrlReal} below we make estimates that imply the convergence of the shear flow.  
\end{remark}

\subsection{Proof of Proposition \ref{prop:CoefControl_Real}} \label{sec:PropCtrlReal}
The purpose of this section is to prove Proposition \ref{prop:CoefControl_Real}, announced in \S\ref{sec:MainEnergy}. 
In order to get good control on \eqref{def:QL}, we will rearrange \eqref{def:pdevp} in the following manner using also \eqref{def:barh2} (see also Remark \ref{rmk:CoefWeak}). For notational convenience, denote $h(t,v) = v^\prime(t,v) - 1$ and write
\begin{align} 
\partial_t h + [\partial_t v] \partial_v h & = \frac{1}{t}\left(-f_0 - h\right) = v^\prime \partial_v[\partial_t v]. \label{def:h}
\end{align}
For notational convenience we will denote
\begin{align}
\bar{h}(t,v) = v^\prime \partial_v[\partial_t v]  = \frac{1}{t}\left(-f_0 - h\right). 
\end{align} 
The decay of $\bar{h}$ quantifies how rapidly $h$ is converging to $-f_0$; one can also see this is as a measure of how rapidly the $x$-averaged vorticity is converging. 
The overline does not refer to complex conjugation but rather to emphasize that $\bar{h}$ is a measure of how close $h$ is to $-f_0$. 
From \eqref{def:h} and \eqref{Euler2} we derive, 
\begin{align}
\partial_t \bar{h} & = - \frac{\bar{h}}{t} - \frac{1}{t}\left(\partial_t f_0 + \partial_t h\right) 
 = -\frac{2}{t}\bar{h} - [\partial_t v] \partial_v\bar{h} + \frac{1}{t}\jap{v^\prime \grad^\perp P_{\neq 0} \phi \cdot \grad f}. \label{def:barh}
\end{align}
The crucial step of the proof of Proposition \ref{prop:CoefControl_Real} is the decay estimate on $\bar{h}$ given in \eqref{ineq:bhc}. The primary challenge to proving \eqref{ineq:bhc} is controlling the last term in \eqref{def:barh}, which is the transfer of information to the zero modes by nonlinear interactions of non-zero modes (see \eqref{ineq:barhenergy} below). 
Since it is most crucial, it is natural to begin there. 

\textit{Proof of \eqref{ineq:bhc}:}\\
From \eqref{def:barh} we have
\begin{align}
\frac{d}{dt} \left( \jap{t}^{2+2s} \norm{\frac{A}{\jap{\partial_v}^s} \bar{h}}_2^2\right) & = -(2-2s)t \jap{t}^{2s}\norm{\frac{A}{\jap{\partial_v}^s} \bar{h}}_2^2 - CK^{v,2}_\lambda - CK^{v,2}_w \nonumber \\ & \quad - 2 \jap{t}^{2+2s} \int \frac{A}{\jap{\partial_v}^s} \bar{h} \frac{A}{\jap{\partial_v}^s}\left([\partial_t v] \partial_v\bar{h} \right) dv \nonumber \\ 
& \quad + 2 t^{-1}\jap{t}^{2+2s} \int \frac{A}{\jap{\partial_v}^s} \bar{h} \frac{A}{\jap{\partial_v}^s} \jap{v^\prime \grad^\perp P_{\neq 0} \phi \cdot \grad f } dv \nonumber \\ 
& = -CK_L^{v,2} - CK^{v,2}_\lambda - CK^{v,2}_w + \mathcal{T}^h + F. \label{ineq:barhenergy}
\end{align} 
The term $\mathcal{T}^h$ is the same nonlinearity that occurs in $[\partial_t v]\partial_v f$ we can treat $\mathcal{T}^h$ in a manner similar to \eqref{eq:Aenergu} but with $u$ replaced just with
$[\partial_t v]$, $f$ replaced by $\bar{h}$, $A$ replaced with $\jap{\partial_v}^{-s}A$ and an additional $t^{2+2s}$ out front (balanced by the decay of $\bar{h}$). 
 We omit the details and conclude by the methods of \S\ref{sec:Transport}, \S\ref{sec:ZeroMdReac} and \S\ref{sec:RemainderEnergy} (except neither the zero mode reaction nor the transport have `D' contributions) and the bootstrap hypotheses that 
 \begin{align} 
\abs{\mathcal{T}^h} & \lesssim \epsilon CK_\lambda^{v,2} + \epsilon CK_{\lambda}^{v,1} + \epsilon \jap{t}^{2s + K_D\epsilon/2}\norm{\frac{A}{\jap{\partial_v}^s} \bar{h} }^2_2. \label{ineq:Th}
\end{align}
We apply the bootstrap control on $CK_\lambda^{v,1}$ and absorb the rest by $CK_\lambda^{v,2}$ and $CK_L^{v,2}$ in \eqref{ineq:barhenergy}.  

The main challenge is in the $F$ term (for `forcing'). This term describes the nonlinear feedback of the non-zero frequencies onto the zero frequencies, which could lead to potential instability by arresting the convergence of the background shear flow (which we are ruling out).
First divide into leading order and higher order contributions: 
\begin{align*}
F & = \sum_{k \neq 0}2t^{-1} \jap{t}^{2+2s} \int  \frac{A}{\jap{\partial_v}^s} \bar{h} \frac{A}{\jap{\partial_v}^s} \left(\grad^\perp \phi_{k} \cdot \grad f_{-k}\right) dv \\ & \quad  + \sum_{k \neq 0}2t^{-1}\jap{t}^{2+2s} \int  \frac{A}{\jap{\partial_v}^s} \bar{h} \frac{A}{\jap{\partial_v}^s} \left(h\grad^\perp \phi_{k} \cdot \grad f_{-k}\right) dv \\ 
& = F^0 + F^\epsilon. 
\end{align*} 
As suggested by \S\ref{sec:RNeps}, $F^\epsilon$ is not significantly harder than $F^0$, in fact the primary complications that arise in the treatment of $R_{N}^{\epsilon,1}$ do not arise in the treatment of $F^\epsilon$.
Hence we focus only on $F^0$; the control of $F^\epsilon$ results in, at worst, similar contributions with an additional power of $\epsilon$.
We begin the treatment of $F^0$ with a paraproduct in $v$ only: 
\begin{align*} 
F^0 & = 2\sum_{M \geq 8}\sum_{k \neq 0}t^{-1}\jap{t}^{2+2s} \int  \frac{A}{\jap{\partial_v}^s} \bar{h} \frac{A}{\jap{\partial_v}^s} \left( \left(\grad^\perp \phi_{k}\right)_{<M/8} \cdot \left(\grad f_{-k}\right)_{M}\right) dv \\ 
& \quad + 2\sum_{M \geq 8}\sum_{k \neq 0}t^{-1}\jap{t}^{2+2s} \int  \frac{A}{\jap{\partial_v}^s} \bar{h} \frac{A}{\jap{\partial_v}^s} \left( \left(\grad^\perp \phi_{k}\right)_{M} \cdot \left(\grad f_{-k}\right)_{<M/8}\right) dv \\ 
& \quad + 2\sum_{M \in \mathbb{D}} \sum_{M^\prime \approx M} \sum_{k \neq 0} t^{-1}\jap{t}^{2+2s} \int  \frac{A}{\jap{\partial_v}^s} \bar{h} \frac{A}{\jap{\partial_v}^s} \left( \left(\grad^\perp \phi_{k}\right)_{M^\prime} \cdot \left(\grad f_{-k}\right)_{M}\right) dv \\ 
& = F^0_{LH} + F^0_{HL} + F^0_{\mathcal{R}}.
\end{align*}
The term $F^0_{HL}$ looks something like the reaction term as studied in \S\ref{sec:RigReac}; dealing with it in a way that allows us to deduce \eqref{ineq:bhc} requires the use of the regularity gap between $\bar{h}$ and $\phi$. 
Consider a single dyadic shell and subdivide based on whether $\phi$ has resonant frequency or not: 
\begin{align*} 
F^0_{HL;M} & = \frac{1}{\pi}\sum_{k \neq 0} t^{-1}\jap{t}^{2+2s} \int_{\eta,\xi} \left(\chi^R + \chi^{NR}\right)  \frac{A_0(\eta)}{\jap{\eta}^s}\overline{\widehat{\bar{h}}}(\eta) \frac{A_0(\eta)}{\jap{\eta}^s} \left( \widehat{\grad^\perp \phi}_{k}(\xi)_{M} \cdot \widehat{\grad f}_{-k}(\eta-\xi)_{<M/8}\right) d\eta d\xi  \\ 
& = F^{0;R}_{HL;M} + F^{0;NR}_{HL;M}, 
\end{align*} 
where $\chi^R = \mathbf{1}_{t \in \mathbf{I}_{k,\xi}}$ and $\chi^{NR} = 1-\chi^{R}$. 

Consider first the NR contribution. 
Subdivide further based on the relationship between time and frequency: 
\begin{align} 
 F^{0;NR}_{HL;M} & =  \frac{1}{\pi}\sum_{k \neq 0}t^{-1}\jap{t}^{2+2s} \int_{10\abs{\eta} \geq t} \chi^{NR} \frac{A_0(\eta)}{\jap{\eta}^s}\overline{\widehat{\bar{h}}}(\eta) \frac{A_0(\eta)}{\jap{\eta}^s} \left( \widehat{\grad^\perp \phi}_{k}(\xi)_{M} \cdot \widehat{\grad f}_{-k}(\eta-\xi)_{<M/8}\right) d\eta d\xi \nonumber \\ 
 & \quad + \frac{1}{\pi}\sum_{k \neq 0}t^{-1}\jap{t}^{2+2s} \int_{10\abs{\eta} < t} \chi^{NR} \frac{A_0(\eta)}{\jap{\eta}^s}\overline{\widehat{\bar{h}}}(\eta) \frac{A_0(\eta)}{\jap{\eta}^s} \left( \widehat{\grad^\perp \phi}_{k}(\xi)_{M} \cdot \widehat{\grad f}_{-k}(\eta-\xi)_{<M/8}\right) d\eta d\xi \nonumber \\ 
& =  F^{0;NR,S}_{HL;M} + F^{0;NR,L}_{HL;M}, \label{ineq:F0NR}
\end{align}
where the `S' is for `short time' and the `L' is for `long time' (relative to frequency). 
Consider first the `S' contribution, for which we take advantage of the $\jap{\partial_v}^{-s}$ to reduce the power of $t$ (also using $\abs{\eta} \approx \abs{\xi} \gtrsim 1$), 
\begin{align*} 
F^{0;NR,S}_{HL;M} & \lesssim \sum_{k \neq 0}t^{1+s} \int_{10\abs{\eta} \geq t} \chi^{NR} \abs{\frac{A_0(\eta)}{\jap{\eta}^s}\widehat{\bar{h}}(\eta) A_0(\eta) \left( \widehat{\grad^\perp \phi}_{k}(\xi)_{M} \cdot \widehat{\grad f}_{-k}(\eta-\xi)_{<M/8}\right)} d\eta d\xi \\ 
& \lesssim \sum_{k \neq 0}t^{1+s} \int_{10\abs{\eta} \geq t} \chi^{NR}\abs{\frac{\abs{\eta}^{s/2} A_0(\eta)}{\jap{\eta}^s}\widehat{\bar{h}}(\eta) A_0(\eta) \abs{k,\xi}^{1-s/2} \widehat{\phi}_{k}(\xi)_{M} \abs{k}^{s/2} \abs{\widehat{\grad f}_{-k}(\eta-\xi)_{<M/8}}} d\eta d\xi. 
\end{align*} 
Hence, by \eqref{ineq:BasicJswap}, \eqref{lem:scon} (using the frequency localizations as usual), \eqref{ineq:IncExp}, \eqref{ineq:SobExp}, \eqref{ineq:L2L2L1} and the bootstrap hypotheses (along with \eqref{ineq:LPOrthoProject}), 
\begin{align}
\sum_{M \geq 8 } F^{0;NR,S}_{HL;M} & \lesssim \sum_{M \geq 8 }\epsilon t^{1+s} \norm{\abs{\partial_v}^{s/2}\frac{A}{\jap{\partial_v}^s}\bar{h}_{\sim M}}_2 \sum_{k \neq 0}\jap{k}^{-2} \norm{ \abs{\grad}^{1-s/2} \chi^{NR} \left(A\phi_{k}\right)_{M} }_2 \nonumber \\
 & \lesssim \epsilon t^{2} \norm{\abs{\partial_v}^{s/2} \frac{A}{\jap{\partial_v}^s}\bar{h}}^2_2 + \epsilon\jap{t}^{2s}\norm{ \abs{\grad}^{1-s/2} \chi^{NR} A\phi}_2^2. \label{ineq:F0S}
\end{align} 
The first term can be absorbed by $CK_{\lambda}^{v,2}$ whereas the latter requires \eqref{ineq:ReactionControl} and then Proposition \ref{lem:PrecisionEstimateNeq0} with the bootstrap controls on the CCK integrals. 

Turn next to the `long' contribution in \eqref{ineq:F0NR}. 
By \eqref{ineq:BasicJswap}, \eqref{lem:scon} (using also that $\abs{\abs{\eta} - \abs{\xi}} \leq \abs{\eta-\xi} < 3\abs{\eta}/16$ on the support of the integrand) and \eqref{ineq:IncExp} we have for some $c \in (0,1)$, 
\begin{align*} 
F^{0;NR,L}_{HL;M} & \lesssim \sum_{k \neq 0}  t^{1+2s} \int_{10\abs{\eta} < t} \chi^{NR} \abs{\frac{A_0(\eta)}{\jap{\eta}^s}\widehat{\bar{h}}(\eta)} \abs{k,\xi}^{1-s} \abs{A\phi_{k}(\xi)_{M}} e^{c\lambda\abs{k,\eta-\xi}^s}\abs{k}^s\abs{\widehat{\grad f}_{k}(\eta-\xi)_{<M/8}}  d\eta d\xi. 
\end{align*} 
Then we use the assumption $s \geq 1/2$ in order to deduce $\abs{k,\xi}^{1-s} \leq \abs{k,\xi}^{s}$ together
with $\abs{\xi - kt} \gtrsim t$, which holds on the support of the integrand since since $10\abs{\eta} < t$ and $1 \lesssim 13 \abs{\eta}/16 \leq \abs{\xi} \leq 19\abs{\eta}/16$. Hence, 
\begin{align*} 
F^{0;NR,L}_{HL;M} & \lesssim \sum_{k \neq 0}  t^{2s-1} \int_{10\abs{\eta} < t} \chi^{NR} \abs{\eta}^{s/2}\abs{\frac{A_0(\eta)}{\jap{\eta}^s}\widehat{\bar{h}}(\eta)} \abs{k,\xi}^{s/2}\left(k^2 + \abs{\xi-kt}^2\right)\abs{A\phi_{k}(\xi)_{M}} \\ & \quad\quad \times e^{c\lambda\abs{k,\eta-\xi}^s}\abs{k}^2 \abs{\widehat{\grad f}_{k}(\eta-\xi)_{<M/8}}  d\eta d\xi. 
\end{align*} 
Therefore (also multiplying by $1 \approx \jap{\frac{\xi}{kt}}^{-1}$), 
\begin{align*} 
F^{0;NR,L}_{HL;M} & \lesssim \sum_{k \neq 0}   t^{3s-1} \int_{10\abs{\eta} < t} \chi^{NR} \abs{\eta}^{s/2}\abs{\frac{A_0(\eta)}{\jap{\eta}^s}\widehat{\bar{h}}(\eta)} \jap{\frac{\xi}{tk}}^{-1} \left(k^2 + \abs{\xi - kt}^2 \right)  \frac{\abs{k,\xi}^{s/2}}{\jap{t}^s}\abs{A\phi_{k}(\xi)_{M}} \\ & \quad\quad \times e^{c\lambda\abs{k,\eta-\xi}^s}\jap{k}^2\abs{\widehat{\grad f}_{-k}(\eta-\xi)_{<M/8}}  d\eta d\xi. 
\end{align*}
Summing in $k$ and $M$ (using \eqref{ineq:LPOrthoProject}) and applying \eqref{ineq:L2L2L1} and \eqref{ineq:SobExp} we have, 
\begin{align}
\sum_{M \geq 8} F^{0;NR,L}_{HL;M} & \lesssim \epsilon \sum_{M \geq 8} t^{3s-1}\norm{\abs{\partial_v}^{s/2} \frac{A}{\jap{\partial_v}^s}\bar{h}_{\sim M}}_2 \sum_{k \neq 0} k^{-2} \norm{\jap{\frac{\partial_v}{t\partial_z}}^{-1}\Delta_L\frac{\abs{\grad}^{s/2}}{\jap{t}^s}\left(A\phi_{k}\right)_{M}}_2 \nonumber \\ 
 & \lesssim \epsilon t^{6s-2}\norm{\abs{\partial_v}^{s/2} \frac{A}{\jap{\partial_v}^s}\bar{h}}^2_2 + \epsilon\norm{\jap{\frac{\partial_v}{t\partial_z}}^{-1}\Delta_L\frac{\abs{\grad}^{s/2}}{\jap{t}^s} A\phi}_2^2.  \label{ineq:F0HLNRL}
\end{align}
One can verify that we may always choose $\tilde q>1/2$ such that $6s -2 < 2 + 2s - 2\tilde q$ provided that we take the restriction that $s < 3/4$; this is an artifact due to the fact that we took $s$ in Proposition \ref{lem:PrecisionEstimateNeq0} rather than $\tilde q$. However, as discussed in \S\ref{sec:MainEnergy}, we can without loss of generality take $s$ close to $1/2$. Therefore, the first term is absorbed by $CK_\lambda^{v,2}$ and the second term is controlled by Proposition \ref{lem:PrecisionEstimateNeq0} and the bootstrap controls on the $CCK$ terms, completing the treatment of $F^{0;NR,L}_{HL}$. 

Next consider $F^{0;R}_{HL;M}$ where, similar to the $R_N^{NR,R}$ term in \S\ref{sec:RigReac}, we have to use the gain from passing $A$ onto $\phi$.
By \eqref{ineq:NRRpartialtw}, \eqref{ineq:WFreqCompNRGain}, \eqref{lem:scon}, \eqref{ineq:SobExp} and \eqref{ineq:IncExp} we deduce for some $c \in (0,1)$,     
\begin{align*} 
\abs{F_{HL;M}^{0;R}} & \lesssim \sum_{k \neq 0} t^{1+2s} \int \chi^R \abs{ \left(\sqrt{\frac{\partial_t w_0(\eta)}{w_0(\eta)}} +  \frac{\abs{\eta}^{s/2}}{\jap{t}^{s}} \right)\frac{A_0(\eta)}{\jap{\eta}^s}\widehat{\bar{h}}(\eta)} \\ 
& \quad\quad\times   \sqrt{\frac{w_k(\xi)}{\partial_tw_k(\xi)}}\frac{\abs{k,\xi}}{\jap{\eta}^s} \frac{w_R(\xi)}{w_{NR}(\xi)} \abs{A\widehat{\phi}_{k}(\xi)_{M}} \abs{\widehat{\grad f}_{-k}(\eta-\xi)_{<M/8}}e^{c\lambda\abs{k,\eta-\xi}^s} d\eta d\xi \\
 & \lesssim \sum_{k \neq 0} t^{1+s} \int \chi^R \abs{ \left(\sqrt{\frac{\partial_t w_0(\eta)}{w_0(\eta)}} +  \frac{\abs{\eta}^{s/2}}{\jap{t}^{s}} \right)\frac{A_0(\eta)}{\jap{\eta}^s}\widehat{\bar{h}}(\eta)} \\ 
& \quad\quad\times   \sqrt{\frac{w_k(\xi)}{\partial_tw_k(\xi)}}\abs{k,\xi}\frac{w_R(\xi)}{w_{NR}(\xi)} \abs{A\widehat{\phi}_{k}(\xi)_{M}} \abs{\widehat{\grad f}_{-k}(\eta-\xi)_{<M/8}}e^{c\lambda\abs{k,\eta-\xi}^s} d\eta d\xi,
\end{align*}
where the last line followed since $\eta \approx kt$. Hence, \eqref{ineq:L2L2L1} and the bootstrap hypotheses imply (also using \eqref{ineq:SobExp} and the low frequency factor to sum in $k$ as well as $k^2 \lesssim \abs{\eta}$ to replace $A$ with $\tilde A$): 
\begin{align} 
\sum_{M \geq 8} \abs{F_{HL;M}^{0;R}} & \lesssim \epsilon t^2\norm{\abs{\partial_v}^{s/2}\frac{A}{\jap{\partial_v}^s}\bar{h}}_2^2 + \epsilon t^{2+2s}\norm{\sqrt{\frac{\partial_t w}{w}}\frac{A}{\jap{\partial_v}^s} \bar{h}}_2^2 \nonumber \\ & \quad  + \epsilon \norm{\sqrt{\frac{w}{\partial_t w}}\abs{\grad}\frac{w_R}{w_{NR}}\chi^R \tilde A \phi}^2_2. \label{ineq:FHLR}
\end{align}
The first two terms are absorbed by the $CK^{v,2}$ terms and hence this suffices for $F_{HL}^{0;R}$ after \eqref{ineq:ReactionControl}, Proposition \ref{lem:PrecisionEstimateNeq0} and the bootstrap controls on the CCK terms.  

Next, we deal with $F^0_{LH}$:
\begin{align*}
F_{LH;M}^0 & \lesssim \sum_{k \neq 0}t^{1+2s} \int \abs{\frac{A_0(\eta)}{\jap{\eta}^s} \hat{\bar{h}}(\eta)} \frac{A_0(\eta)}{\jap{\eta}^s} \abs{\widehat{\grad^\perp \phi}_{-k}(\eta-\xi)_{<M/8}} \abs{\widehat{\grad f}_{k}(\xi)_{M}} d\eta d\xi. 
\end{align*}
Here there is a serious loss of regularity from the $\grad f$ factor, and we will make fundamental use of $s \geq 1/2$. 
Indeed, by the frequency localization $\abs{\eta-\xi} \leq 3\abs{\eta}/16$, \eqref{ineq:BasicJswap} and \eqref{ineq:IncExp} we get the following, by absorbing $s$ derivatives using the regularity gap and $M \geq 8$,
\begin{align*}
\abs{F_{LH;M}^0} & \lesssim \sum_{k \neq 0}t^{1+2s} \int \abs{\frac{A_0(\eta)}{\jap{\eta}^s} \hat{\bar{h}}(\eta)} \abs{\xi}^{1-s}  e^{\lambda\abs{k,\eta-\xi}^s}\jap{k}\abs{\widehat{\grad^\perp \phi}_{-k}(\eta-\xi)_{<M/8}} \abs{A\widehat{f}_{k}(\xi)_{M}} d\eta d\xi. 
\end{align*}
Then crucially we use $1-s \leq s$ and $\abs{\eta} \approx \abs{\xi}$ to deduce
\begin{align*} 
\abs{F_{LH;M}^0} & \lesssim \sum_{k \neq 0}t^{1+2s} \int \abs{\eta}^{s/2}\abs{\frac{A_0(\eta)}{\jap{\eta}^s} \hat{\bar{h}}(\eta)} e^{\lambda\abs{k,\eta-\xi}^s}\jap{k}\abs{\widehat{\grad^\perp \phi}_{-k}(\eta-\xi)_{<M/8}} \abs{\abs{\xi}^{s/2}A\widehat{f}_{k}(\xi)_{M}} d\eta d\xi.
\end{align*}
Therefore, we apply \eqref{ineq:L2L2L1} and Lemma \ref{lem:LossyElliptic} to deduce (also gaining additional powers in $k$ to sum): 
\begin{align*} 
\abs{F_{LH;M}^0} \lesssim \epsilon \sum_{k \neq 0} k^{-2} t^{2s-1} \norm{\abs{\partial_v}^{s/2} \frac{A}{\jap{\partial_v}^s}\bar{h}_{\sim M}}_2\norm{\abs{\partial_v}^{s/2} (Af_{k})_{M}}_2.
\end{align*}
Therefore, summing in $k$ and $M$, 
\begin{align} 
\sum_{M \geq 8} \abs{F_{LH;M}^0} \lesssim \epsilon t^{2+2s - 2\tilde q} \norm{\abs{\partial_v}^{s/2} \frac{A}{\jap{\partial_v}^s}\bar{h}}_2^2   + \epsilon t^{2(s+\tilde q) - 4} \norm{\abs{\partial_v}^{s/2} Af}_2^2. \label{ineq:F0LH}
\end{align}
We have $2(s+\tilde q) - 4 < -2\tilde q$ if we choose $1/2 < \tilde q < 1-s/2$, which can always be done for $s$ close to $1/2$. 
Therefore for $\epsilon$ sufficiently small, the first term \eqref{ineq:F0LH} is absorbed by $CK_\lambda^{v,2}$ and the latter term is controlled by $\epsilon CK_\lambda$. 

The remainder term $F^0_{\mathcal{R}}$ is easy to handle, as in \S\ref{sec:RemainderEnergy}.
We omit the treatment and conclude from the bootstrap hypotheses
\begin{align} 
\abs{F_{\mathcal{R}}^0} \lesssim \jap{t}^{2s-1}\epsilon^2 \norm{\frac{A}{\jap{\partial_v}^s}\bar{h}}_2 \lesssim \epsilon t^{1+2s}\norm{\frac{A}{\jap{\partial_v}^s}\bar{h}}_2^2 + \epsilon^3 \jap{t}^{2s-3}. \label{ineq:F0R}
\end{align} 
The first term is absorbed by $CK_L^{v,2}$ and the latter is time integrable since $s < 1$. 

Combining \eqref{ineq:barhenergy}, \eqref{ineq:Th}, \eqref{ineq:F0S}, \eqref{ineq:F0HLNRL}, \eqref{ineq:FHLR}, \eqref{ineq:F0LH}, \eqref{ineq:F0R} with the bootstrap hypotheses and Proposition \ref{lem:PrecisionEstimateNeq0} completes the 
proof of \eqref{ineq:bhc}.

\textit{Proof of \eqref{ineq:hc}:}\\
From \eqref{def:h} we have
\begin{align}
\frac{1}{2}\frac{d}{dt}\norm{A^R h}_2^2 = -CK^{h}_\lambda - CK^h_w - \int A^Rh A^R\left([\partial_t v]\partial_v h \right) dv + \int A^R h A^R \bar{h} dv, \label{eq:henergy}
\end{align} 
where 
\begin{subequations} 
\begin{align}
CK_w^{h}(\tau) & = \norm{\sqrt{\frac{\partial_t w}{w}}A^R h(\tau)}_2^2 \\ 
CK_\lambda^{h}(\tau) & = (-\dot{\lambda}(\tau))\norm{\abs{\partial_v }^{s/2} A^R h(\tau)}_2^2.
\end{align} 
\end{subequations}
Using the integration by parts trick as in \S\ref{sec:MainEnergy}:
\begin{align*} 
-\int A^Rh A^R\left([\partial_t v]\partial_v h \right) dv & = \frac{1}{2}\int \partial_v [\partial_t v] \abs{A^R h}^2 dv 
         +\int A^Rh \left([\partial_t v] A^R \partial_v h - A^R\left([\partial_t v]\partial_v h \right) \right) dv \\ 
& = \mathcal{E} + \mathcal{M}. 
\end{align*} 
By Sobolev embedding and the bootstrap control on $[\partial_t v]$, 
\begin{align*} 
\mathcal{E} \lesssim \frac{\epsilon}{\jap{t}^{2 - K_D\epsilon/2}} \norm{A^R h}_2^2.
\end{align*}
As in \S\ref{sec:MainEnergy}, the commutator $\mathcal{M}$ is decomposed with a paraproduct, 
\begin{align*} 
\mathcal{M} & = \sum_{M \geq 8} \int A^Rh \left([\partial_t v]_{<M/8} A^R \partial_v h_M - A^R\left([\partial_t v]_{<M/8}\partial_v h_{M} \right) \right) dv \\ 
& \quad + \sum_{M \geq 8} \int A^Rh \left([\partial_t v]_{M} A^R \partial_v h_{<M/8} - A^R\left([\partial_t v]_{M}\partial_v h_{<M/8} \right) \right) dv \\  
& \quad + \sum_{M \in \mathbb{D}} \sum_{M^\prime \approx M} \int A^Rh \left([\partial_t v]_{M} A^R \partial_v h_{M^\prime} - A^R\left([\partial_t v]_{M}\partial_v h_{M^\prime} \right) \right) dv \\ 
& = T^h + R^h + \mathcal{R}^h.  
\end{align*}
The $T^h$ and $\mathcal{R}^h$ terms can be treated as in the methods used in \S\ref{sec:Transport} and \S\ref{sec:RemainderEnergy} with $u$ replaced just with $[\partial_t v]$ (also there are no  R vs NR losses from Lemma \ref{lem:Jswap} since the high frequency factors all have the same `resonant' regularity).
Hence, we omit the treatment and simply conclude
\begin{align} 
T^h + \mathcal{R}^h \lesssim \epsilon CK_\lambda^h + \epsilon^3 \jap{t}^{-2 + K_D\epsilon/2}.  \label{ineq:ThRemh}
\end{align}
In the `reaction' term, $R^h$, we have the same problem that arises in \S\ref{sec:ZeroMdReac}: a loss of regularity as $[\partial_t v]$ has $s$ fewer derivatives and only `non-resonant' regularity, resulting in as much as $1+s$ derivative loss. 
Consider one dyadic shell on the frequency side as in \S\ref{sec:RigReac}, 
\begin{align*} 
R^h_M & = -\frac{1}{2\pi}\int_{\eta,\xi} A^R \overline{\hat{h}}(\eta)_{\sim M} \left[A^R(\eta) - A^R(\eta-\xi)\right]\widehat{[\partial_t v]}(\xi)_M \widehat{\partial_v h}(\eta- \xi)_{<M/8} d\eta d\xi \\ 
& = R_M^{h;1} + R_M^{h;2}.
\end{align*}
The treatment of $R_M^{h;2}$ is essentially the same as the analogous $R_N^3$ in \S\ref{sec:RigReac} and is hence omitted: 
\begin{align} 
\sum_{M \geq 8}\abs{R_M^{h;2}} \lesssim \frac{\epsilon^3}{\jap{t}^{2-K_D\epsilon/2}}. \label{ineq:Rh2}
\end{align} 
The more interesting $R_M^{h,1}$ is treated in a manner similar to \S\ref{sec:ZeroMdReac} above.
We first divide the integral into the contributions where $A^R(t,\eta)$ disagrees noticeably with $A_0(t,\xi)$ and where it does not: 
\begin{align*}
R_M^{h,1} & = -\frac{1}{2\pi}\int_{\eta,\xi} \left[\sum_{k \neq 0} \mathbf{1}_{t \in \I_{k,\eta}} \mathbf{1}_{t \in \I_{k,\xi}} + \chi^{\ast}\right] A^R \overline{\hat{h}}(\eta)_{\sim M} A^R(\eta) \widehat{[\partial_t v]}(\xi)_M \widehat{\partial_v h}(\eta- \xi)_{<M/8} d\eta d\xi \\ 
& = \left(\sum_{k \neq 0}R_{M,k}^{h,1;D}\right) + R_M^{h,1;\ast}, 
\end{align*}
where $\chi^\ast = 1 - \sum_{k \neq 0} \mathbf{1}_{t \in \mathbf{I}_{k,\eta}} \mathbf{1}_{t \in \mathbf{I}_{k,\xi}}$. 
We first treat $R_M^{h,1;\ast}$. Due to the presence of $\chi^\ast$, we do not lose much by replacing $J^R(t,\eta)$ with $J_0(t,\xi)$ (recall the definition of $A^R$ and $J^R$ in \eqref{def:AR}).  Indeed, by the proof of Lemma \ref{lem:Jswap} we have on the support of the integrand: 
\begin{align*}
\chi^\ast \frac{J^R(t,\eta)}{J_0(t,\xi)} \lesssim e^{10\mu\abs{\eta-\xi}^{1/2}} \chi^\ast. 
\end{align*}
Therefore, since $\abs{\abs{\eta} - \abs{\xi}} \leq \abs{\eta-\xi} < 3\abs{\eta}/16$ on the support of the integrand, \eqref{lem:scon}, \eqref{ineq:IncExp} and \eqref{ineq:SobExp} together imply that for some $c \in (0,1)$ we have, 
\begin{align*} 
\abs{R_M^{h,1;\ast}} & \lesssim  \int_{\eta,\xi} \chi^{\ast} \abs{A^R \hat{h}(\eta)_{\sim M}}\abs{ A_0(\xi)\widehat{[\partial_t v]}(\xi)_M e^{c\lambda\abs{\eta-\xi}^s}\widehat{\partial_v h}(\eta- \xi)_{<M/8}} d\eta d\xi.  
\end{align*}
Using that $\abs{\eta} \approx \abs{\xi} \gtrsim 1$ on the support of the integrand and applying \eqref{ineq:L2L2L1} (with the bootstrap hypotheses) we deduce,
\begin{align} 
\abs{R_M^{h,1;\ast}} & \lesssim \int_{\eta,\xi} \chi^{\ast} \abs{A^R \hat{h}(\eta)_{\sim M}} \frac{\abs{\eta}^{s/2}\abs{\xi}^{s/2}}{\jap{\xi}^s} \abs{ A_0(\xi)\widehat{[\partial_t v]}(\xi)_M e^{c\lambda\abs{\eta-\xi}^s}\widehat{\partial_v h}(\eta- \xi)_{<M/8}} d\eta d\xi \nonumber \\ 
 & \lesssim  \epsilon \norm{\abs{\partial_v}^{s/2}A^R h_{\sim M}}_2 \norm{\abs{\partial_v}^{s/2}\frac{A}{\jap{\partial_v}^s} [\partial_t v]_M}_2 \nonumber \\ 
& \lesssim \epsilon \jap{t}^{-2s}\norm{\abs{\partial_v}^{s/2}A^R h_{\sim M}}_2^2 + \epsilon \jap{t}^{2s}\norm{\abs{\partial_v}^{s/2}\frac{A}{\jap{\partial_v}^s} [\partial_t v]_M}^2_2.  \label{ineq:RhM2ast}
\end{align}
For $\epsilon$ sufficiently small, the first term is absorbed by $CK_\lambda^h$ and the second term is controlled by $CK_\lambda^{v,1}$ since $2s < 2 + 2s - 2\tilde q$. 

Next turn to $R_{M,k}^{h,1;D}$. 
The following version of \eqref{ineq:RatJ2partt} applies on the support of the integrand with an analogous proof (recall \eqref{def:AR}): 
\begin{align} 
\frac{J^R(\eta)}{J_0(\xi)} \lesssim \frac{\abs{\eta}}{k^2}\sqrt{\frac{\partial_t w_0(t,\eta)}{w_0(t,\eta)}}\sqrt{\frac{\partial_t w_0(t,\xi)}{w_0(t,\xi)}}e^{20\mu\abs{\eta-\xi}^{1/2}}. \label{ineq:RatJ2parttJR}
\end{align} 
Therefore, by \eqref{lem:scon} and \eqref{ineq:IncExp} there exists $c\in (0,1)$ such that
\begin{align*} 
\abs{R_{M,k}^{h,1;D}} & \lesssim  \int_{\eta,\xi} \mathbf{1}_{t \in \I_{k,\eta}} \mathbf{1}_{t \in \I_{k,\xi}} \abs{\sqrt{ \frac{\partial_t w}{w} }A^R \hat{h}(\eta)_{\sim M}}\frac{\abs{\eta}}{k^2}\abs{\sqrt{\frac{\partial_t w}{w}}A\widehat{[\partial_t v]}(\xi)_M e^{c\lambda\abs{\eta-\xi}^s}\widehat{\partial_v h}(\eta- \xi)_{<M/8}} d\eta d\xi. 
\end{align*}
On the support of the integrand $kt \approx \eta \approx \xi$, so by \eqref{ineq:L2L2L1} and the bootstrap hypotheses,
\begin{align} 
\abs{R_{M,k}^{h,1;D}} & \lesssim t^{1+s} \int_{\eta,\xi} \mathbf{1}_{t \in \I_{k,\eta}} \mathbf{1}_{t \in \I_{k,\xi}} \abs{\sqrt{ \frac{\partial_t w}{w} }A^R \hat{h}(\eta)_{\sim M}} \frac{1}{\jap{\xi}^s} \abs{\sqrt{\frac{\partial_t w}{w}}A\widehat{[\partial_t v]}(\xi)_M e^{c\lambda\abs{\eta-\xi}^s}\widehat{\partial_v h}(\eta- \xi)_{<M/8}} d\eta d\xi \nonumber \\ 
& \lesssim \epsilon t^{1+s} \norm{\sqrt{ \frac{\partial_t w}{w}} A^R \mathbf{1}_{t \in \I_{k,\partial_v}} h_{\sim M}}_2\norm{\sqrt{ \frac{\partial_t w}{w} } \frac{A}{\jap{\partial_v}^{s}} \mathbf{1}_{t \in \I_{k,\partial_v}} [\partial_t v]_{M}}_2, \nonumber 
\end{align} 
where we are denoting the Fourier multiplier $\widehat{\left(\mathbf{1}_{t \in \I_{k,\partial_v}} f\right)}(t,\eta) = \mathbf{1}_{t \in \I_{k,\eta}} \hat{f}(t,\eta)$. 
For $t$ fixed, the supports of these multipliers for different $k$ are disjoint in frequency, and hence we can sum: 
\begin{align}
\sum_{k\neq 0}\abs{R_{M,k}^{h,1;D}} & \lesssim \sum_{k \neq 0} \epsilon \norm{\sqrt{ \frac{\partial_t w}{w}}A^R \mathbf{1}_{t \in \I_{k,\partial_v}} h_{\sim M}}^2_2 + \epsilon t^{2+2s}\norm{\sqrt{ \frac{\partial_t w}{w} } \frac{A}{\jap{\partial_v}^{s}}\mathbf{1}_{t \in \I_{k,\partial_v}} [\partial_t v]_{M}}^2_2 \nonumber \\  
& \approx \epsilon \norm{\sqrt{ \frac{\partial_t w}{w}}A^R h_{\sim M}}^2_2 + \epsilon t^{2+2s}\norm{\sqrt{ \frac{\partial_t w}{w} } \frac{A}{\jap{\partial_v}^{s}} [\partial_t v]_{M}}^2_2, \label{ineq:RhMD} 
\end{align}
which are respectively absorbed by $CK_w^h$ and bounded by $CK_w^{v,1}$. This concludes the treatment of the first (non CK) term in \eqref{eq:henergy}. 

The treatment of the term involving $\bar{h}$ in \eqref{eq:henergy} is similar to $R_M^{h,1}$.
That the dependence is linear is the reason for the presence of the large constant $K_v$ in \eqref{ineq:hc}.  
As in the treatment of $R_M^{h,1}$ we divide based on frequency: 
\begin{align*} 
\int A^R h A^R \bar{h} dv & = \int_{\eta} \left[\sum_{k \neq 0} \mathbf{1}_{t \in \I_{k,\eta}} + \chi^{\ast}\right] A^R \overline{\hat{h}}(\eta) A^R(\eta) \widehat{\bar{h}}(\eta) d\eta  = \left(\sum_{k \neq 0}H^D_k \right) + H^\ast, 
\end{align*}
where $\chi^\ast = 1 - \sum_{k \neq 0} \mathbf{1}_{t \in \mathbf{I}_{k,\eta}}$. 
  
First turn to $H^{D}_k$. Here, \eqref{ineq:RatJ2parttJR} with $\xi = \eta$ holds on the support of the integrand, as does $\abs{\eta} \gtrsim 1$, and hence
\begin{align*} 
\abs{H^D_k} &  \lesssim \int_{\eta} \mathbf{1}_{t \in \I_{k,\eta}} \abs{A^R \hat{h}(\eta)} \frac{\partial_t w(\eta)}{w(\eta)} \frac{\abs{\eta}^{1+s}}{k^2} \abs{\frac{A_0(\eta)}{\jap{\eta}^s}\widehat{\bar{h}}(\eta)} d\eta. 
\end{align*}
As the Fourier restrictions have disjoint support and since $kt \approx \eta$ on the support of the integrand, we get, by Cauchy-Schwarz, the following for some fixed constant $C > 0$, 
\begin{align} 
\sum_{k\neq 0}\abs{H_k^D} 
 & \leq \sum_{k \neq 0}\frac{1}{4}\norm{\sqrt{\frac{\partial_t w}{w}}A^R \mathbf{1}_{t \in \I_{k,\partial_v}} h }_2^2 + Ct^{2+2s}\norm{\sqrt{\frac{\partial_t w}{w}}\frac{A}{\jap{\partial_v}^s}\mathbf{1}_{t \in \I_{k,\partial_v}}\bar{h}}_2^2 \nonumber \\ 
 & \leq \frac{1}{4}\norm{\sqrt{\frac{\partial_t w}{w}}A^R h }_2^2 + Ct^{2+2s}\norm{\sqrt{\frac{\partial_t w}{w}}\frac{A}{\jap{\partial_v}^s} \bar{h}}_2^2. \label{ineq:HD}
\end{align} 
The first term is absorbed by $CK_w^h$ and the latter controlled by the bootstrap hypothesis on $CK^{v,2}_w$, provided we choose $K_v \gg C$.  

Next we turn to $H^\ast$. Due to the presence of $\chi^\ast$, $A^R(t,\eta) \approx A_0(t,\eta)$ on the support of the integrand, by the same proof as \eqref{ineq:BasicJswap}. Hence, we do not need to concern ourselves with the distinction. 
However, we still need to recover the gap of $s$ derivatives.
We treat high and low frequencies separately: for some $C > 0$, we have by Cauchy-Schwarz,
\begin{align} 
\abs{H^{\ast}} & \lesssim \int_{\abs{\eta} \geq 1} \chi^{\ast} \abs{A^R h(\eta)} \abs{\eta}^{s} \abs{\frac{A}{\jap{\eta}^s}\bar{h}(\eta)} d\eta + \int_{\abs{\eta} < 1} \chi^{\ast} \abs{A^R h(\eta)} \abs{\frac{A}{\jap{\eta}^s}\bar{h}(\eta)} d\eta \nonumber \\ 
& \lesssim \norm{\abs{\partial_v}^{s/2} A^R h}_2\norm{\abs{\partial_v}^{s/2} \frac{A}{\jap{\partial_v}^s} \bar{h}}_2 + \norm{A^Rh}_2\norm{\frac{A}{\jap{\partial_v}^s} \bar{h}}_2 \nonumber \\ 
& \leq \frac{\delta_\lambda}{10 t^{2\tilde q}}\norm{\abs{\partial_v}^{s/2} A^R h}_2^2 +  \frac{C}{\delta_\lambda} t^{2\tilde q}\norm{\abs{\partial_v}^{s/2} \frac{A}{\jap{\partial_v}^s} \bar{h}}_2^2 + C\norm{A^R h}_2 \norm{\frac{A}{\jap{\partial_v}^s} \bar{h}}_2.  \label{ineq:Hast}
\end{align}
The first term is absorbed by $CK^{h}_\lambda$ and the latter terms are controlled by the bootstrap hypotheses on $CK_\lambda^{v,2}$ and $\bar{h}$ provided we choose $K_v \gg C\delta_{\lambda}^{-2}$ and $2+2s > 4\tilde q$.   

Finally, summing together \eqref{ineq:ThRemh}, \eqref{ineq:Rh2}, \eqref{ineq:RhM2ast}, \eqref{ineq:RhMD}, \eqref{ineq:HD}, \eqref{ineq:Hast} (using almost orthogonality) with \eqref{eq:henergy} and the bootstrap hypotheses implies that for $\epsilon$ chosen sufficiently small,
\begin{align}
\norm{A^Rh(t)}_2^2 + \frac{1}{2}\int_1^t CK_w^h(\tau) d\tau + \frac{1}{2}\int_1^t  CK_\lambda^{h}(\tau) d\tau  \lesssim K_v\epsilon^2, \label{ineq:hcalmost}
\end{align}
which almost proves \eqref{ineq:hc}. To complete the proof we need to apply the product rules in Lemma \ref{lem:ProdAlg}. 
Indeed, writing $(v^{\prime})^2-1 = (v^\prime-1)^2 + 2(v^\prime - 1)$, and applying Lemma \ref{lem:ProdAlg} and \eqref{ineq:hcalmost} gives,
\begin{align} 
CCK_w^1 + CCK_\lambda^1 \lesssim  CK_w^h + CK_\lambda^{h} + \epsilon^2 \left(CK_w^h + CK_\lambda^{h}\right). \label{ineq:hcCCK}
\end{align} 
Similarly, for $CCK^2_\lambda$ and $CCK_w^2$ terms, write $v^\prime \partial_v v^\prime = \partial_v (v^\prime - 1) + (v^\prime -1)\partial_v (v^\prime - 1)$ and apply Lemma \ref{lem:ProdAlg}: 
\begin{align}
CCK_\lambda^2 + CCK_w^{2} & \lesssim CK_w^h + CK^h_\lambda + \norm{\sqrt{\frac{\partial_t w}{w}}\frac{A^R}{\jap{\partial_v}}\left( (1-v^\prime)\partial_v v^\prime\right)}_2^2 + \norm{\abs{\partial_v}^{s/2}\frac{A^R}{\jap{\partial_v}}\left( (1-v^\prime)\partial_v v^\prime\right)}_2^2 \nonumber \\ 
& \lesssim  CK^h_w + CK^h_\lambda + \epsilon^2 \left(CK_w^h + CK_\lambda^{h}\right).\label{ineq:CCK2ctrl}
\end{align}  
Hence, by possibly adjusting $K_v$ and choosing $\epsilon$ even smaller, \eqref{ineq:hcalmost}, \eqref{ineq:hcCCK} and \eqref{ineq:CCK2ctrl} imply \eqref{ineq:hc}.  

\textit{Proof of \eqref{ineq:ckvctrl}:}\\
Both of the terms in \eqref{ineq:ckvctrl} are controlled in essentially the same way. 
To see the control on the first term, start with dividing into high and low frequency and  use the bootstrap control on $[\partial_t v]$:
\begin{align}
\jap{\tau}^{2+2s}\abs{\dot\lambda(\tau)}\norm{\abs{\partial_v}^{s/2} \frac{A}{\jap{\partial_v}^s}[\partial_t v](\tau)}_2^2 & \leq \jap{\tau}^{2+2s}\abs{\dot\lambda(\tau)}\norm{\abs{\partial_v}^{s/2}\frac{A}{\jap{\partial_v}^s}[\partial_t v](\tau)_{\leq 1}}_2^2 \nonumber \\ & \quad + \jap{\tau}^{2+2s}\abs{\dot\lambda(\tau)} \norm{\abs{\partial_v}^{s/2}\frac{A}{\jap{\partial_v}^s}[\partial_t v](\tau)_{>1}}_2^2 \nonumber \\ 
& \hspace{-3cm}  \lesssim \jap{\tau}^{2s - 2 - 2\tilde q + K_D\epsilon}\epsilon^2 + \jap{\tau}^{2+2s}\abs{\dot\lambda(\tau)}\norm{\abs{\partial_v}^{s/2}\frac{A}{\jap{\partial_v}^s}\partial_v [\partial_t v](\tau)_{>1}}_2^2. \label{ineq:ckvbd}
\end{align}
The first term is integrable for $\epsilon$ small and the latter term is bounded by
\begin{align} 
\norm{\abs{\partial_v}^{s/2}\frac{A}{\jap{\partial_v}^s}\partial_v [\partial_t v](\tau)}_2^2
& \lesssim \norm{\abs{\partial_v}^{s/2}\frac{A}{\jap{\partial_v}^s} (v^\prime \partial_v [\partial_t v](\tau))}_2^2 \nonumber \\ & \quad + \norm{\abs{\partial_v}^{s/2}\frac{A}{\jap{\partial_v}^s} \left(\frac{v^\prime-1}{v^\prime} v^\prime\partial_v [\partial_t v](\tau)\right)}_2^2. \label{ineq:CKv1_ctrl2}
\end{align}
The first term in \eqref{ineq:CKv1_ctrl2} is already controlled by $CK_\lambda^{v,2}$. By \eqref{ineq:wtAProd}, the second term is bounded by the following for some $c \in (0,1)$: 
\begin{align} 
\norm{\abs{\partial_v}^{s/2}\frac{A}{\jap{\partial_v}^s} \left(\frac{v^\prime-1}{v^\prime} v^\prime\partial_v [\partial_t v](\tau)\right)}_2^2 & \lesssim  \norm{\abs{\partial_v}^{s/2}\frac{A}{\jap{\partial_v}^s} \left(\frac{v^\prime-1}{v^\prime} \right)}_2^2 \norm{v^\prime\partial_v [\partial_t v](\tau)}_{\G^{c\lambda,\sigma}}^2 \nonumber \\ & \quad + \norm{\abs{\partial_v}^{s/2}\frac{A}{\jap{\partial_v}^s}v^\prime\partial_v [\partial_t v](\tau)}_2^2 \norm{\frac{v^\prime-1}{v^\prime}}_{\G^{c\lambda,\sigma}}^2. \label{ineq:ckvpOv}
\end{align} 
By the bootstrap hypotheses, we may write $(v^\prime)^{-1}$ as the uniformly convergent geometric series
\begin{align*} 
\frac{1}{v^\prime(t,v)} = 1 + \sum_{n=1}^\infty (v^\prime(t,v)-1)^{n}. 
\end{align*}
Therefore, by \eqref{ineq:GAlg} and the bootstrap control on $v^\prime - 1$, for some $C>0$ we have for $\epsilon$ small: 
\begin{align} 
\norm{\frac{v^\prime-1}{v^\prime}}_{\G^{c\lambda,\sigma}} = \norm{\sum_{n=1}^\infty (v^\prime -1)^n}_{\G^{c\lambda,\sigma}} \leq \sum_{n=1}^{\infty}(C\epsilon)^n \lesssim \epsilon.  \label{ineq:ckvp1_ctrl2}
\end{align} 
Together with the bootstrap control on $CK_\lambda^{v,2}$, this suffices to treat the second term in \eqref{ineq:ckvpOv}.  
To control the first term, we use a repeated application of \eqref{ineq:AProd} and \eqref{ineq:Aalg}, choosing $\epsilon$ sufficiently small to sum the resulting series (denoting constants in these inequalities as $C_p$ and $C_a$ respectively),   
\begin{align*} 
\norm{\abs{\partial_v}^{s/2}\frac{A}{\jap{\partial_v}^s} \left(\frac{v^\prime-1}{v^\prime} \right)}_2^2 & = \norm{\abs{\partial_v}^{s/2}\frac{A}{\jap{\partial_v}^s} \sum_{n=1}^{\infty} (v^\prime-1)^n }_2^2 \\ 
& \leq \left( \sum_{n=1}^\infty \norm{\abs{\partial_v}^{s/2}\frac{A}{\jap{\partial_v}^s} (v^\prime-1)^n }_2 \right)^2 \\ 
& \leq \left( \norm{\abs{\partial_v}^{s/2}\frac{A}{\jap{\partial_v}^s} (v^\prime-1)}_2 \sum_{n=1}^\infty \epsilon^{n-1}4^{n-1}K_v^{n-1}\sum_{j=1}^n C_p^{n-j} C_a^{j-1}\right)^2 \\
& \leq \norm{\abs{\partial_v}^{s/2}\frac{A}{\jap{\partial_v}^s} (v^\prime-1)}_2^2 \left( \sum_{n=1}^\infty \epsilon^{n-1} n 4^{n-1} C_p^{n-1}K_v^{n-1}C_a^{n-1}\right)^2 \\ 
& \lesssim \norm{\abs{\partial_v}^{s/2}\frac{A}{\jap{\partial_v}^s} (v^\prime-1)}_2^2.
\end{align*} 
Therefore, putting this together with \eqref{ineq:ckvbd}, \eqref{ineq:CKv1_ctrl2}, \eqref{ineq:ckvpOv}, \eqref{ineq:ckvp1_ctrl2} and the bootstrap control on $v^\prime \partial_v[\partial_t v]$ implies, 
\begin{align*} 
\jap{\tau}^{2+2s}\abs{\dot{\lambda}(\tau)}\norm{\abs{\partial_v}^{s/2}\frac{A}{\jap{\partial_v}^s}\partial_v [\partial_t v](\tau)}_2^2
& \lesssim CK_\lambda^{v,2} + \epsilon^2\abs{\dot{\lambda}(\tau)}\norm{\abs{\partial_v}^{s/2}\frac{A}{\jap{\partial_v}^s} (v^\prime-1)}_2^2.  
\end{align*}
The first term is integrable by the bootstrap hypotheses and the latter term is integrable by the $CK_\lambda^h$ bound in \eqref{ineq:hcalmost}. From \eqref{ineq:ckvbd}, this completes the treatment of the first term in \eqref{ineq:ckvctrl}. 

To see control on the second term in \eqref{ineq:ckvctrl}, first divide into high and low frequency and use that $\partial_t w$ is only supported in frequencies larger than $1/2$ (see \eqref{def:wk}): 
\begin{align*}
\jap{\tau}^{2 + 2s} \norm{\sqrt{\frac{\partial_t w}{w}} \frac{A}{\jap{\partial_v}^s}[\partial_t v](\tau)}_2^2 
& \lesssim \jap{\tau}^{2 + 2s} \norm{\sqrt{\frac{\partial_t w}{w}}\frac{A}{\jap{\partial_v}^s}\partial_v [\partial_t v](\tau)_{>1/2}}_2^2.
\end{align*} 
From here, it is very similar to the treatment of the first term in \eqref{ineq:ckvctrl}, except now we apply \eqref{ineq:wtAProd} as opposed to \eqref{ineq:AProd}. We omit this for the sake of brevity.

\textit{Proof of \eqref{ineq:partialvintro}:} \\ 
This estimate is relatively straightforward to prove since, due to the lower regularity, we can utilize Lemma \ref{lem:LossyElliptic} as opposed to Proposition \ref{lem:PrecisionEstimateNeq0}. 
It is possible to only use a gap of $s$ derivatives and apply Proposition \ref{lem:PrecisionEstimateNeq0}, however, one will be repeating a more intricate version of the arguments used to deduce \eqref{ineq:bhc}. 

Computing the evolution of $[\partial_t v]$ gives
\begin{align} 
\frac{d}{dt}\left(\jap{t}^{4-K_D\epsilon}\norm{[\partial_t v]}_{\G^{\lambda(t),\sigma-6}}^2\right) & = (4-K_D\epsilon)t \jap{t}^{2-K_D\epsilon}\norm{[\partial_t v]}_{\G^{\lambda(t),\sigma-6}}^2 \nonumber \\ & \quad + \jap{t}^{4-K_D\epsilon}\frac{d}{dt}\norm{\frac{A}{\jap{\partial_v}^s}[\partial_t v]}_{\G^{\lambda(t),\sigma-6}}^2.  \label{ineq:ddtpartt}
\end{align}
Denoting the multiplier $A^S(t,\partial_v) = e^{\lambda(t)\abs{\partial_v}^s}\jap{\partial_v}^{\sigma-6}$ (`S' for `simple'), the latter term gives
\begin{align} 
\jap{t}^{4-K_D\epsilon}\frac{d}{dt}\norm{[\partial_t v]}_{\G^{\lambda(t),\sigma-6}}^2 & = 2\jap{t}^{4-K_D\epsilon} \dot\lambda(t)\norm{\abs{\partial_v}^{s/2}[\partial_t v]}_{\G^{\lambda(t),\sigma-6}}^2 \nonumber  \\ & \quad + 
  2\jap{t}^{4-K_D\epsilon}\int A^S[\partial_t v] A^S \partial_t[\partial_{t} v] dv, \label{ineq:timederivpartialt}
\end{align} 

From \eqref{def:ddtv}, 
\begin{align} 
 2\jap{t}^{4-K_D\epsilon}\int A^S[\partial_t v] A^S \partial_t[\partial_{t} v] dv & = -\frac{4\jap{t}^{4-K_D\epsilon}}{t}\norm{[\partial_t v]}_{\G^{\lambda(t),\sigma-6}}^2 \nonumber \\ 
& \quad - 2\jap{t}^{4-K_D\epsilon} \int A^S [\partial_t v] A^S \left([\partial_t v]\partial_v [\partial_t v]\right) dv 
 \nonumber \\ & \quad - \frac{2\jap{t}^{4-K_D\epsilon}}{t}\int A^S [\partial_t v] A^S\left(v^\prime \langle \grad^\perp P_{\neq 0}\phi\cdot \grad \tilde u\rangle \right) dv \nonumber \\ 
& = V_1 + V_2 + V_3. \label{def:V1V2V3}
\end{align}
By \eqref{ineq:GAlg}, an argument analogous to \eqref{ineq:ckvp1_ctrl2} and the bootstrap hypotheses we have, 
\begin{align}
V_2 & \lesssim \jap{t}^{4-K_D\epsilon}  \norm{[\partial_t v]}_{\G^{\lambda,\sigma-6}} \norm{[\partial_t v]\left(1 + \frac{1-v^\prime}{v^\prime} \right) v^\prime \partial_v [\partial_t v]}_{\G^{\lambda,\sigma-6}} \nonumber \\ 
& \lesssim \jap{t}^{4-K_D\epsilon}  \norm{[\partial_t v]}_{\G^{\lambda,\sigma-6}}^2 \norm{v^\prime \partial_v [\partial_t v]}_{\G^{\lambda,\sigma-6}}\left(1 + \norm{\frac{v^\prime -1}{v^\prime}}_{\G^{\lambda,\sigma-6}}\right) \nonumber \\ 
& \leq \frac{K_D\epsilon}{2} \jap{t}^{3-K_D\epsilon-s}  \norm{[\partial_t v]}_{\G^{\lambda,\sigma-6}}^2, \label{ineq:V2}
\end{align} 
where we define $K_D$ to be the maximum of the constant appearing in this term and one other below.

Treating $V_3$ is not hard due to the regularity gap of $6$ derivatives. 
Note that
\begin{align} 
\grad \tilde u  = -
\begin{pmatrix} 
v^\prime (\partial_v - t\partial_z)\partial_z\phi \\ 
\partial_vv^\prime (\partial_v - t\partial_z)\phi + v^\prime(\partial_v - t\partial_z)\partial_v \phi 
\end{pmatrix}
, \label{def:gradtu}
\end{align}
and therefore by \eqref{ineq:GAlg}, \eqref{def:gradtu}, Lemma \ref{lem:LossyElliptic} and the bootstrap hypotheses, 
\begin{align} 
\norm{\grad P_{\neq 0} \tilde u(t)}_{\G^{\lambda(t),\sigma-5}} \lesssim \frac{\epsilon}{\jap{t}}. \label{ineq:gradulossy}
\end{align}
Projecting to individual frequencies in $z$ and noting that $k \neq 0$ (by the $z$ average and the projection on $\phi$), by \eqref{ineq:GAlg} and Cauchy-Schwarz we have 
\begin{align*} 
V_3 & = \sum_{k \neq 0} \frac{2\jap{t}^{4-K_D\epsilon}}{t}\int A^S [\partial_t v] A^S\left[v^\prime  \grad^\perp \phi_k \cdot \grad \tilde u_{-k} \right] dv \\ 
& \lesssim \sum_{k \neq 0} \frac{\jap{t}^{4-K_D\epsilon}}{t}\norm{A^S[\partial_t v]}_2\norm{A^S\grad^\perp \phi_k}_2 \norm{A^S \grad \tilde u_{-k}}_2\left(1 + \norm{A^S(v^\prime - 1)}_2 \right) \\ 
& \lesssim \frac{\jap{t}^{4-K_D\epsilon}}{t}\norm{[\partial_t v]}_{\G^{\lambda,\sigma-6}}\norm{\grad^\perp P_{\neq 0}\phi}_{\G^{\lambda,\sigma-6}} \norm{\grad P_{\neq 0} \tilde u}_{\G^{\lambda,\sigma-6}}\left(1 + \norm{v^\prime - 1}_{\G^{\lambda,\sigma-6}} \right).
\end{align*}
By \eqref{ineq:gradulossy}, the bootstrap hypotheses on $v^\prime -1$, and Lemma \ref{lem:LossyElliptic}  we then have for some $C > 0$, 
\begin{align} 
V^3 & \lesssim \jap{t}^{-K_D\epsilon} \epsilon^2 \norm{[\partial_t v]}_{\G^{\lambda,\sigma-6}} \leq \frac{K_D\epsilon \jap{t}^{4-K_D\epsilon}}{2t} \norm{[\partial_t v]}_{\G^{\lambda,\sigma-6}}^2 + C\epsilon^3 t^{-3 -K_D\epsilon}. \label{ineq:V3}
\end{align}  
Putting together \eqref{ineq:V2},\eqref{ineq:V3} with \eqref{def:V1V2V3} and \eqref{ineq:timederivpartialt} and integrate, we derive \eqref{ineq:partialvintro}. This completes the proof of Proposition \ref{prop:CoefControl_Real}.
Being the last Proposition remaining, the proof of Theorem \ref{thm:Main} is complete.  

\section{Conclusion} \label{sec:Conclusion}
Our proof is very different from that of Mouhot and Villani \cite{MouhotVillani11}, however 
there are still some analogies and common themes that are worth pointing out. 
Let us mention  the most important mathematical parallels:
\begin{itemize}
\item In  \cite{MouhotVillani11}, the role of the plasma echoes has similarities with 
that of the Orr critical times, being that both are responsible for the main nonlinear growth.  
 The moment estimates in  \cite{MouhotVillani11} 
controls this growth whereas here the ``toy model'' plays this role. 
 Note that the  exponential growth in time of \cite{MouhotVillani11} is replaced here by a controlled regularity loss.  

\item The use of para-differential calculus (combined with the well-established existence theory) allows to avoid 
the use of the Newton iteration in \cite{MouhotVillani11}. 
For example, the paraproduct decomposition permits us to separate the natural transport effects and the reaction effects. 

\item Our treatment of transport  in the energy estimates 
and the   well-chosen change of variables  allow us 
 to avoid the use of `deflection maps' as in \cite{MouhotVillani11}.
\end{itemize}
Although the Euler and Vlasov-Poisson systems have several fundamental differences,
 after completing this work we succeeded, with C. Mouhot,
to extend the Landau damping result in \cite{MouhotVillani11}
to all Gevrey class smaller than three (e.g. $s > 1/3$) using some of  the  ideas of this work \cite{BMM13}.
 
An obvious question about Theorem \ref{thm:Main} is whether or not $s > 1/2$ is optimal. 
The toy model \eqref{toy}, which estimates the worst possible growth and is behind  
the regularity requirement, took absolute values and hence does not take into account 
 the potential for oscillations and   
cancellations.  
 In fact, numerical simulations of the 3-wave model of \cite{VMW98} show dramatic oscillations. 
Even when there are infinitely many interacting modes, oscillations may often
 cancel and yield much weaker growth than the one given by 
 \eqref{toy}.  However,   we think  it would be extremely difficult to rule out the possibility that there exist rare configurations which are ``resonant''  in some sense 
  and lead to a growth which matches our toy model (we expect that such configurations are highly non-generic). 

There are many other related problems where the linear operator predicts 
a decay by mixing  similar to \eqref{def:2DEuler}. 
The most obvious extension is to include a more general class of shear flows where now we linearize around the flow 
$(V(y),0)$. 
We believe that this requires the incorporation of several non-trivial enhancements, since it fundamentally changes the structure of the critical times which would manifest in our approach most clearly in the elliptic estimates. 
A related extension would be to study the problem in the presence of no-penetration boundaries in $y$.  
A third problem would be to remove the periodicity in $x$, altering the physical mechanism from mixing to \emph{filamentation} (in fact the behavior may be different, as is possible in Landau damping, see \cite{glassey94,glassey95}).
Here there are fundamental difficulties: our proof relies heavily on the `well-separation' of critical times, which no longer holds in the unbounded case. 
Other examples in fluid mechanics pointed out in \S\ref{sec:Intro} include the $\beta$-plane model, stratified flows and the vortex axisymmetrization problem.
Proving decay by mixing on the nonlinear level for any of these models seems to be a very interesting question. Of all of these, we expect the  $\beta$-plane model to be the easiest.

\subsection*{Acknowledgments}
The authors would like to warmly thank Clement Mouhot for sharing his many insights about this problem and for enlightening discussions on this topic and the topic of Landau damping.    
We would also like to thank the following people for references and suggestions: Antoine Cerfon, George Hagstrom, Susan Friedlander, Andrew Majda, Vladimir Sverak, Vlad Vicol, Cedric Villani and Chongchun Zeng as well as David G\'erard-Varet for pointing out an error in the first draft.  
J. Bedrossian would like to thank David Uminsky for pointing him to Landau damping and its connection with 2D Euler. 

\appendix
\section{Appendix}
\subsection{Littlewood-Paley decomposition and paraproducts} \label{Apx:LPProduct}
In this section we fix conventions and notation regarding Fourier analysis, Littlewood-Paley and paraproduct decompositions. 
See e.g. \cite{Bony81,BCD11}  for more details. 

For $f(z,v)$ in the Schwartz space, we define the Fourier transform $\hat{f}_k(\eta)$ where $(k,\eta) \in \Integer \times \Real$, 
\begin{align*} 
\hat{f}_k(\eta) = \frac{1}{2\pi}\int_{\Torus \times \Real} e^{-i z k - iv\eta} f(z,v) dz dv. 
\end{align*}
Similarly we have the Fourier inversion formula, 
\begin{align*} 
f(z,v) = \frac{1}{2\pi}\sum_{k \in \Integer} \int_{\Real} e^{i z k + iv\eta} \hat{f}_k(\eta) d\eta. 
\end{align*} 
As usual the Fourier transform and its inverse are extended to $L^2$ via duality. 
We also need to apply the Fourier transform to functions of $v$ alone, for which we use analogous conventions.
With these conventions note, 
\begin{align*} 
\int f(z,v) \overline{g}(z,v) dzdv & = \sum_{k}\int \hat{f}_k(\eta) \overline{\hat{g}}_{k}(\eta) d\eta \\ 
\widehat{fg} & = \frac{1}{2\pi}\hat{f} \ast \hat{g} \\ 
(\widehat{\grad f})_k(\eta) & = (ik,i\eta)\hat{f}_k(\eta). 
\end{align*}

This work makes heavy use of the Littlewood-Paley dyadic decomposition.  
Here we fix conventions and review the basic properties of this classical theory, see e.g. \cite{BCD11} for more details. 
First we define the Littlewood-Paley decomposition only in the $v$ variable. 
Let $\psi \in C_0^\infty(\Real;\Real)$ be such that $\psi(\xi) = 1$ for $\abs{\xi} \leq 1/2$ and $\psi(\xi) = 0$ for $\abs{\xi} \geq 3/4$ and define $\rho(\xi) = \psi(\xi/2) - \psi(\xi)$, supported in the range $\xi \in (1/2,3/2)$. 
Then we have the partition of unity 
\begin{align*} 
1 = \psi(\xi) + \sum_{M \in 2^\Naturals} \rho(M^{-1}\xi), 
\end{align*}  
where we mean that the sum runs over the dyadic numbers $M = 1,2,4,8,...,2^{j},...$
 and we define the cut-off $\rho_M(\xi) = \rho(M^{-1}\xi)$, each supported in $M/2 \leq \abs{\xi} \leq 3M/2$. 
For $f \in L^2(\Real)$ we define
\begin{align*} 
f_{M} & = \rho_M(\abs{\partial_v})f, \\ 
f_{\frac{1}{2}} & =  \psi(\abs{\partial_v})f, \\ 
f_{< M} & = f_{\frac{1}{2}} + \sum_{K \in 2^{\Naturals}: K < M} f_K, 
\end{align*}
which defines the decomposition  
\begin{align*} 
f = f_{\frac{1}{2}} + \sum_{M \in 2^\Naturals} f_M.  
\end{align*}
Normally one would use $f_0$ rather than the slightly inconsistent $f_{\frac{1}{2}}$, however $f_0$ is reserved for the much more commonly used projection onto the zero mode only in $z$ (or $x$).  
Recall the definition of $\mathbb D$ from \S\ref{sec:Notation}. There holds the almost orthogonality and the approximate projection property 
\begin{subequations} \label{ineq:LPOrthoProject}
\begin{align} 
\norm{f}^2_2 & \approx \sum_{M \in \mathbb D} \norm{f_M}_2^2 \\
 \norm{f_M}_2 & \approx  \norm{(f_{M})_M}_2. 
\end{align}
\end{subequations}
The following is also clear:
\begin{align*} 
\norm{\abs{\partial_v}f_M}_2 \approx M \norm{f_M}_2. 
\end{align*}
We make use of the notation 
\begin{align*} 
f_{\sim M} = \sum_{K \in \mathbb{D}: \frac{1}{C}M \leq K \leq CM} f_{K}, 
\end{align*}
for some constant $C$ which is independent of $M$.
Generally the exact value of $C$ which is being used is not important; what is important is that it is finite and independent of $M$. 
Similar to \eqref{ineq:LPOrthoProject} but more generally, if $f = \sum_{k} D_k$ for any $D_k$ with $\frac{1}{C}2^{k} \subset \textup{supp}\, D_k \subset C2^{k}$ it follows that 
\begin{align} 
\norm{f}^2_2 \approx_C \sum_{k} \norm{D_k}_2^2. \label{ineq:GeneralOrtho}
\end{align}
During much of the proof we are also working with Littlewood-Paley decompositions defined in the $(z,v)$ variables, 
with the notation conventions being analogous. 
Our convention is to use $N$ to denote Littlewood-Paley projections in $(z,v)$ and $M$ to denote projections only in the $v$ direction. 

We have opted to use the compact notation above, rather than the commonly used alternatives
\begin{align*} 
\Delta_{j}f  = f_{2^j}, \quad\quad S_jf  = f_{<2^j},
\end{align*}
in order to reduce the number of characters in long formulas. 
The last unusual notation we use is 
\begin{align*} 
P_{\neq 0}f = f - <f>, 
\end{align*}
which denotes projection onto the non-zero modes in $z$.

Another key Fourier analysis tool employed in this work is the paraproduct decomposition, introduced by Bony \cite{Bony81} (see also \cite{BCD11}). 
Given suitable functions $f,g$ we may define the paraproduct decomposition (in either $(z,v)$ or just $v$), 
\begin{align*} 
fg & = T_fg + T_gf + \mathcal{R}(f,g)\\ 
& = \sum_{N \geq 8} f_{<N/8}g_N + \sum_{N \geq 8} g_{<N/8}f_N + \sum_{N \in \mathbb{D}}\sum_{N/8 \leq N^\prime \leq 8N} g_{N^\prime}f_{N}, 
\end{align*}
where all the sums are understood to run over $\mathbb D$. 
In our work we do not employ the notation in the first line since at most steps in the proof we are forced to explicitly write the sums and treat them term-by-term anyway. 
This is due to the fact that we are working in non-standard regularity spaces 
and, more crucially, are usually applying multipliers which do not satisfy any version of $AT_f g \approx T_f Ag$. 
Hence, we have to prove almost everything `from scratch' and can only rely on standard para-differential calculus as a guide.

\subsection{Elementary inequalities and Gevrey spaces} \label{apx:Gev}

In the sequel we show some basic inequalities which are extremely useful for working in this scale of spaces.  
The first three are versions of Young's inequality (applied on the frequency-side here).
\begin{lemma}
Let $f(\xi),g(\xi) \in L_\xi^2(\Real^d)$, $\jap{\xi}^\sigma h(\xi) \in L_\xi^2(\Real^d)$ and $\jap{\xi}^\sigma b(\xi) \in L_\xi^2(\Real^d)$  for $\sigma > d/2$, 
Then we have 
\begin{align} 
\norm{f \ast h}_2 & \lesssim_{\sigma, d} \norm{f}_2\norm{\jap{\cdot}^\sigma h}_2, \label{ineq:L2L1}  \\
\int \abs{f(\xi) (g \ast h)(\xi)} d\xi & \lesssim_{\sigma,d} \norm{f}_2\norm{g}_2\norm{\jap{\cdot}^\sigma h}_2 \label{ineq:L2L2L1} \\ 
\int \abs{f(\xi) (g \ast h \ast b) (\xi)} d\xi & \lesssim_{\sigma,d} \norm{f}_2\norm{g}_2\norm{\jap{\cdot}^\sigma h}_2\norm{\jap{\cdot}^\sigma b}_2. \label{ineq:L2L2L1L1}  
\end{align}
\end{lemma}
\begin{proof} 
Inequality \eqref{ineq:L2L1} follows from the $L^2 \times L^1 \rightarrow L^2$ Young's inequality
and Cauchy-Schwarz: 
\begin{align*}
\int \abs{h(\xi)} d\xi \leq \left( \int \frac{1}{\jap{\xi}^{2\sigma}} d\xi \right)^{1/2} \norm{\jap{\cdot}^\sigma h}_{2} \lesssim \norm{\jap{\cdot}^\sigma h}_{2}. 
\end{align*}
Inequality \eqref{ineq:L2L2L1} follows from Cauchy-Schwarz and \eqref{ineq:L2L1}. 
For \eqref{ineq:L2L2L1L1}, apply Young's inequality twice: 
\begin{align*} 
\int \abs{f(\xi) (g \ast h \ast b) (\xi)} d\xi \lesssim \norm{f}_2 \norm{g \ast h \ast b}_2 \lesssim \norm{f}_2 \norm{g}_2 \norm{h \ast b}_1 \lesssim \norm{f}_2 \norm{g}_2 \norm{h}_1 \norm{b}_1, 
\end{align*} 
and proceed as above. 
\end{proof} 

The next set of inequalities show that one can often gain on the index of regularity when comparing frequencies which are not too far apart (provided $0 < s < 1$).

\begin{lemma}
Let $0 < s < 1$ and $x \geq y \geq 0$ (without loss of generality).  
\begin{itemize} 
\item[(i)] If $x + y > 0$,
\begin{align} 
\abs{x^s - y^s} \lesssim_s \frac{1}{x^{1-s} + y^{1-s}}\abs{x-y}. \label{ineq:TrivDiff}
\end{align}
\item[(ii)] If $\abs{x-y} \leq x/K$ for some $K > 1$ then 
\begin{align} 
\abs{x^s - y^s} \leq \frac{s}{(K-1)^{1-s}}\abs{x-y}^s. \label{lem:scon}
\end{align} 
Note $\frac{s}{(K-1)^{1-s}} < 1$ as soon as $s^{\frac{1}{1-s}} + 1 < K$. 
\item[(iii)] In general, 
\begin{align*} 
\abs{x + y}^s \leq \left(\frac{x}{x+y}\right)^{1-s}\left(x^s + y^s\right). 
\end{align*}  
In particular, if $y \leq x \leq Ky$ for some $K < \infty$ then 
\begin{align} 
\abs{x + y}^s \leq \left(\frac{K}{1 + K}\right)^{1-s}\left(x^s + y^s\right). \label{lem:strivial}
\end{align} 
\end{itemize}
\end{lemma}

\begin{proof}  

Inequality \eqref{ineq:TrivDiff} follows easily from considering separately $x \geq 2y$ and $x < 2y$. 

To prove \eqref{lem:scon} we use that in this case $y^{-1} \leq K/(K-1) x^{-1}$ and hence, 
\begin{align*} 
x^s \leq y^s + \frac{s}{y^{1-s}}(x-y) \leq y^s + \frac{s}{(K-1)^{1-s}}\abs{x-y}^{s} 
\end{align*}

To see \eqref{lem:strivial},  
\begin{align*} 
\abs{x+y}^s & = \left(\frac{x}{x+y}\right)\abs{x + y}^s + \left(\frac{y}{x+y}\right)\abs{x + y}^s
 \leq \left(\frac{x}{x+y}\right)^{1-s}\left(x^s + y^s\right). 
\end{align*}
\end{proof}

Using \eqref{ineq:L2L2L1}, \eqref{lem:scon} and \eqref{lem:strivial} together with a paraproduct expansion, the following product lemma is relatively straightforward. 
For contrast, the lemma holds when $s = 1$ only for $c = 1$. 
\begin{lemma}[Product lemma] \label{lem:GevProdAlg}
For all $0<s<1$, $\sigma \geq 0$ and $\sigma_0 > 1$, there exists $c = c(s,\sigma,\sigma_0) \in (0,1)$ such that the following holds for all $f,g \in \mathcal{G}^{\lambda,\sigma;s}$:
\begin{subequations}
\begin{align} 
\norm{fg}_{\G^{\lambda,\sigma;s}} & \lesssim \norm{f}_{\G^{c\lambda,\sigma_0;s}} \norm{g}_{\G^{\lambda,\sigma;s}} + \norm{g}_{\G^{c\lambda,\sigma_0;s}} \norm{f}_{\G^{\lambda,\sigma;s}}, \label{ineq:GProduct}
\end{align}
\end{subequations}
in particular, $\mathcal{G}^{\lambda,\sigma;s}$ has the algebra property:
\begin{align} 
\norm{f g}_{\G^{\lambda,\sigma}} & \lesssim \norm{f}_{\G^{\lambda,\sigma}} \norm{g}_{\G^{\lambda,\sigma}} \label{ineq:GAlg} 
\end{align} 
\end{lemma} 

Gevrey and Sobolev regularities can be related with the following two inequalities. 
\begin{itemize}
\item[(i)] For all $x \geq 0$, $\alpha > \beta \geq 0$, $C,\delta > 0$, 
\begin{align} 
e^{Cx^{\beta}} \leq e^{C\left(\frac{C}{\delta}\right)^{\frac{\beta}{\alpha - \beta}}} e^{\delta x^{\alpha}};  \label{ineq:IncExp}
\end{align}
\item[(ii)] For all $x \geq 0$, $\alpha,\sigma,\delta > 0$, 
\begin{align} 
e^{-\delta x^{\alpha}} \lesssim \frac{1}{\delta^{\frac{\sigma}{\alpha}} \jap{x}^{\sigma}}. \label{ineq:SobExp}
\end{align}
\end{itemize}
Together these inequalities show that for $\alpha > \beta \geq 0$, $\norm{f}_{\mathcal{G}^{C,\sigma;\beta}} \lesssim_{\alpha,\beta,C,\delta,\sigma} \norm{f}_{\mathcal{G}^{\delta,0;\alpha}}$.

\subsection{Coordinate transformations in Gevrey spaces} \label{apx:coTrans}
The proof of Theorem \ref{thm:Main} requires moving from $(x,y)$ to $(z,v)$ coordinates at the beginning (in Lemma \ref{lem:shorttime}) and then back again in \S\ref{sec:ProofConcl}. 
It is crucial to notice the regularity losses incurred in this section, as discussed in more depth in \cite{MouhotVillani11} where related inequalities play an important role.  

It is well-known (see e.g. \cite{LevermoreOliver97}) that the $\mathcal{G}^{\lambda;s}$ norms have an equivalent `physical-side' representation which will be convenient here: 
\begin{align} 
\norm{f}_{\G^{\lambda}} & \approx \left[\sum_{n = 0}^\infty \left(\frac{\lambda^n}{(n!)^{\frac{1}{s}}} \norm{D^n f}_2\right)^2\right]^{1/2}.  \label{def:Gphys} 
\end{align} 
In this section it will also be useful to have a slightly more general scale of norms: 
\begin{align}
\norm{f}_{l^qL^p;\lambda} = \left[ \sum_{n = 0}^\infty \left(\frac{\lambda^n}{(n!)^{\frac{1}{s}}} \norm{D^n f}_p\right)^q \right]^{1/q}. \label{def:l1Lq}
\end{align} 
By H\"older's inequality and Sobolev embedding (also \eqref{ineq:SobExp}): for $\lambda > \lambda^\prime$ and $p,q \in [1,\infty]$,
\begin{subequations} \label{ineq:RelatScale}
\begin{align} 
\norm{f}_{l^pL^q;\lambda^\prime} & \leq \norm{f}_{l^1L^q;\lambda^\prime} \lesssim_{\lambda-\lambda^\prime} \norm{f}_{l^pL^q;\lambda}, \label{ineq:lpLqHo} \\  
\norm{f}_{l^2L^\infty;\lambda^\prime} & \lesssim_{\lambda-\lambda^\prime}  \norm{f}_{l^2L^2;\lambda}. \label{ineq:lpLqSE}
\end{align} 
\end{subequations} 
One of the norms in the scale \eqref{def:l1Lq} satisfies an algebra property:
\begin{align} 
\norm{fg}_{l^1L^\infty;\lambda} \leq \norm{f}_{l^1L^\infty;\lambda}\norm{g}_{l^1L^\infty;\lambda}, \label{ineq:l1LinftyProduct}
\end{align} 
which follows from Leibniz's rule and Young's inequality in a manner similar to several proofs in \cite{MouhotVillani11}; we omit the details.  
A more sophisticated inequality is the following, which estimates the effect of composition on Gevrey regularity. 
\begin{lemma}[Composition inequality]\label{lem:CompoIneq}
For all $s \in (0,1]$, $p \in [1,\infty]$ and $\lambda > 0$,  
\begin{align}
\norm{F \circ (Id +  G)}_{l^1L^p;\lambda} \leq \norm{\textup{det} \, (Id + \grad G)^{-1}}^{1/p}_{\infty}\norm{F}_{l^1L^p;\lambda + \nu}, \label{ineq:compoineq}
\end{align}   
where
\begin{align*} 
\nu = \norm{G}_{l^1L^\infty;\lambda}. 
\end{align*}
\end{lemma}
\begin{proof} 
For simplicity we restrict the following proof to one dimension, however it holds also in higher dimensions with a similar proof. 
Denote $H = Id + G$. 
Proceeding as in \cite{MouhotVillani11}, by the Fa\`a di Bruno formula we have 
\begin{align*} 
\norm{F \circ H}_{l^1L^p;\lambda} \leq \sum_{k = 0}^\infty \norm{ (D^kF)\circ H}_p \sum_{\sum_{j=1}^n jm_j = n, \sum_{j=1}^n m_j = k} \frac{\lambda^n}{(n!)^{\frac{1}{s}-1} m_1 ! ... m_n!} \prod_{j = 1}^n (j!)^{\left(\frac{1}{s}-1\right)m_j} \norm{\frac{D^{j}H}{(j!)^{\frac{1}{s}}}}_\infty^{m_j},
\end{align*}
where the second summation runs over all possible combinations which satisfy the conditions indicated.
Under these conditions, we claim by induction on $k$ that the following always holds:
\begin{align} 
k! \prod_{j=1}^n (j!)^{m_j} \leq n!. \label{ineq:induct}
\end{align} 
Since $k \leq n$ it is trivial for $k = n = 1$. 
To see the inductive step, suppose it is true for a given combination of $m_j$ with $\sum_{j = 1}^n m_j = k$ and 
$\sum_{j=1}^n jm_j = n$ for some choices of $k$ and $n$. Now consider replacing $m_{j_0} \mapsto m_{j_0}+1$ for some $1 \leq j_0 \leq n+1$ (where we consider $m_{n+1} = 0$). This increments $k$ by one and $n$ by $j_0$ and hence by the inductive hypotheses we need only check
\begin{align*} 
(k+1) j_0! \leq (n+1)...(n+j_0), 
\end{align*}
which is clear since $n \geq k$, from which \eqref{ineq:induct} follows. 
Hence by \eqref{ineq:induct}, 
\begin{align*} 
\norm{F \circ H}_{l^1L^p;\lambda} & \leq \sum_{k = 0}^\infty \norm{ (D^kF)\circ H}_p \frac{1}{(k!)^{\frac{1}{s}-1}} \sum_{\sum_{j=1}^n jm_j = n, \sum_{j=1}^n m_j = k} \frac{\lambda^n}{m_1 ! ... m_n!} \prod_{j = 1}^n\norm{\frac{D^{j}H}{(j!)^{\frac{1}{s}}}}_\infty^{m_j} \\ 
&  = \sum_{k = 0}^\infty \norm{ (D^kF) \circ H}_p \frac{1}{(k!)^{\frac{1}{s}}} \left[ \sum_{j = 1}^{\infty}\lambda^j \norm{\frac{D^{j}H}{(j!)^{\frac{1}{s}}}}_\infty\right]^k, \\
&  = \sum_{k = 0}^\infty \norm{ (D^kF) \circ H}_p \frac{1}{(k!)^{\frac{1}{s}}} \left[ \lambda + \sum_{j = 1}^{\infty}\lambda^j \norm{\frac{D^{j}G}{(j!)^{\frac{1}{s}}}}_\infty\right]^k,  
\end{align*}
where the second to last line followed from the multinomial formula. 
The proof is completed by changing variables in the Lebesgue norm. 
\end{proof} 

The next tool is the following inverse function theorem in Gevrey regularity. 
\begin{lemma}[Inverse function theorem] \label{lem:IFT}
Let $\alpha(x):\T\times \Real \rightarrow \T\times \Real$ be a given smooth function and consider the equation 
\begin{align} 
\beta(x) = \alpha(x + \beta(x)). \label{eq:IFT}
\end{align}
For all $\lambda > \lambda^\prime > 0$, there exists an $\epsilon_0 = \epsilon_0(\lambda,\lambda^\prime) > 0$ such that 
if $\norm{\alpha}_{l^1L^\infty;\lambda} < \epsilon_0$ then \eqref{eq:IFT} has a smooth solution $\beta$ which satisfies 
\begin{align*} 
\norm{\beta}_{l^1L^\infty;\lambda^\prime} \lesssim \norm{\alpha}_{l^1L^\infty;\lambda}. 
\end{align*}  
By \eqref{eq:IFT} and Lemma \ref{lem:CompoIneq} we have the following for $p \in [1,\infty]$ (the RHS is not necessarily finite) 
\begin{align} 
\norm{\beta}_{l^1L^p;\lambda^\prime} \lesssim \norm{\alpha}_{l^1L^p; \lambda^\prime + \norm{\beta}_{l^1L^\infty;\lambda^\prime}}. 
\end{align}  
\end{lemma} 
\begin{proof} 
The proof follows from a Picard iteration. It is important that we chose a norm which has an algebra property \eqref{ineq:l1LinftyProduct} and for which composition does not lose too much regularity; otherwise one would have to use a Newton iteration.
Let $\beta_0(x) = \alpha(x)$ and inductively define
\begin{align*} 
\beta_{k+1}(x) = \alpha(x + \beta_{k}(x)). 
\end{align*} 
Hence, 
\begin{align*} 
\beta_{k+1}(x) - \beta_{k}(x) = \int_0^1D\alpha(x + s\beta_{k-1}(x) + (1-s)\beta_{k}(x))(\beta_{k}(x) - \beta_{k-1}(x)) ds. 
\end{align*} 
Therefore from \eqref{ineq:l1LinftyProduct} and Lemma \ref{lem:CompoIneq}. 
\begin{align*} 
\norm{\beta_{k+1} - \beta_k}_{l^1L^\infty;\lambda^\prime} \leq \left(\sup_{s \in (0,1)}\norm{D\alpha}_{l^1L^\infty;\lambda^\prime + s\norm{\beta_{k-1}}_{l^1L^\infty;\lambda^\prime} + (1-s)\norm{\beta_{k}}_{l^1L^\infty;\lambda^\prime} }\right)\norm{\beta_{k} - \beta_{k-1}}_{l^1L^\infty;\lambda^\prime}. 
\end{align*} 
By induction, for $\epsilon_0$ chosen sufficiently small the iteration converges and the lemma follows. 
\end{proof}

Using Lemma \ref{lem:IFT} we may prove Lemma \ref{lem:shorttime} stated in \S\ref{sec:MainEnergy}.

\begin{proof}[Lemma \ref{lem:shorttime}]
From the results of \cite{BB77,FoiasTemam89,LevermoreOliver97}, it follows that for any prescribed $\bar\lambda \in \left(3\lambda/4 + \lambda^\prime/4,\lambda_0\right)$, and $\bar\epsilon > 0$, we may choose $\epsilon^\prime$ sufficiently small such that $\sup_{t \in (0,1)} \norm{\omega(t)}_{\bar\lambda} < \bar\epsilon$ (note $\epsilon^\prime < \bar\epsilon < \epsilon$). Using the characteristics, we may also assert the spatial localization $\int \abs{y\omega(t,x,y)} dx dy \leq \bar \epsilon $ (possibly after reducing $\bar\epsilon$, and hence $\epsilon^\prime$). 
By Sobolev embedding, we also have control on every $L^p$ norm and in particular it is clear that the control on $\norm{1-\partial_y v}_\infty \leq 7/12$ can be imposed by reducing $\bar\epsilon$ if necessary. 
From \eqref{def:zv}, we have Gevrey control on $v = v(t,y)$ and $z = z(t,x,y)$ 
however since $f(t,z,v) = \omega(t,x(t,z,v),y(t,v))$ in order to control the Gevrey norm of $f$ we need to solve 
for $x,y$ in terms of $z,v$. 
By  \ref{def:zv} and the estimates on $\omega$, writing $\left(x(t,z,v),y(t,z,y)\right) = (z+tv,v) + \beta(z,v)$, we may apply \eqref{ineq:RelatScale} and a slight variant of Lemma \ref{lem:IFT} (adjusting $\bar \epsilon$ if necessary)
to solve for $\beta$ with Gevrey-$\frac{1}{s}$ regularity. 
Then by \eqref{ineq:RelatScale} and Lemma \ref{lem:CompoIneq} we can deduce (for $\bar \epsilon$ sufficiently small): 
\begin{align*} 
\sup_{t \in (0,1)}\norm{f(t)}_{3\lambda/4 + \lambda^\prime/4} & < \epsilon \\ 
E(1) & < \epsilon \\ 
\norm{1-v^\prime}_\infty & < 6/10. 
\end{align*} 
We also used the general inequality (verified by direct computation): for any $\nu > 0$ and $t \in [0,1]$, 
\begin{align*} 
\norm{\omega \circ (z+tv,v)}_{\G^{\nu}} \lesssim_{\nu} \norm{\omega}_{\G^{\nu}}. 
\end{align*} 
Finally, \eqref{ineq:CKctrlShort} now follows for $\bar \epsilon$ sufficiently small (also after an application of \eqref{ineq:SobExp} in the case of $CK_\lambda$ and possibly adjusting $\bar\lambda$). 
\end{proof} 

\subsection{Rapid convergence of background flow} \label{apx:LogCorrection}
The proof of \eqref{ineq:xdamping_slow} found in \S\ref{sec:ProofConcl} follows from writing the $x$ average of \eqref{def:2DEulerMomentum} in the $(z,y)$ variables to \eqref{def:Umomen} and integrating using that the priori estimates \eqref{ineq:apriori} provide decay estimates in the $(z,y)$ variables. 
In fact, the derivation of \eqref{def:Umomen} is capturing a subtle cancellation between the vorticity and $U^y$ that originates from the structure of the linear problem. 
 Consider \eqref{def:2DEulerMomentum} and take $x$ averages of the first equation. Then one derives the following: 
\begin{align*}
\partial_t <U^x> + <U^y \partial_ y U^x> = 0. 
\end{align*}
We will consider the nonlinear term with $U^y$ and $U^x$ replaced by solutions to the linearized Euler equations. 
Since $y$ derivatives are growing linearly, from a rough order of magnitude approximation one would expect that the nonlinear term decays like $O(t^{-2})$. In fact, we have $<U^y \partial_y U^x> = -<U^y \omega>$, so anything faster than a $O(t^{-2})$ decay indicates that something interesting is happening. 
First note from the Biot-Savart law:  
\begin{align*} 
<U^y \partial_ y U^x> = -\frac{1}{2\pi} \int \psi_x\psi_{yy} dx.
\end{align*} 
Writing $\psi_{yy}$ on the Fourier side and using the $(x,y)$ analogue of \eqref{orr-cri}: 
\begin{align} 
\widehat{\psi_{yy}}(t,k,\xi) & = \frac{\xi^2 \widehat{\omega_{in}}(k,\xi+kt)}{k^2 + \xi^2}  = \left(\abs{kt}^2 - 2 kt(\xi+kt) + \abs{\xi + kt}^2 \right)  \frac{\widehat{\omega_{in}}(k,\xi+kt)}{k^2 + \xi^2}. \label{eq:magical}
\end{align} 
Therefore, while $P_{\neq}\phi_{yy}$ is not decaying or strongly converging to anything, we have the remarkable property that $\psi_{yy} = t^2\psi_{xx} + O(t)$. 
Indeed, from \eqref{eq:magical} we have the following, since the leading order cancels due to the $x$ average:
\begin{align*} 
\abs{\left(\widehat{\frac{1}{2\pi} \int \psi_x\psi_{yy} dx}\right)(t,\eta)} & = \abs{\frac{i}{2\pi}\sum_{k \neq 0}\int_{\xi} k \left(\abs{\xi + kt}^2 -2 kt(\xi+kt)\right)\frac{\omega_{in}(-k,\eta-\xi - kt)\omega_{in}(k,\xi+kt)}{\left(k^2 +\xi^2 \right)\left(k^2 + \abs{\eta-\xi}^2\right)} d\xi } \\ 
& \lesssim t\sum_{k \neq 0}\int_\xi \abs{\frac{\omega_{in}(-k,\eta-\xi - kt)\jap{k,\xi+kt}^3\omega_{in}(k,\xi+kt)}{\left(k^2 +\xi^2 \right)\left(k^2 + \abs{\eta-\xi}^2\right)}} d\xi \\ 
& \lesssim t\sum_{k \neq 0}\int_\xi \abs{\frac{\omega_{in}(-k,\eta-\xi - kt)\jap{k,\eta-\xi-kt}^2\jap{k,\xi+kt}^5\omega_{in}(k,\xi+kt)}{\left(k^2 +\xi^2 \right)\left(k^2 + \abs{\eta-\xi}^2\right)\jap{k,\eta-\xi-kt}^2\jap{k,\xi+kt}^2}} d\xi \\ 
& \lesssim \frac{t}{\jap{t}^4}\sum_{k \neq 0}\int_\xi \jap{-k,\eta-\xi - kt}^2\abs{\omega_{in}(-k,\eta-\xi - kt)\jap{k,\xi+kt}^5\omega_{in}(k,\xi+kt)} d\xi. 
\end{align*}
Then from \eqref{ineq:L2L1} we have (without making an attempt to be optimal), 
\begin{align*} 
\norm{<U^y \partial_ y U^x>}_2 = \norm{\frac{1}{2\pi} \int \psi_x\psi_{yy} dx}_2 \lesssim \frac{1}{\jap{t}^3}\norm{\omega_{in}}_{H^5}^2. 
\end{align*} 
Since the nonlinear behavior matches the linear behavior to leading order, this indicates that indeed, \eqref{ineq:xdamping_slow} should be expected on the nonlinear level.

{\small\bibliographystyle{abbrv} \bibliography{eulereqns}}
\end{document}